\documentclass[12pt,a4paper,oneside,leqno]{amsart}

\usepackage[all]{xy}
\usepackage{hyperref,amssymb,amsmath,bbm,url,geometry,cancel,aliascnt,cleveref,enumitem}
\usepackage[mathscr]{euscript}
\usepackage{mathrsfs}
\usepackage[utf8]{inputenc}

\newtheorem*{thmi}{Theorem}
\newtheorem{thm}{Theorem}[subsection]
\newaliascnt{prop}{thm}
\newtheorem{prop}[prop]{Proposition}
\aliascntresetthe{prop}
\newaliascnt{cor}{thm}
\newtheorem{cor}[cor]{Corollary}
\aliascntresetthe{cor}
\newaliascnt{lm}{thm}
\newtheorem{lm}[lm]{Lemma}
\aliascntresetthe{lm}

\theoremstyle{definition}
\newaliascnt{df}{thm}
\newtheorem{df}[df]{Definition}
\aliascntresetthe{df}
\newaliascnt{ex}{thm}
\newtheorem{ex}[ex]{Example}
\aliascntresetthe{ex}
\newaliascnt{exo}{thm}

\aliascntresetthe{exo}
\newaliascnt{num}{thm}
\newtheorem{num}[num]{}
\aliascntresetthe{num}
\crefformat{num}{#2#1#3}
\Crefformat{num}{#2#1#3}

\theoremstyle{remark}
\newaliascnt{rem}{thm}
\newtheorem{rem}[rem]{Remark}
\aliascntresetthe{rem}

\numberwithin{equation}{thm}

\newcommand{\smod}[1]{#1\mkern-2mu\operatorname{-mod}}

\DeclareMathOperator{\cPic}{\mathscr Pic^\ZZ} \DeclareMathOperator{\Perf}{\mathscr Perf}

\DeclareMathOperator{\Dqc}{D_{qc}}

\newcommand{\K}[1]{\operatorname K_{#1}}
\newcommand{\KM}[1]{\operatorname K^M_{#1}}
\newcommand{\KW}[1]{\operatorname K^W_{#1}}
\newcommand{\KMW}[1]{\operatorname K^{MW}_{#1}}
\newcommand{\I}[1]{\operatorname I^{#1}}
\newcommand{\gI}[1]{\bar{\operatorname I}^{#1}}
\DeclareMathOperator{\GW}{GW}
\DeclareMathOperator{\W}{W}

\DeclareMathOperator{\rk}{rk}

\newcommand{\CHW}[1]{\widetilde{\mathrm{CH}}{}^{#1}}
\newcommand{\Cq}[1]{\widetilde{\mathrm{C}}{}^{#1}}

\newcommand{\F} {\mathbf F}

\newcommand{\NN} {\mathbb N}
\newcommand{\ZZ} {\mathbb Z}

\newcommand{\FF} {\mathbb F}
\newcommand{\RR} {\mathbb R}

\renewcommand{\AA} {\mathbb A}
\newcommand{\GG} {\mathbb G_m }
\newcommand{\GGx}[1] {\mathbb G_{m,k} }
\newcommand{\PP} {\mathbb P}

\newcommand{\cL}{\mathcal L} \newcommand{\cM}{\mathcal M}

\newcommand{\cO}{\mathcal O}

\newcommand{\cI}{\mathcal I}

\newcommand{\C}{\mathcal C} 

\newcommand{\un}{\mathbbm 1}

\DeclareMathOperator{\tr}{tr} \DeclareMathOperator{\pair}{\pi} \DeclareMathOperator{\copair}{\delta}

\DeclareMathOperator{\Ker}{Ker}

\DeclareMathOperator{\Img}{Im}

\DeclareMathOperator{\Frac}{Frac}

\newcommand{\tw}[1]{\langle#1\rangle} \newcommand{\dtw}[1]{\langle\langle#1\rangle\rangle} \DeclareMathOperator{\ev}{ev} 

\DeclareMathOperator{\Tr}{Tr}

\DeclareMathOperator{\Res}{Res}

\DeclareMathOperator{\forget}{F}
\DeclareMathOperator{\hyper}{H}
\DeclareMathOperator{\thyper}{\tilde H}

\DeclareMathOperator{\Der}{\mathrm D}

\newcommand{\derL}{\mathbf{L}}
\newcommand{\derR}{\mathbf{R}}

\newcommand{\pur}{\mathfrak p}

\DeclareMathOperator{\Id}{Id}
\newcommand{\spec}[1] {\operatorname{\mathrm{Spec}}(#1)}

\DeclareMathOperator{\Spec}{Spec}
\DeclareMathOperator{\Proj}{Proj}
\DeclareMathOperator{\Hom}{Hom}

\DeclareMathOperator{\uHom}{\underline{Hom}}
\DeclareMathOperator{\End}{End}
\DeclareMathOperator{\Ext}{Ext}

\DeclareMathOperator{\Pic}{Pic}
\DeclareMathOperator{\CH}{CH}

\DeclareMathOperator{\ord}{ord}
\DeclareMathOperator{\Div}{div}
\DeclareMathOperator{\tdiv}{\widetilde{div}}
\DeclareMathOperator{\tord}{\widetilde{ord}}

\DeclareMathOperator{\tdeg}{\widetilde{deg}}

\DeclareMathOperator{\car}{char} 
\begin{document}

\title{Notes on Milnor-Witt K-theory}

\author{F.~D\'eglise}
\date{August 2026}

\begin{abstract}
These notes develop the foundations of Milnor–Witt K-theory over arbitrary fields,
 without assuming perfectness or separability, and with a systematic treatment of twists.
 Extending the works of Morel and Feld,
 we construct the four fundamental functorialities ---
 restriction, corestriction (transfer), product, and residue ---
 and prove all relations between them, making the twist isomorphisms explicit.
 Transfers are defined and computed in full generality through Grothendieck's differential trace
 and the Scheja-Storch (\emph{Bézoutian}) method,
 and shown to agree with those obtained by the Bass-Tate approach.
 We also establish refined residue and transfer formulas involving quadratic multiplicities,
 including valuation-theoretic defect terms,
 together with a direct proof of the residue-corestriction relation.
 The resulting framework provides the algebraic and functorial foundations needed
 for the construction of Chow-Witt groups and quadratic cycle theories over arithmetic bases.
\end{abstract}

\maketitle

\setcounter{tocdepth}{3}
\tableofcontents

\section{Introduction}

\subsection{Milnor's K-theory and conjectures}\label{sec:MilnorKth}

One is retrospectively amazed by John Milnor's cornerstone work \cite{Milnor} on K-theory and quadratic forms.
 At first sight, Milnor failed to define higher K-groups of a field $K$,
 as one now knows that the K-theory he introduced coincides with higher K-theory only up to degree $2$,
 by Matsumoto's theorem.\footnote{Among
 his inspirations, Milnor cites Moore and Matsumoto's works.
 See \ref{num:MilnorKth} and \Cref{rem:Matsumoto} for a reminder.}
 On the other hand, the graded ring defined by Milnor was to be recognized more than twenty years
 later\footnote{By a famous theorem of Totaro \cite{Tot}.} as an invariant as fundamental as algebraic K-theory:
 it is the $(n,n)$-part of the motivic cohomology ring of $K$.
 Thus, Milnor's definition was truly the first appearance of motivic cohomology, in its symbolic guise.

Moreover, Milnor had the brilliant insight of relating his new K-theory ring to two apparently unconnected classical invariants:
 Galois cohomology of $K$ with coefficients in the $2$-torsion ring $\ZZ/2$, and the graded algebra associated with the fundamental ideal of the Witt ring of $K$.
 One motivation came from the theory of characteristic classes \cite{MilSta}\footnote{Recall
 that this book is based on Milnor's 1957 lectures!}, and in particular from the so-called
 Stiefel-Whitney classes $w_n(\xi) \in H^n(X,\ZZ/2)$ of a real vector bundle $\xi$
 over a smooth manifold $X$.
 Among the inspirations of Milnor was a very short note \cite{Delz} of Delzant defining an algebraic
 analogue of these classes, for a field of characteristic not $2$:
$$
w^D_n:\GW(K) \rightarrow H^n(G_K,\ZZ/2)
$$
where $\GW(K)$ is the Grothendieck group of quadratic $K$-vector spaces\footnote{Though Delzant
 was obviously inspired by the Witt ring introduced in \cite{Witt},
 this is the first occurrence of $\GW(K)$ in the literature. See \ref{df:GW} more generally.}
 and the right-hand side is the $\ZZ/2$-cohomology of the absolute Galois group of $K$,
 or equivalently the \'etale $\ZZ/2$-cohomology of $\Spec(K)$.

Milnor remarks that Delzant's Stiefel-Whitney classes factorize through his K-groups
 modulo $2$, thus obtaining a factorization of $w_n^D$ as:
$$
\GW(K) \xrightarrow{w_n} \KM n(K)/2 \xrightarrow{h_K} H^n(G_K,\ZZ/2)
$$
where $h_K$ is sometimes called the \emph{norm residue homomorphism} or the \emph{Galois symbol}.
 The first of the Milnor conjectures states that $h_K$
 is an isomorphism for all $n>0$.\footnote{This was later generalized by Bloch and Kato
 by replacing $2$ with an arbitrary prime. A complete proof of the Bloch-Kato conjecture is available
 in \cite{HW}, and a detailed account of the history of the Milnor conjecture can be found in Section 1.7 of
 \emph{loc. cit.}}

Even more remarkably, Milnor recognizes a way to go backward along the map $w_n$,
 and formulates another conjecture\footnote{see question 4.3 of \cite{Milnor} or \Cref{thm:MilnorConj}
 of the present paper}
 relating the Witt ring with his K-theory ring.
 This question was solved more than 30 years later by Orlov, Vishik and Voevodsky in \cite{OVV},
 after Voevodsky's proof of the first Milnor conjecture and his revolutionary idea of introducing
 \emph{motivic homotopy theory}.

\subsection{Barge and Morel obstruction theory}\label{sec:BM_CHtilde}

Extending Milnor's ideas on characteristic classes,
 and building on the ideas of motivic homotopy, Jean Barge and Fabien Morel
 introduced in \cite{BM} an algebro-geometrical analogue of the Euler class
 of real oriented vector bundles (see \cite[\textsection 9, Def. p. 98]{MilSta}).
 The technical innovation of their definition is the construction of an appropriate
 algebraic analogue of the integral singular cohomology of a real manifold,
 a cohomology ring that they call the \emph{Chow groups of oriented cycles},
 which one now calls after Jean Fasel's foundational works \cite{FaselRing, FaselCW}
 the \emph{Chow-Witt groups}.

While Milnor K-theory modulo $2$ is a suitable receptacle for Stiefel-Whitney
 classes (over fields), as seen above, Barge and Morel had the idea of gluing above
 this $2$-torsion ring
 the integral information coming from the fundamental ideal $\I{}(K)$
 and the Milnor K-ring $\KM*(K)$: see \cite[Section~1]{BM}
 or \Cref{cor:fundamental_square_KMW} here.
 The resulting graded ring, denoted by $J^*(K)$ in \emph{loc. cit.}
 is now called the \emph{Milnor-Witt K-theory} of $K$ and denoted by $\KMW*(K)$.
 The study of this functor on fields is the main subject of the present
 expository notes.

Before diving into our motivations,
 let us come back briefly to the work of Barge and Morel:
 based on results of Rost \cite{Rost} and Schmid \cite{Schmid},
 they define the Chow-Witt groups of a smooth $k$-scheme
 as the cohomology of a Gersten-like complex with coefficients in $\KMW *$,
 which is now called the \emph{Rost-Schmid complex}.
 Then they define the Euler class of an \emph{oriented} algebraic vector bundle
 and make various conjectures about it (Conjecture p.~289,
 Rem.~2.4 in \emph{loc. cit.}).\footnote{The conjecture was partially solved
 by Morel in \cite[Th.~3.12]{Mor} and was pushed much further by works
 of Asok and Fasel; see \cite{AF} for a survey, and in particular \textsection 4.2
 for a review on (Barge-Morel) Euler classes.}

\subsection{Basic aim and scope of these notes}

The aim of these notes is to lay the foundations for the theory of quadratic cycles
 and Chow-Witt groups, 
 based on Feld's axiomatic approach \cite{FeldMWmod}.
 We will therefore present the theory of the Milnor-Witt ring of fields,
 with an emphasis on its functoriality.
 As motivation, the hasty reader can take a look at \Cref{sec:MW-mod},
 for the functoriality properties
 we aim to establish. We improve the theory known so far by removing any
 assumption on the fields considered,\footnote{Usually, one considers fields
 of characteristic not $2$, and one assumes that they are finitely
 generated over some perfect base field.} providing full detailed proofs,
 and extending the generality of some of Feld's formulas.

Our presentation of Milnor-Witt K-theory mixes two historical approaches.
 The first one, due to Barge and Morel, as already mentioned above,
 gives the construction as a gluing of Witt's theory of quadratic forms\footnote{Or
 rather Grothendieck-Witt theory
 of inner product spaces in order to allow fields of characteristic two;}
 and Milnor's K-ring. See \Cref{cor:fundamental_square_KMW} for the statement.
 The second one, due to Hopkins and Morel, is much closer to Milnor's viewpoint,
 and gives a beautiful presentation of the Milnor-Witt K-theory ring
 in terms of explicit generators and relations (\Cref{df:KMW}).
 In fact, the richness of the theory comes from the comparison of these
 two approaches. This is based on Milnor's conjecture: here we refer to
 \cite{GZS} in characteristic not $2$, and to \cite{Robin} in characteristic $2$.
 Except for references to these two papers, and to basics on Witt rings (\cite{MH})
 and Milnor K-theory (\cite{BT}, \cite{KatoMilnor}),
 these notes are self-contained.

As explained in the previous section,
 the justification for defining Milnor-Witt K-theory in addition to Milnor K-theory
 is to be able to develop an algebraic orientation theory.
 The concrete manifestation of this orientation theory is the existence of twists on
 the former K-theory, which we regard as part of the structure almost from the outset.
 The first phenomenon that demonstrates the need to consider 
 twists is the construction of residues associated to a discretely valued field $(K,v)$
 with residue field $\kappa_v$:
$$
\partial_v:\KMW*(K) \rightarrow \KMW*(\kappa_v,\omega_v).
$$
 Here the \emph{twist} $\omega_v$ on the right-hand side is the \emph{normal space}
 associated with the valuation $v$; see \Cref{df:specialization}
 and therein for details. In our opinion, the consideration of twists
 also sheds light on Witt's theory as explained in \Cref{ex:Witt_res}.

\subsection{Construction of transfers}

In fact, a large portion of the paper is devoted to the study of transfer maps on
 Milnor-Witt K-theory. This is no surprise, as it was a famous problem
 left open by Bass and Tate for Milnor K-theory (\cite[I.\textsection 5]{BT}), only
 resolved by Kato in \cite[\textsection 1.7, Prop. 5]{KatoNorm}.
 Given a finite field extension $E/k$, or $\varphi:k \rightarrow E$,
 the transfer map has the form (see \Cref{df:MW-transfers}):
\begin{equation}\label{eqi:traces_KMW}
\Tr_{E/k}^{MW}=\varphi^*:\KMW n (E,\omega_{E/k}) \rightarrow \KMW n(k),
\end{equation}
where $\omega_{E/k}$ is the determinant of the cotangent complex of $E/k$.
 (We call this the canonical module; see \Cref{df:can_sheaf}.)
 Here again, the twist is essential, though it is trivial
 for (and only for) separable extensions.

For finitely generated field extensions over a perfect base field of characteristic not $2$,
 these transfer maps were introduced by Morel in \cite[Chap. 4]{Mor},
 in the more general context of strongly $\AA^1$-invariant sheaves
 but mostly neglecting twists. The theory was recast for Milnor-Witt K-theory,
 still with the same restriction on fields, by Feld in \cite{FeldTohoku}.

There are two methods to define transfers on functors defined on fields.
 The first one is to follow the approach of Bass and Tate via residue maps
 and what is called after Rost the \emph{Weil reciprocity formula}
 (see \Cref{rem:Weil_RC}(2) for the case of the projective line).
 The second one is by gluing known transfer maps. This is closer to the approach 
 of Fasel for defining pushforwards on Chow-Witt groups
 (see \cite[Cor. 10.4.5]{FaselCW}).

\bigskip

\noindent\textit{Bass-Tate method.} In these notes, we exploit both approaches.
 For the Bass-Tate method,
 we have chosen to introduce Chow-Witt groups of Dedekind schemes.
 This serves both as an illustration of the theory
 and as a convenient framework to express the Weil reciprocity formula.
 In fact, we reformulate the latter as the following computation of Chow-Witt
 groups of quadratic $0$-cycles of the projective line:
\begin{thmi}(see \Cref{thm:PB1})
Let $k$ be a field, and $\cL$ be an invertible sheaf on the projective line $\PP^1_k$.
 We let $\cL_\infty$ be the restriction of $\cL$ over the point at $\infty$
 and $\omega_\infty$ be the normal sheaf of the immersion $i_\infty:\{\infty\} \rightarrow \PP^1_k$.

 Then the Chow-Witt group
 of quadratic divisors of $\PP^1_k$ with coefficients in $\cL$
 is given by:
$$
\CHW 1(\PP^1_k,\cL) \simeq \begin{cases}
\GW(k,\omega_\infty \otimes \cL_\infty) & \deg(\cL) \text{ even,} \\
\ZZ & \deg(\cL)  \text{ odd.}
\end{cases}
$$
Explicitly, the isomorphism is given by pushforward along $i_\infty$.
\end{thmi}
This theorem perfectly illustrates the role of twists in Milnor-Witt K-theory and Chow-Witt groups,
 which is the major difference with Milnor K-theory and usual Chow groups.
 It was first proved by Fasel in \cite{FaselPB}. Our proof is more direct,
 and allows to drop any restriction on the base field.
 Applying the above result when $\cL$ is the canonical sheaf $\omega=\cO(-2)$ on $\PP^1_k$
 gives the degree map on Chow-Witt groups:
$$
\deg:\CHW 1(\PP^1_k,\omega) \rightarrow \GW(k).
$$
This degree map actually encompasses all the transfer maps for Milnor-Witt K-theory
 for any monogenic field extension of $k$.\footnote{We
 have formulated here the theorem for ``$\GG$-degree $0$'' for the sake of clarity.
 Indeed, Chow-Witt groups are the $0$-th $\GG$-graded part of a bigraded group
 (see \Cref{df:CHW_curves} for more details).} This is the geometric interpretation of the 
 method of Bass and Tate.\footnote{see \cite[(5.4)]{BT}, and especially diagram (3)}
 The problem with this approach is to show
 that these transfers are independent of the chosen generators, and that one can
 extend the definition to any finite field extension.

\bigskip

\noindent \textit{Gluing and differential traces.} Instead of proving this directly,
 as is done by Morel and Feld,
 we use the second mentioned approach.
 On the one hand, one has well-defined transfer maps on Milnor K-theory after Kato.
 For the Grothendieck-Witt part (based on inner product forms to deal with characteristic $2$),
 we show that one can define twisted transfer maps by directly using the trace morphism
 that follows from Grothendieck duality formalism, and which we call the differential trace map
 (see \Cref{df:diff_trace}):
$$
\Tr_{E/k}^\omega:\omega_{E/k} \rightarrow k.
$$
Using this map and a classical method of Scharlau, one deduces transfers for twisted
 Grothendieck-Witt groups. The advantage of these transfers
 is that they do not depend on any choice,
 and yet can be compared precisely to Scharlau's ones
 (see \Cref{rem:comp_Scharlau_trace}).\footnote{Note that this kind of construction
 of twisted transfers for Witt groups (of $\ZZ[\frac 1 2]$-schemes)
 has been previously considered by several authors including \cite{Gille, Nena, CaHo}.
 Our treatment is simple and direct, and is well-suited for explicit computations,
 as explained below.}
 Moreover, they can be glued appropriately to Kato's transfers
 on Milnor K-theory and induce the desired transfers \eqref{eqi:traces_KMW}
 on twisted Milnor-Witt K-theory.

The important result is to compare these glued transfers with the one obtained
 by the Bass-Tate method. According to the uniqueness property of the latter,
 this involves checking a \emph{twisted form of the quadratic reciprocity law}
 for the transfers based on the differential trace maps. This is an enhancement
 of the classical result of Scharlau, now valid in arbitrary characteristics
 (see \Cref{thm:KMW_GW-differential_Reciprocity} for the exact statement).
 As a consequence, our definition agrees with that of Morel, and we also
 reprove the independence theorem of Morel and Feld
 (see \Cref{prop:comparison_MW-transfers}).

The nice feature of the second way of defining transfers is that
 it is well-suited for computing trace maps.
 In particular, the differential trace map can be computed explicitly using
 a method of Scheja and Storch.\footnote{The historical reference for
 this method is \cite{SS} but we will use \cite{Kunz} as a reference; see \Cref{sec:G-SS-residues}
 for more details.}
 It allows us to compute trace maps in Milnor-Witt K-theory
 in a way analogous to the method of Kass and Wickelgren for computing (local) $\AA^1$-Brouwer
 degree (see \cite{KW_EKL}, \cite{BMP}).
 Indeed, trace maps can be explicitly described in terms of \emph{Bézoutians}
 (see Definitions~\ref{df:bezoutian},~\ref{df:SS_trace} for recall).
 This is especially important in the inseparable case.
 Let us illustrate this computation with the following statement:
\begin{thmi}[see \Cref{ex:compute_MW-trace}(3)]
Let $k$ be a field of positive characteristic $p>0$,
 and $a \in k^\times$ be an element which is not a $p$-th power.
 Consider the inseparable extension $E=k[\sqrt[q] a]$ of $k$, $q=p^n$.

Let $\tau^\alpha_{E/k}$ be the \emph{Tate trace map}
 associated with $E/k$ and the choice of $\alpha=\sqrt[q] a$
 (see \cite[\textsection 1, (2)]{Tate} and \Cref{rem:Scheja-Storch_trace_monogeneous}).
 Let $w=dt \otimes (\overline{t^q-a})^*$ be the nonzero element of the canonical
 module $\omega_{E/k}$.
 Then for any unit $u \in E^\times$, the following formula holds in $\GW(k)$:
$$
\Tr_{E/k}^{MW}(\tw u \otimes w)=[\tau^\alpha_{E/k}(u.-)]
$$
where the right-hand side denotes the class of the inner product:
 $(x,y) \mapsto \tau^\alpha_{E/k}(uxy)$.
\end{thmi}
This formula was actually one of our motivations for writing these notes.
 It shows that the Tate trace map is obtained from
 the canonical differential trace map 
 by trivializing the invertible module $\omega_{E/k}$
 through the element $w$ introduced in the theorem,
 associated with the chosen monogenic presentation.

It once again illustrates the importance of the twists in computations
 related to Milnor-Witt K-theory. In particular, in the above formula,
 changing $w$ usually completely modifies the result of the computation.
 We give further examples of this phenomenon in \Cref{ex:compute_MW-trace},
 as well as an analogue of the degree formula for Milnor K-theory:
 \Cref{cor:degree_formula}.

\subsection{Towards Chow-Witt groups}
 As already mentioned, the development of the functorial properties
 of Milnor-Witt K-theory, as axiomatized in the notion of Milnor-Witt
 cycle modules in \cite{FeldMWmod}, gives a solid foundation for Chow-Witt
 groups \cite{FaselCW}. As an illustration, 
 we will use in \cite{DFJ} the formulas established in \Cref{sec:MW-mod} of the present work 
 to extend the original definition of Fasel \cite{FaselCW}
 to singular schemes without requiring the existence of a base field,
 thereby opening the theory to arithmetic applications.

We have taken particular care with formula (R3b), proved in \Cref{thm:R3b}.
 This formula allows one to define pushforwards on Chow-Witt groups,
 and in particular \emph{degrees} of quadratic cycles for proper varieties.
 Therefore, it is central to applications in quadratic enumerative geometry.
 To give some background on this formula, let us recall that it was first stated
 by Rost in \cite{Rost}, without giving a proof in the Milnor K-theory case
 (see \cite[7.4.3]{GSza} for one). An argument was given in \cite[Cor. 10.4.5]{FaselCW}
 to deduce from Rost's formula the case needed for Chow-Witt groups. We give
 here a direct proof of this delicate formula, based on the theory of valued fields.

Moreover, several fundamental formulas of Milnor-Witt K-theory
 can be extended to situations that involve appropriately defined \emph{quadratic multiplicities}.
 This is in particular the case for formula (R3b) that admits a refinement (R3b+),
 proved in \Cref{prop:R3b+}, involving multiplicities,
 and based on the algebraic notion of defect of discrete valuation rings (see \Cref{prop:defect}
 for a reminder on this notion).
 This refined formula is new, although it was already alluded to by Rost in \cite[Rem. (1.8)]{Rost},
 for Milnor K-theory.
 Note that defects can appear only when working in the arithmetic case: if one restricts to schemes
 essentially of finite type over a field, the discrete valuation ring that occurs as localization
 of such a scheme at a regular point of codimension one is always excellent.

\subsection{Reading guide}

For readers who appreciate an axiomatic approach,
 it is advisable to begin with the list of structural maps
 of Milnor-Witt K-theory (Section~\ref{sec:functoriality})
 and the basic rules they satisfy (Section~\ref{sec:rules}).

This work is divided into three parts, 
 each of which includes a review of the necessary background material.
 We begin by briefly reviewing this background, which will be used throughout the paper.

First, we review the theory of quadratic forms over fields,
 but more precisely inner product spaces to deal
 with arbitrary characteristics in Section~\ref{sec:GW}.
 It contains results and computations
 on Grothendieck-Witt and Witt rings relevant to our purposes.

Second, the reader will find a short recollection on cotangent complexes
 and canonical sheaves (or modules) in Section~\ref{sec:cotangent},
 directed towards explicit computations.
 
Finally, we have given in Section \ref{sec:Gduality}
 reminders on Grothendieck coherent duality theory,
 and explain the link with the work of Scheja and Storch, which
 allows one to derive explicit calculations. In particular, we define explicitly
 what we call the \emph{differential trace map},
 and give various interesting properties and formulas:
 the expression of Grothendieck residue symbols (\Cref{df:Grothendieck_residues}),
 the computations in terms of explicit presentations and Bézoutians:
 see \Cref{prop:compute_residues_SS} and \eqref{eq:compute_residues_SS}.

\bigskip

Let us now return to the main structure of the paper.

The first part concerns the definition and basic properties
 of the Milnor-Witt K-ring of a field\footnote{See \Cref{rem:generalizedKMW}
 for the case of rings.}: this is essentially Section \ref{sec:KMW}.
 As mentioned, we start with the presentation by generators and relations
 (\Cref{df:KMW}), and then relate it to the presentation in
 terms of Milnor K-theory and the fundamental ideal
 (\Cref{cor:fundamental_square_KMW}). Recall this relation is a (non-trivial)
 consequence of the second Milnor conjecture (stated in \Cref{thm:MilnorConj}).\\
 Explicit computations are given in, for instance, \ref{prop:KMW&GW},
 \ref{ex:compute_KMW1}, \ref{ex:KMW1/2}, \ref{cor:Kereta} and \ref{ex:KMW_with_sqrt}.
 The main specificity of Milnor-Witt K-theory is twists. We introduce them
 in a second step, and use them to define residue maps at the end of
 Section \ref{sec:KMW}.

The second part is devoted to transfers.
 We start by introducing in Section \ref{sec:CHW}
 the Chow-Witt groups of arbitrary Dedekind schemes,
 both as an illustration (for instance, of the use of twists)
 and as an essential tool for the Bass-Tate method.
 Recall that, compared to classical intersection theory,
 Chow-Witt groups are twisted by a line bundle, and come with
 a bigrading: the first grading is by codimension
 and the second one is a $\GG$-grading, which can be explained by the existence
 of Tate twists for motives.\footnote{Beware that $\GG$-twists refer in practice to twists
 by $\ZZ(1)[1]$ in motivic notation; moreover, the bigraded Chow-Witt groups
 do not correspond to the bigraded Milnor-Witt motivic cohomology of \cite{BCDFO}:
 it is only related to these groups through the coniveau spectral sequence.}
 The main result of the section is the computation of the twisted Chow-Witt groups of the
 projective line (partially stated above), see \Cref{thm:PB1}. The two key tools used here
 are $\AA^1$-invariance of Chow-Witt groups and the localization long exact sequence.
 The connection with usual Chow groups is discussed in \ref{num:CHW&CHbasic}
 and \ref{num:CHW&CH}.

The core study of transfers is done in \Cref{sec:transfers}.
 It follows the plan outlined above.
 We first introduce in Section~\ref{sec:quad-deg+traces} trace maps
 in the monogenic case using the quadratic degree map of \Cref{df:quad_deg}.
 We then revisit a result of Scharlau (\Cref{thm:KMW_GW-differential_Reciprocity}),
 by extending it to Grothendieck-Witt groups and incorporating twists into the statement.
 We refer to this as the \emph{quadratic reciprocity formula}. 
 This allows us to compare in Section~\ref{sec:general-traces}
 the traces obtained in the monogenic case with those built using the differential trace map;
 see \Cref{prop:comparison_MW-transfers}.

\bigskip

The third part of these notes, Section \ref{sec:MW-mod},
 gathers the main functorial properties of Milnor-Witt K-theory,
 following and extending the axiomatic framework of \cite{FeldMWmod, FeldTohoku}.
 We first summarize the existence of four functorialities, built in the previous sections:
 restrictions, corestrictions (also called transfers), action of units (which is a particular
 case of the underlying ring structure) and residues. 
 Then we state several basic properties that can be easily derived from
 what was proved before. There is a major exception, mentioned in the preceding summary:
 formula (R3b), stated in \Cref{thm:R3b}, for which we provide a complete direct proof.
 We conclude the section with Subsection~\ref{sec:strong},
 which contains refined formulas (R1c+), (R3a+) and (R3b+),
 each involving some quadratic multiplicities.

Further applications to Chow-Witt groups can be found in \cite{DFJ}.

\subsection{Acknowledgments}

I am deeply grateful to Jean Fasel for many enlightening discussions on Chow–Witt groups and hermitian 
K-theory, and to Niels Feld for numerous exchanges during and after his PhD, and for his detailed comments on an earlier version of this work.
Discussions with Adrien Dubouloz, especially during the PhD of Clémentine Lemarié-Rieusset and in the course of our collaboration, were a constant source of stimulation and made me realize the central role of twists in the theory.
I also thank Fangzhou Jin, Baptiste Calmès, Stephen McKean, and Robin Carlier for their interest and helpful exchanges, and Fabien Morel for his early influence on my understanding of the subject.

I am indebted to the referee for a patient and detailed reading, and to the editors, Mattia Cavicchi and Jitendra Bajpai, for their patience and trust.
Finally, I wish to thank Daniele Faenzi, Adrien Dubouloz, and Ronan Terpereau for the invitation to lecture at the Spring School ``Invariants in Algebraic Geometry'', where part of these notes took shape.

The author was supported by the ANR project ``HQDIAG'' (Grant No. ANR-21-CE40-0015), and by the MSCA Doctoral Network ``ReMoLD'' (Grant No. 101168795).

\subsection{AI declaration use}

As a signatory of the \emph{Leiden Declaration on Artificial Intelligence and Mathematics}, \url{https://doi.org/10.5281/zenodo.20302944},
 the author declares that the interaction with Large Language Models was used to:
\begin{itemize}
\item Conduct bibliographical research, in particular in connection with the defect of discrete valuation rings used in \Cref{prop:defect}.
\item Proofread and audit successive versions of the manuscript during its final preparation,
 in particular to suggest (many) language corrections and mathematical typos.
 Mathematically, the model identified a missing assumption in \Cref{ex:field_extension_complete_inter},
 which was subsequently corrected, as well as issues in the presentation of the construction of the differential trace maps.
 The latter led the author
 to substantially modify the presentation of this definition; see \Cref{df:diff_trace}.
 The resulting suggestions were examined across Git-tracked versions of the manuscript, and the corrections were selected,
 modified when necessary, and manually incorporated by the author.
\item Adapt the format of the preprint version to the template provided by the editors.
\end{itemize}
No suggested correction was incorporated directly:
 each resulted from iterative prompting and was independently checked and manually edited by the author.

\section{Milnor-Witt K-theory and Grothendieck-Witt groups} \label{sec:KMW}

\subsection{Grothendieck-Witt groups and symmetric bilinear forms}\label{sec:GW}

\begin{num}(cf. \cite{MH})\label{num:inner_spaces}
Let $K$ be a field.
 An inner product space or simply \emph{inner space} $(E,\phi)$ over $K$
 is a finite $K$-vector space $E$ with a bilinear form
$$
\phi:E \otimes_K E \rightarrow K
$$
which is symmetric and non-degenerate: $E \rightarrow E^\vee, x \mapsto \phi(x,-)$ is an isomorphism.
 The dimension of $E/K$ is called the \emph{rank} of the inner space $(E,\phi)$.
 A morphism $(E,\phi) \rightarrow (F,\psi)$ of inner spaces is
 a $K$-linear morphism $f:E \rightarrow F$ such that $\psi(f(u),f(v))=\phi(u,v)$.

The category of inner spaces admits direct sums and tensor products:
\begin{align*}
(E,\phi) \perp (F,\psi) \rightarrow (E \oplus F,\phi+\psi) \\
(E,\phi) \otimes (F,\psi) \rightarrow (E \otimes_K F,\phi.\psi).
\end{align*}
Therefore the set $I_K$ of isomorphism classes\footnote{This is indeed a set, in bijection
 with $$\sqcup_{n \geq 0} \mathrm{Sym}_n(K)/\sim$$
 where $\mathrm{Sym}_n(K)$ is the set of invertible symmetric
 $(n \times n)$-matrices with coefficients in $K$, and $\sim$ is the congruence
 relation on such matrices: $M \sim N$ if $M=PNP^t$;}
 of inner spaces over $K$
 is a commutative monoid for $\oplus$, and a commutative semi-ring for $\oplus, \otimes$.
 The following definition comes from Milnor and Husemöller \cite{MH}.
 It is a variant of the fundamental definition of Witt \cite{Witt}\footnote{now called the Witt group,
 see below},
 using the Grothendieck construction,
 that apparently first appeared in the short work of Delzant \cite{Delz}.
\end{num}
\begin{df}\label{df:GW}
The \emph{Grothendieck-Witt ring} $\GW(K)$ of $K$ is the group completion of the monoid $(I_K,\oplus)$,
 with products induced by the tensor product $\otimes$.
\end{df}
The rank of inner spaces induces a ring map:
\begin{equation}\label{eq:GWrank}
\GW(K) \xrightarrow{\ \rk\ } \ZZ.
\end{equation}

\begin{rem}
If the characteristic of $K$ is different from $2$,
 for any $K$-vector space $V$,
 there is a one-to-one correspondence between symmetric bilinear forms $\phi$ on $V$
 and quadratic forms $q$.\footnote{$q(x)=\phi(x,x)$, $\phi(x,y)=\frac 1 2(q(x+y)-q(x)-q(y))$ !}
 Then the Grothendieck-Witt ring can be defined in terms of isomorphism classes of quadratic forms.

This is no longer true in characteristic $2$, but the definition based on inner spaces
 is the correct one for $\AA^1$-homotopy. Nevertheless, one abusively uses terms such as
 \emph{quadratic} intersection theory, in any characteristic.
\end{rem}

\begin{ex}\label{ex:GW_<>}
\begin{enumerate}
\item Let $u$ be a unit in $K$.
 Then the bilinear form $K \otimes K \rightarrow K, (x,y) \mapsto u.xy$ defines an inner space of rank $1$.
 Its class in the Grothendieck-Witt ring is denoted by $\tw u$.
 Obviously, $\tw{uv^2}=\tw{u}$. Therefore, one has a canonical map:
$$
Q(K):=K^\times/(K^\times)^2 \rightarrow \GW(K).
$$
The group $Q(K)$ will be called the group of \emph{quadratic classes} of $K$.
\item Given units $u_i \in K^\times$, we put $\tw{u_1,\hdots,u_n}=\tw{u_1}+\hdots+\tw{u_n}$.

A bilinear form on a framed $K$-vector space is defined by a symmetric invertible matrix.
 The above element of $\GW(K)$ is represented by the $K$-vector space $K^n$ and the diagonal matrix
 with coefficients $u_i$.
\end{enumerate}
\end{ex}

\begin{ex}\label{ex:GW}
\begin{enumerate}
\item If $K$ is an algebraically closed field the rank map $\rk:\GW(K) \rightarrow \ZZ$ is an isomorphism.\footnote{This
 is obvious in characteristic not $2$, as any inner space admits an orthogonal base, and every element in $K$ is a square.} 
 More generally, $\rk$ is an isomorphism whenever every unit is a square in $K$ (see \Cref{ex:Witt_square}).
\item It is well-known that a quadratic form over a real vector space is determined
 by its signature. In other words, any $\sigma \in \GW(\RR)$ can be
 uniquely written as $\sigma=p.\tw{1}+q.\tw{-1}$, $\rk(\sigma)=p+q$ and the signature of $\sigma$
 is defined as the pair $(p,q)$.
 The map $\GW(\RR) \rightarrow \ZZ \oplus \ZZ, \sigma \mapsto (p,q)$ is an isomorphism.
\item Let $K=\FF_q$ be a finite field, $q=p^n$.
 Then the following sequence of abelian groups is exact:
\begin{align*}
0 \rightarrow Q(\FF_q) & \rightarrow \GW(\FF_q) \xrightarrow{\rk}  \ZZ \rightarrow 0 \\
\bar u & \mapsto 1-\tw{u}
\end{align*}
where $Q(\FF_q)$ is the group of quadratic classes of $\FF_q$ (\Cref{ex:GW_<>}(1)).
Note that this fits with item (1) above!

The preceding sequence is obviously split. Moreover, the abelian group $\FF_q^\times$ is cyclic of order $(q-1)$.
 In particular, Lagrange's theorem implies that $Q(\FF_q)$ is zero if $q$ is even, and $\ZZ/2$ if $q$ is odd.
 Consequently:
$$
\GW(\FF_q)=\begin{cases}\ZZ & q \text{ even,} \\ \ZZ/2 \oplus \ZZ & q \text{ odd.} \end{cases}
$$
\end{enumerate}
\end{ex}

Consider the notations of \Cref{ex:GW_<>}.
 The element $h=\tw{1,-1}$ is called the (class of the) \emph{hyperbolic form}.
 One can recover the following famous definition (and extension in arbitrary characteristic) of Witt.
\begin{df}
One defines the Witt ring of a field $K$ as the quotient ring:
$$
\W(K)=\GW(K)/(h).
$$
\end{df}
The hyperbolic form being of rank $2$, the map \eqref{eq:GWrank} induces a morphism of rings:
$$
\W(K) \rightarrow \ZZ/2
$$
which is again called the \emph{rank map}.

\begin{rem}\label{rem:GW_ring}
\Cref{df:GW}, as well as the previous one,
 can be extended to an arbitrary (commutative) ring $R$ instead of a field $K$
 (see \cite[I.\textsection 4, Prop. 1]{Kneb} for the Grothendieck-Witt ring,
 and \cite[I. 7.1]{MH} for the Witt ring): instead of finite dimensional $K$-vector spaces,
 one considers finitely generated projective $R$-modules $M$ equipped with a
 non-degenerate symmetric bilinear form
$$
\phi:M \otimes_R M \rightarrow R \mid \phi':M \xrightarrow \sim \Hom_R(M,R)=M^\vee
$$
and considers the Grothendieck group $\GW(R)$ associated with the monoid of isomorphism classes
 of $(M,\phi)$.

It follows from Minkowski's convex body theorem (see \cite[Chap. II, 4.4]{MH})
 that one can define an isomorphism of rings, called the \emph{signature},
$$
\sigma:\W(\ZZ) \xrightarrow \sim \ZZ.
$$
The map $\sigma$ associates to the class of $(M,\phi)$
 the signature of $(M \otimes_\ZZ \RR,\phi \otimes_\ZZ \RR)$.

As $h$ is non-$\ZZ$-torsion in $\GW(\ZZ)$
 (because it is not so in $\GW(\RR)$), one obtains that
 $\GW(\ZZ)$ is a free rank $2$ abelian group, and:
$$
\GW(\ZZ)=\ZZ.\tw 1 \oplus \ZZ.h.
$$
To get a presentation as a ring, we consider the element: $\epsilon=-\tw{-1}$.
 Then one deduces from the above isomorphism an isomorphism of rings:
$$
\GW(\ZZ)=\ZZ[\epsilon]/(\epsilon^2-1).
$$
We will retain that this ring always acts (by functoriality) on rings 
 of the form $\GW(K)$,
 and more generally on the invariants of $\AA^1$-homotopy theory
 such as the Milnor-Witt K-theory.
\end{rem}

\begin{df}\label{df:fundamental_I}
We define the \emph{fundamental ideal} of $\W(K)$ as:
$$
\I{}(K):=\Ker\big(\rk:\GW(K) \rightarrow \ZZ\big)
 \simeq \Ker\big(\rk:\W(K) \rightarrow \ZZ/2\big).
$$
Typical elements of $\I{}(K)$ are given by the following
 \emph{Pfister forms} associated with $u \in K^\times$:
$$
\dtw u:=1-\tw u.
$$
\end{df}

\begin{rem}
\begin{enumerate}[wide]
\item According to \cite[3.3]{MH}, $\I{}(K)$ is the only (prime) ideal
 of $\W(K)$ with residue field $\FF_2$.
\item This ideal is of fundamental (historical) importance
 as it is central to the Milnor conjecture on quadratic forms:
 see \Cref{thm:MilnorConj}.
\end{enumerate}
\end{rem}

\begin{ex}
Consider the case of a finite field $K=\FF_q$, $q=p^n$.
According to \Cref{ex:GW}, one gets that 
$$
\I{}(\FF_q)=Q(\FF_q)=\begin{cases}
0 & q \text{ even,} \\
\ZZ/2 & q \text{ odd.}
\end{cases}
$$
So if $q$ is even, $\W(\FF_q)=\ZZ/2$, via the rank morphism.
 If $q$ is odd, one has (applying again the preceding example) a short exact sequence:
$$
0 \rightarrow \ZZ/2 \rightarrow \W(\FF_q) \xrightarrow{\rk} \ZZ/2 \rightarrow  0
$$
which is split if $q=1 \mod 4$, and non-split if $q=3 \mod 4$.
 In fact, as a ring, one gets more precisely:
$$
\W(\FF_q)=\begin{cases}
\ZZ/2 & q \text{ even,} \\
\ZZ/2[t]/(t-1)^2 & q=1 \mod 4, \\
\ZZ/4 & q=3 \mod 4.
\end{cases}
$$
In any case, one deduces that $\I n(\FF_q)=0$ if $n>1$.
\end{ex}

The following result is an elaboration of Witt's theorems
 on quadratic forms (see \cite{Witt}).
\begin{thm}\label{thm:GW}
The abelian group $\GW(K)$ admits a presentation whose generators are given by symbols $\tw{u}$
 for $u \in K^\times$ (mapping to the elements of \Cref{ex:GW_<>}) with relations:
\begin{itemize}
\item[(GW1)] $\tw{uv^2}=\tw{u}$,
\item[(GW2)] $\tw{u,v}=\tw{u+v,(u+v)uv}$, $u+v \neq 0$,
\end{itemize}
where we have used the notation $\tw{u,v}:=\tw{u}+\tw{v}$.

Moreover, the relation (GW2) implies the following one:
\begin{itemize}
\item[(GW3)] $\tw{u,-u}=\tw{1,-1}$.
\end{itemize}
\end{thm}
The analogous presentation for the Witt group is well-known
 (see \cite[Lem. IV.1.1]{MH}). We refer the reader to \cite[Th. 2.8]{Robin} for a direct
 proof (see also \cite[Th. 4.7]{EKM}).\footnote{Here is the trick to get relation (GW3)
 from (GW2). One can assume $u \neq -1$, and one writes using (GW2):
$$
\tw{-u,u+1}=\tw{1,-u(u+1)}, \tw{-1,1+u}=\tw{u,-u(1+u)}.
$$
Subtracting these two equalities yields (GW3).}

\begin{rem}\label{rem:presentations_GW&W}
\begin{enumerate}[wide]
\item If one wants a presentation of $\GW(K)$ as a ring,
 one needs only to add the relation $\tw{uv}=\tw u \tw v$.
\item One can take as generators of the abelian group $\GW(K)$ the symbols $\tw{\bar u}$
 where $\bar u \in Q(K)$ is a quadratic class of a unit $u \in K^\times$.
 Then one has only a single relation, given by (GW2) (with $u$, $v$ replaced by  $\bar u$, $\bar v$).
\item Modulo the hyperbolic plane, one recovers the classical presentation of the Witt group $\W(K)$:
  it is generated by symbols $\tw u$ for a unit $u \in K^\times$ subject to the relations:
\begin{itemize}
\item[(W0)] $\tw{1,-1}=0$
\item[(W1)] $\tw{uv^2}=\tw{u}$
\item[(W2)] $\tw{u,v}=\tw{u+v,(u+v)uv}$, $u+v \neq 0$.
\end{itemize}
Again, one can start with symbols $\tw{\bar u}$ of a quadratic class $\bar u \in Q(K)$,
 in which case relation (W1) is unnecessary.
\end{enumerate}
\end{rem}

\begin{ex}\label{ex:Witt_square}
One deduces from the above presentation the following remarkable property of the Witt group
 of a field $K$. The following conditions are equivalent:
\begin{enumerate}
\item the rank map $\rk:\W(K) \rightarrow \ZZ/2$ is an isomorphism;
\item $Q(K)=1$, \emph{i.e.} every unit in $K$ is a square.
\end{enumerate}
\end{ex}

\begin{num}\label{num:twisted_GW}\textit{Twists}.
In what follows, it will be crucial to consider $\cL$-valued inner spaces
 for an arbitrary invertible $K$-vector space $\cL$. These are $K$ vector spaces $V$
 with a symmetric bilinear form $V \otimes_K V \rightarrow \cL$ such that the adjoint map
 $V \rightarrow \Hom_K(V,\cL)$ is an isomorphism.

Then one can define as above, using again the orthogonal sum,
 a $\GW(K)$-module $\GW(K,\cL)$, called the \emph{$\cL$-twisted Grothendieck-Witt group of $K$}.
 After modding out by $h$, one obtains a $\W(K)$-module $\W(K,\cL)$,
 the \emph{$\cL$-twisted Witt group of $K$}.

The tensor product of $K$-vector spaces induces an action of the ring $\GW(K)$ on $\GW(K,\cL)$,
 and more generally an exterior product:
$$
\GW(K,\cL) \otimes \GW(K,\cL') \rightarrow \GW(K,\cL \otimes \cL'),
$$
and similarly for the twisted Witt groups.
\end{num}

\begin{rem}
Both abelian groups $\GW(K,\cL)$ and $\W(K,\cL)$ are \emph{non canonically} isomorphic
 to their untwisted counterparts.
 However, these twists can be interpreted as local orientations
 in the theory of Chow-Witt groups.
\end{rem}

\subsection{Definition by generators and relations}

\begin{num}\textit{Milnor K-theory}.\label{num:MilnorKth}
Let us first recall that the Milnor K-theory $\KM*(K)$ of a field $K$
 is defined as the $\ZZ$-graded algebra generated by symbols $\{a\}$ in degree $+1$ for $a \in K^\times$
 modulo the relations:
\begin{enumerate}
\item[(M1)] $\{a,1-a\}=0$
\item[(M2)] $\{ab\}=\{a\}+\{b\}$
\end{enumerate}
where we have put $\{a_1,\hdots,a_n\}=\{a_1\}\hdots\{a_n\}$.

Note in particular that: $\KM 0(K)=\ZZ$, $\KM 1(K)=K^\times$.
\end{num}

\begin{rem}\label{rem:Matsumoto}
In general, there is a canonical \emph{symbol map} with values in (Quillen) algebraic K-theory:
$$
\KM n(K) \rightarrow \K n(K)
$$
which is an isomorphism if $n \leq 2$.
 The case $n\leq 1$ is easy, but the case $n=2$ is a difficult theorem due to Matsumoto
 (see \cite{Matsumoto}).
 The cokernel of the symbol map is called the \emph{indecomposable part} of algebraic K-theory.
\end{rem}

We now have all the tools to formulate the Milnor conjecture, now a theorem due to Kato in characteristic $2$,
 and Orlov, Vishik and Voevodsky in the remaining cases:
\begin{thm}[Kato, Orlov-Vishik-Voevodsky]\label{thm:MilnorConj}
Let $K$ be an arbitrary field
 and consider the notation of \Cref{df:fundamental_I}.

Then for any $n \geq 0$, the map  $K^\times \rightarrow \I{}(K), u \mapsto \dtw u$
 induces a ring morphism:
$$
\mu:\KM *(K)/2\KM*(K) \rightarrow \oplus_{n \geq 0} \I n(K)/\I{n+1}(K)
$$
which is an isomorphism.
\end{thm}
See \cite[Question 4.3]{Milnor} for the statement of the conjecture,
 \cite{KatoMilnor} for the proof when $K$ is of characteristic $2$ and \cite{OVV} (or \cite{MorelMilnor}) for the proof in the other cases.

\begin{num}\label{rem:graded_I}\textit{Notation}.-- 
It is customary to denote by $\I *(K)$ the $\ZZ$-graded $\W(K)$-algebra
 where we conventionally put $\I n(K)=\W(K)$ for $n\leq 0$, $\I n(K)$, while for $n>0$ is the $n$-th power of the ideal $\I{}(K)$
 and the product is induced by that of $\W(K)$.

 Then $\I{}(K)$ induces an ideal in $\I*(K)$ and we denote by $\gI*(K)$ the quotient $\ZZ$-graded $\W(K)$-algebra,\footnote{One also
 finds the notation $i^n(K)$ for $\gI n(K)$;}
 so that $\gI n(K)=\I n(K)/\I{n+1}(K)$ if $n \geq 0$, and $0$ otherwise.
 Note that it is clear that the action of $\W(K)$ factors through the rank map so that $\gI*(K)$
 is actually a $\ZZ/2$-algebra.

With this notation and the previous theorem,
 the morphism $\mu$ defined by Milnor takes the form of an isomorphism of $\ZZ$-graded algebras over $\ZZ/2$:
$$
\mu:\KM*(K)/2 \rightarrow \gI*(K).
$$
\end{num}

\begin{ex}\label{ex:discriminant}
The case $n=0$ is trivial. In the case $n=1$, the map takes the form $\mu_1:Q(K) \rightarrow \gI 1(K)=\I{}(K)/\I2(K)$,
 where $Q(K)$ is the group of quadratic classes. Then an explicit inverse is given by the discriminant map
$$
d: \gI 1(K) \rightarrow Q(K), [(E,\phi)] \mapsto (-1)^{r(r-1)/2}.\det(M_\phi)
$$
where $(E,\phi)$ is an inner space of even rank $r$, and $M_\phi$ is any matrix that represents it.
 See \cite[Th. 4.1]{Milnor} and \cite[Def. p.~77]{MH}.
 For the case $n=2$, and the interpretation of $\gI 2(K)$ in terms of Clifford invariant of quadratic forms,
 we refer the reader to \cite{MH}, Theorem~III.5.8 and its proof.
\end{ex}

\begin{num}
The following definition, due to Hopkins and Morel (see \cite[Section 5]{MorelW}),
 gives an extension of Milnor's theory which mixes generators and relations of Milnor K-theory and of the Grothendieck-Witt ring:
\end{num}
\begin{df}\label{df:KMW}
Let $K$ be any field. We define the \emph{Milnor-Witt ring},
 or \emph{Milnor-Witt K-theory}, $\KMW*(K)$ of $K$
 as the $\ZZ$-graded associative algebra with the following presentation.

Generators are given by symbols $[a]$ of degree $+1$ for $a \in K^\times$,
 and a symbol $\eta$ of degree $-1$ called the \emph{Hopf element}.
 Let us introduce the following notations to formulate the relations:
\begin{align*}
[a_1,\hdots,a_n]&=[a_1]\hdots[a_n] \\
h&=2+\eta[-1]
\end{align*}

Relations are given as follows, whenever they make sense:
\begin{enumerate}
\item[(MW1)] $[a,1-a]=0$
\item[(MW2)] $[ab]=[a]+[b]+\eta.[a,b]$
\item[(MW3)] $\eta[a]=[a]\eta$
\item[(MW4)] $\eta h=0$
\end{enumerate}
Obviously, Milnor-Witt K-theory is a covariant functor with respect to morphisms of fields.
 Given such a map $\varphi:K \rightarrow L$, there is an obvious morphism of $\ZZ$-graded rings (homogeneous of degree $0$):
$$
\varphi_*:\KMW*(K) \rightarrow\KMW*(L).
$$
This map is sometimes called the \emph{restriction} (\emph{e.g.}, \cite[Def. (1.1), p. 330]{Rost}).
\end{df}

\begin{rem}\label{rem:generalizedKMW}
Given any ring $A$,
 the preceding definition makes sense so that we can define the ring $\KMW*(A)$.\footnote{See
 \cite[Def. 4.10]{SchEuler} for more developments.}
 The resulting $\ZZ$-graded ring is covariantly functorial in the ring $A$. 
 This extended definition is useful for example when $A$ is a local ring
 as we will see in \Cref{thm:local_Morel}.

Note that one can directly compute this ring when $A=\ZZ$:
$$
\KMW*(\ZZ)=\ZZ\big[\epsilon,\eta,[-1]\big]/(\epsilon^2-1,\epsilon+1+\eta[-1],(1-\epsilon)\eta,(1-\epsilon)[-1])
$$
where $\epsilon$, $\eta$, $[-1]$ are respectively in degree $0$, $-1$ and $1$.\footnote{See also \Cref{num:epsilon_KMW}
 for other occurrences of the important symbol $\epsilon$.}
 In particular, $$\KMW0(\ZZ)=\ZZ[\epsilon]/(\epsilon^2-1)=\GW(\ZZ)$$
 according to \Cref{rem:GW_ring}.
 The ring $\KMW*(\ZZ)$ always acts on rings $\KMW*(K)$ (and more generally
 on invariants of $\AA^1$-homotopy theory).
\end{rem}

\begin{num}\textit{Relation with Milnor K-theory}.\label{num:KMW&KM}
One immediately observes that if one adds $\eta=0$ to the above relations (MW*),
 one recovers the relations (M*) of Milnor K-theory.
 In other words, sending the generators $\{a\}$ to the class of $[a]$ in $\KMW*(K)/(\eta)$
 induces an isomorphism of $\ZZ$-graded algebras:
$$
\KM *(K) \xrightarrow \sim \KMW*(K)/(\eta).
$$
In particular, for any integer $q \in \ZZ$, one deduces an exact sequence of abelian groups:
\begin{equation}\label{eq:preKMW&KM}
\KMW {q+1} (K) \xrightarrow{\gamma_\eta} \KMW q (K) \xrightarrow \forget \KM q (K) \rightarrow 0.
\end{equation}
In the other direction, one can look at the morphism of $\NN$-graded abelian groups
$$
(K^\times)^{\otimes,*} \rightarrow \KMW * (K), u_1 \otimes \hdots \otimes u_q \mapsto h.[u_1,...,u_q] (q \geq 0).
$$
Because of relations (MW1), (MW2) and (MW4), this map factors through relations (M1) and (M2), and therefore induces
 a well-defined morphism of $\ZZ$-graded abelian groups:
$$
\hyper:\KM * (K) \rightarrow \KMW * (K).
$$
Following \cite[Chap. 2, \textsection 1]{BCDFO}, we introduce the following terminology
 for these two maps:
\end{num}
\begin{df}\label{df:KMW&KM}
The morphism of $\ZZ$-graded algebras\footnote{both homogeneous of degree $0$}
 $\forget:\KMW *(K) \rightarrow  \KM * (K)$ 
 and the morphism of $\ZZ$-graded abelian groups $\hyper:\KM * (K) \rightarrow \KMW * (K)$
 are respectively called the \emph{forgetful} and \emph{hyperbolic maps}.
\end{df}
\begin{num}\label{num:KMW&KM2}
By definition, each of the above maps is uniquely characterized by the following properties:
\begin{align*}
&\forget(\eta)=0, \quad \forget([a])=\{a\}. \\
&H(\{a\})=h.[a].
\end{align*}
Moreover, one deduces the following relations\footnote{Use that $h$ modulo $\eta$ is equal to $2$ in $\KMW*(K)$ for the first one}:
\begin{align*}
F \circ H&=2.\Id \\
H \circ F&=\gamma_h.
\end{align*}
\end{num}

\begin{rem}\label{rem:KMW1/2}
In particular, one can remark that the forgetful map induces a split epimorphism:
$$
\KMW*(K)[1/2] \rightarrow \KM*(K)[1/2].
$$
This fact will be made more precise in \Cref{ex:KMW1/2}.
\end{rem}

Let us come back to the study of the general groups $\KMW * (K)$.
 One obtains the following presentation of each individual graded component, as an abelian group:
\begin{prop}\label{prop:KMW_ab_pres}
Consider an arbitrary field $K$ and an integer $n \in \ZZ$.
 Then the abelian group $\KMW n(K)$ is generated by symbols of the form:
$$
[\eta^r,a_1,\hdots,a_{n+r}], r \geq 0, n+r \geq 0, a_i \in K^\times
$$
modulo the following three relations:
\begin{enumerate}
\item[\emph{(MW1ab)}] $[\eta^r,a_1,\hdots,a_{n+r}]=0$ if $a_i+a_{i+1}=1$ for some $i$
\item[\emph{(MW2ab)}] $[\eta^r,a_1,\hdots,a_ib_i,\hdots,a_{n+r}]=[\eta^r,a_1,\hdots,a_i,\hdots,a_{n+r}]+[\eta^r,a_1,\hdots,b_i,\hdots,a_{n+r}]
 +[\eta^{r+1},a_1,\hdots,a_i,b_i,\hdots,a_{n+r}]$
\item[\emph{(MW4ab)}] $[\eta^r,a_1,\hdots,-1,\hdots,a_{n+r-1}]=-2[\eta^{r-1},a_1,\hdots,\cancel{-1},\hdots,a_{n+r-1}]$ for $r\geq 2$
\end{enumerate}
\end{prop}
See \cite{Robin} for the proof.

\begin{cor}\label{cor:KMW_ab_pres}
Assume that $n \geq 1$. Then the abelian group $\KMW n (K)$ is generated by the elements $[u_1,\hdots,u_n]$
 for an $n$-uplet of units $u_i \in K^\times$.
\end{cor}
This simply follows from the previous proposition by using relation (MW2ab).

\begin{rem}\label{rem:KMW_ab_pres}
In particular, the abelian group $\KMW 1 (K)$ is generated by symbols $[u]$ for $u \in K^\times$.
 However, beware that the map $\iota:K^\times \rightarrow \KMW 1 (K), u \mapsto [u]$ is not a morphism of groups in general
 --- except notably when $K=\FF_2$.
 Indeed, one can express the addition law in $\KMW 1 (K)$ by the formula:
$$
[u]+[v]=[uv]-\eta.[u,v],
$$
and $\eta.[u,v]$ is not zero in general.
Note also that the forgetful map
$$
\KMW 1 (K) \xrightarrow \forget \KM 1 (K)=K^\times
$$
is a surjective morphism of abelian groups, but $\iota$ is a splitting of $\forget$
 only after forgetting the group structure.
 In fact, we will give an explicit description of this group in \Cref{ex:compute_KMW1}.
\end{rem}

\begin{num}\label{num:epsilon_KMW}
Following Morel, one considers the following important element
 in Milnor-Witt K-theory:\footnote{Taking into account the isomorphism of \Cref{prop:KMW&GW},
 the latter element of $\KMW 0 (K)$ corresponds to the element $\epsilon=-\tw{-1}$ of $\GW(K)$
 as introduced in \Cref{rem:GW_ring} (taking into account the canonical map $\GW(\ZZ) \rightarrow \GW(K)$).
 This justifies our abuse of notation.}
$$
\epsilon=-(1+\eta.[-1]) \in \KMW 0 (K).
$$
Then relation (MW4) can be rewritten as $\epsilon.\eta=\eta$.
 Moreover, the defect of commutativity of the multiplicative structure of Milnor-Witt
 K-theory can be precisely expressed in terms of $\epsilon$ as follows.
\end{num}
\begin{prop}\label{prop:KMW-epsilon-commut}
For any field $K$, one has the following relation:
$$
\forall (\alpha,\beta) \in \KMW n(K) \times \KMW m(K),
 \alpha\beta=\epsilon^{nm}.\beta\alpha.
$$
\end{prop}
One says that the $\ZZ$-graded algebra $\KMW*(K)$ is \emph{$\epsilon$-commutative}.
 To prove this formula, one reduces to the case $\alpha=[a]$, $\beta=[b]$ for units $a$, $b$
 (see \cite[Cor. 2.6]{Robin}).

\begin{num}\textit{Quadratic multiplicities}.\label{num:quad_mult}
One associates to any integer $n \in \ZZ$ the following element of $K_0^{MW}(K)$:
$$
n_\epsilon=\begin{cases}
\sum_{i=0}^{n-1} (-\epsilon)^i & n \geq 0 \\
\epsilon.(-n)_\epsilon & n<0.
\end{cases}
$$
An equivalent computation:
$$
n_\epsilon=\begin{cases}
m.h  & n=2m \\
m.h+1 & n=2m+1
\end{cases}
$$
Beware that the induced arrow $\ZZ \rightarrow \KMW 0 (K), n \mapsto n_\epsilon$ is a monoid morphism for multiplication
$$
(nm)_\epsilon=n_\epsilon m_\epsilon
$$
but not for the addition in general (compute $3_\epsilon$ and  $4_\epsilon$).
\end{num}

\begin{rem}\label{rem:quad_mult}
\begin{enumerate}[wide]
\item A principle of quadratic enumerative geometry (see \cite{Levine_QE}) is that,
 under a careful choice of orientations,
 degrees of classical enumerative geometry should be replaced by $\epsilon$-degrees
 as defined above.
\item With the previous notation, relation (MW4) translates to:
$$
2_\epsilon.\eta=0
$$
This should remind the reader of the fact that the classical/topological Hopf map $\eta:S^3 \rightarrow S^2$
 induces a $2$-torsion element in the stable homotopy groups of spheres, which accounts for the isomorphism:
$$
\pi^{st}_3(S^2)=\ZZ/2.\eta
$$
where the left-hand side group is the third stable homotopy group of $S^2$.
\item In negative degree, the quadratic multiplicities $n_\epsilon$ become drastically simpler!
 Indeed, modulo $h$ or equivalently in $\W(K)$,
$$
n_\epsilon=\begin{cases}
1 & n \text{ odd} \\
0 & n \text{ even.}
\end{cases}
$$
\end{enumerate}
\end{rem}

\subsection{Relations with quadratic forms}

Using the presentation obtained in \Cref{prop:KMW_ab_pres},
 together with the presentation of Grothendieck-Witt groups in \Cref{thm:GW},
 Morel deduces the following computation (for full details, see \cite[Prop. 2.4, Lem. 2.11]{Robin}):
\begin{prop}\label{prop:KMW&GW}
The following map is well-defined
$$
\GW(K) \rightarrow \KMW0(K), \tw{a} \mapsto 1+\eta.[a]
$$
and induces an isomorphism of rings.

For any $n>0$, the multiplication map: $\KMW 0(K) \xrightarrow{\eta^n} \KMW{-n}(K)$
 induces an isomorphism:
$$
\W(K)=\GW(K)/(h) \rightarrow \KMW{-n}(K).
$$

Finally, for any $n>0$, the abelian group $\KMW n (K)$ is generated by symbols
 of the form $[a_1,\hdots,a_n]$ for units $a_i \in K^\times$.
\end{prop}
As a consequence, we will view the elements of $\GW(K)$ as elements in degree $0$ of Milnor-Witt K-theory.
 Note moreover that $\GW(K)$ lands in the center of the ring $\KMW*(K)$.

\begin{ex}
The notation $h \in \KMW 0(K)$ in relation (MW4) was therefore justified, as it corresponds
 to the hyperbolic form in $\GW(K)$. Note that relation (GW3) in \Cref{thm:GW} can be written as:
\begin{equation}\label{eq:GW_fixes_h}
\forall u \in K^\times, \tw{u}.h=h.
\end{equation}
Remark also that $h^2=2.h$ (direct computation).
\end{ex}

Recall from \Cref{thm:MilnorConj} that given a unit $u \in K^\times$,
 one defines the \emph{Pfister form} associated with $u$ as the element $\dtw u=1-\tw u$ of $\W(K)$.
\begin{cor}\label{cor:KMW_eta-inversed}
Let $\W(K)[t,t^{-1}]$ be the $t$-periodic $\ZZ$-graded algebra with $t$ a formal variable 
 in degree $1$.

Then there exists a unique morphism of $\ZZ$-graded algebras
$$
\phi:\KMW * (K)[\eta^{-1}] \rightarrow \W(K)[t,t^{-1}], [u] \mapsto -\dtw u.t, \eta \mapsto t^{-1}
$$
and it is an isomorphism.
\end{cor}
\begin{proof}
The uniqueness of $\phi$ is obvious.
 We need to show that it is well-defined.
 First note that relation (MW4) implies that $h=0$ in $\KMW*(K)[\eta^{-1}]$. 
 Thus, it suffices to show that the elements $-\dtw u.t$ and $t^{-1}$ of $\W(K)[t,t^{-1}]$ satisfy the relations (MW1), (MW2) and (MW3).
 Relation (MW1) follows from relation (W2) in the Witt ring (see \Cref{rem:presentations_GW&W}(3)).
 Relation (MW2) follows from the rule $\tw u \tw v=\tw{uv}$ in the Witt ring, while relation (MW3) is obvious.
 
Finally, the preceding proposition shows that multiplication by $\eta$ induces an isomorphism on the negative
 part of the $\ZZ$-graded algebra $\KMW*(K)$. In particular, the canonical map $\KMW*(K) \rightarrow \KMW*(K)[\eta^{-1}]$
 is an isomorphism in negative degree.
 On the other hand, $\phi(1+\eta[u])=1-t^{-1}\dtw u.t=1-(1-\tw u)=\tw u$.
 Therefore, applying again the preceding proposition, one deduces that $\phi$ is an isomorphism in negative degree.
 As both the source and target of $\phi$ are $\ZZ$-periodic, one deduces that $\phi$ is an isomorphism in all degrees.
\end{proof}

\begin{num}\label{num:KW&graded_I}
As in \cite{MorelW}, one can define the Witt K-theory of $K$ as the quotient $\ZZ$-graded algebra:
$$
\KW *(K)=\KMW *(K)/(h).
$$
Indeed, the relations (MW*) correspond to the relations of \emph{loc. cit.}, Definition 3.1.

On the other hand, one can consider the subalgebra $\I*(K)$
 of $\W(K)[t,t^{-1}]$ generated by $\I{}(K).t$ and $t^{-1}$ (recall \Cref{df:fundamental_I}
 and the notation of \Cref{rem:graded_I}).
 The main result of \emph{loc. cit.} is that \Cref{thm:MilnorConj} implies
 the following finer comparison result.
\end{num}
\begin{thm}\label{thm:KW&graded_I}
Consider the preceding notation.
 Then there exists a unique morphism $\psi$ of $\ZZ$-graded algebras that
 fits into the commutative diagram
$$
\xymatrix@R=14pt@C=30pt{
\KW * (K)\ar^{\psi}[r]\ar_\nu[d] & \I*(K)\ar[d] \\
\KMW*(K)[\eta^{-1}]\ar^\phi[r] & \W(K)[t,t^{-1}]
}
$$
where $\nu$ is the canonical map (use relation (MW4)),
 and the right-hand vertical one is the obvious inclusion.

Moreover, $\psi$ is an isomorphism.
\end{thm}
This theorem was first proved in \cite{MorelW} when the characteristic of $K$ is different from $2$.
 We refer the reader to \cite[Th. 3.8]{GZS} for a proof in the latter case,
 and to \cite{Robin} for the proof in the characteristic $2$ case.

\begin{num}\label{num:Morel-Milnor_map}
As an application of the previous theorem, one deduces a canonical map:
$$
\mu'_n:\KMW n(K) \longrightarrow \KW n (K) \xrightarrow{(-1)^n.\psi_n} \I n(K)
$$
which can be uniquely characterized, as a morphism $\mu':\KMW*(K) \rightarrow \I*(K)$ of $\ZZ$-graded algebras,
 as the map which sends $[u]$ to the Pfister form $\dtw u \in \I 1(K)$ and the element $\eta$
 to the class $-\tw 1$ in $\I{-1}(K)=\W(K)$.
\end{num}
\begin{cor}\label{cor:fundamental_square_KMW}
The following commutative square of $\ZZ$-graded algebras is cartesian:
$$
\xymatrix@=14pt{
\KMW*(K)\ar^F[r]\ar_{\mu'}[d] & \KM*(K)\ar^\mu[d] \\
\I*(K)\ar^\pi[r] & \gI*(K)
}
$$
Here $F$ is the forgetful map (\Cref{df:KMW&KM}) and $\mu$ is the map defined by Milnor
 (\Cref{thm:MilnorConj} and \Cref{rem:graded_I}).\footnote{Recall
 it sends a generator $\{u\}$ to the class of the Pfister form $\dtw u \in \I 1(K)$ modulo $\I 2(K)$}
\end{cor}
\begin{proof}
Indeed, the above square in degree $n$ fits into the commutative diagram:
$$
\xymatrix@=18pt{
0\ar[r] & \KW{n+1} (K)\ar^/-2pt/{\bar \gamma_{\eta,n}}[r]\ar_{(-1)^n.\psi_{n+1}}[d] &
\KMW n (K)\ar^{F_n}[r]\ar_{\mu'_n}[d] & \KM n (K)\ar^{\mu_n}[d]\ar[r] & 0 \\
0\ar[r] & \I{n+1}(K)\ar[r] &
 \I n(K)\ar[r] & \gI n(K)\ar[r] &0,
}
$$
 where $\bar \gamma_{\eta,n}$ is induced by multiplication by $\eta$,
 and the result follows as $\mu'$ and $F$ are surjective, while $\psi$ is an isomorphism.
\end{proof}

\begin{ex}\label{ex:compute_KMW1}
Looking at degree $1$, we deduce the following explicit description of $\KMW 1(K)$, for any field $K$.
 The group $\KMW 1(K)$ is made of pairs $(\varphi,u)$ where $\varphi$ is the Witt-class of an inner space 
 $\phi:V \otimes_K V \rightarrow K$ of even rank, $u \in K^\times$ is a unit such that
 $d(\varphi)=\bar u \in Q(K)$
where $d$ is the discriminant of $\varphi$ (see \Cref{ex:discriminant}). In other words,
 an element of $\KMW 1(K)$ is given by the Witt class of an inner space over $K$ of even rank
 and a lift of its discriminant in $K^\times$.

In this description, for any unit $u \in K^\times$, the symbol $[u] \in \KMW1(K)$ is sent
 to the pair $(\dtw u,\{u\})$.
\end{ex}

\begin{ex}\label{ex:KMW1/2}
As $\gI*(K)$ is $2$-torsion, one deduces from the previous corollary the following interesting fact
 which extends \Cref{rem:KMW1/2}. After inverting $2$, the canonical maps $F$ and $\mu'$ of the previous corollary
 induce an isomorphism of $\ZZ$-graded rings:
$$
\KMW*(K)[1/2] \xrightarrow{F \times \mu'} \KM*(K)[1/2] \times \I*(K)[1/2].
$$
\end{ex}

\begin{cor}\label{cor:Kereta}
\begin{enumerate}[wide]
\item One has an equality of ideals of $\KMW*(K)$:
$$
\Ker(\gamma_\eta)=(h)=\mathrm{Im}(\hyper)
$$
where $\gamma_\eta$ is multiplication by $\eta$.
In particular the sequence \eqref{eq:preKMW&KM} can be extended into a long exact sequence:
\begin{equation}\label{eq:KMW&KM}
\KM *(K) \xrightarrow{\hyper} \KMW *(K) \xrightarrow{\gamma_\eta} \KMW*(K) \xrightarrow F \KM*(K) \rightarrow 0
\end{equation}
which can be truncated and gives the short exact sequence:
$$
0 \rightarrow \I*(K) \xrightarrow{\bar \gamma_\eta} \KMW*(K) \xrightarrow F \KM*(K) \rightarrow 0
$$
such that $\bar \gamma_\eta$ is homogeneous of degree $-1$.
\item Moreover, the forgetful map $\forget:\KMW*(K) \rightarrow \KM*(K)$
 identifies the principal ideal $(h)$ with the principal ideal $2\KM*(K)$
 generated by $2$ in the Milnor ring of $K$.
 One deduces a short exact sequence:
$$
0 \rightarrow 2\KM*(K) \xrightarrow{\thyper} \KMW*(K) \xrightarrow{\mu'} \I*(K) \rightarrow 0
$$
where $\thyper$ sends $2 \in \KM 0(K)$ to $h$,
 and for $n>0$, sends a $2$-divisible symbol $\{a_1,\ldots,a_n\} \in \KM n(K)$
 to the element $[a_1,\ldots,a_n] \in \KMW n(K)$.

Finally, $\mu' \circ \bar \gamma_\eta$ is equal in degree $n$
 to $(-1)^n.i_n$ where $i_n:I^{n+1}(K) \rightarrow I^n(K)$
 is the canonical inclusion.
\end{enumerate}
\end{cor}
\begin{proof}
Indeed, the preceding theorem implies that $\nu$ is injective,
 which implies that $\Ker(\gamma_\eta)=(h)$ as ideals of $\KMW*(K)$.
 This concludes the first assertion as, by construction, the image of $H$ is the ideal $(h)$.
 The first two exact sequences follow directly, taking into account the
 isomorphism $\psi:\KW*(K) \rightarrow \I*(K)$.
 The last exact sequence follows from the preceding corollary.
\end{proof}

\begin{ex}\label{ex:KMW_with_sqrt}
We finish this subsection with a computation
 that easily follows from \Cref{cor:fundamental_square_KMW} and \Cref{ex:Witt_square}.
 If every unit in $K$ admits a square root, one has:
$$
\KMW n(K)=
\begin{cases}
\ZZ & n=0 \\
\ZZ/2 & n<0 \\
\KM n(K) & n>0.
\end{cases}
$$
Recall also that if $K$ is algebraically closed, 
 for all $n>1$, $\KMW n(K)=\KM n(K)$ is divisible.
 These are therefore very large groups!
\end{ex}

\subsection{Twists}

We now introduce twists on Milnor-Witt K-theory,
 along the lines of \Cref{num:twisted_GW}.
 As already mentioned (see Section \ref{sec:BM_CHtilde}),
 they will account for the local orientations that appear on quadratic cycles
 (see Section \ref{sec:quad_cycles} and in particular \Cref{num:quad_div}).
 Moreover, they are necessary to obtain canonical residue maps (see \Cref{df:specialization}).
\begin{df}\label{df:KMW_twist}
Let $K$ be a field, and $\mathcal L$ be an invertible (\emph{i.e.} of dimension one over $K$) $K$-vector space.
 Consider the set $\mathcal L^\times:=\mathcal L-\{0\}$.
 The action of $K^\times$ on $\KMW*(K)$ via the map $K^\times \rightarrow \KMW0(K), a \mapsto \tw{a}$
 (resp. on $\cL^\times$ by scalar multiplication), gives 
 a structure of $\ZZ[K^\times]$-algebra (resp. $\ZZ[K^\times]$-module) 
 on $\KMW*(K)$ (resp. $\ZZ[\mathcal L^\times]$).
 We define the \emph{$\cL$-twisted Milnor-Witt K-theory} of $K$ (or simply the Milnor-Witt K-theory
 of the pair $(K,\cL)$)
 in degree $n \in \ZZ$ as the following abelian group:
$$
\KMW n(K,\cL):=\KMW n(K) \otimes_{\ZZ[K^\times]} \ZZ[\mathcal L^\times].
$$
\end{df}
Elements of $\KMW n(K,\cL)$ are therefore formal sums of elements of the form
 $\sigma \otimes l$ where $\sigma \in \KMW n (K)$ and $l \in \cL^\times$.

\begin{num}\label{num:twists&functions}
We will identify the untwisted group $\KMW *(K)$ with $\KMW*(K,K)$ via the obvious isomorphism:
$$
\KMW *(K) \rightarrow \KMW *(K,K), \sigma \mapsto \sigma \otimes 1.
$$
Further, given any choice of $l \in \cL^\times$, we get an isomorphism of invertible $K$-vector spaces
 $\Theta_l:K \rightarrow \cL, \lambda \mapsto \lambda.l$ and therefore an isomorphism:
$$
\ev_l=(\Theta_l^{-1})_*:\KMW *(K,\cL) \rightarrow \KMW *(K,K)=\KMW *(K).
$$
According to this definition, for any $u \in K^\times$, one has:
$$
\ev_{ul}=\tw u.\ev_l.
$$
Given an element $\alpha \in \KMW*(K,\cL)$, one obtains a function:
$$
\underline \alpha:\cL^\times \rightarrow \KMW *(K), l \mapsto \ev_l(\alpha)
$$
which is $K^\times$-equivariant: $\underline \alpha(ul)=\tw u.\underline \alpha(l)$.
 In other words, one further deduces the following isomorphism\footnote{To obtain
 the inverse, choose an arbitrary $l \in \cL^\times$,
 and consider $f \mapsto f(l) \otimes l$;}
 of $\ZZ$-graded abelian groups:
$$
\begin{array}{rcl}
\KMW * (K,\cL) & \rightarrow  & \Hom_{K^\times}\big(\cL^\times,\KMW *(K)\big) \\
\alpha & \mapsto & \underline \alpha. \\
\end{array}
$$
\end{num}

\begin{rem}
In particular, the twisted groups $\KMW *(K,\cL)$ are all abstractly isomorphic,
 but via a \emph{non-canonical} isomorphism.

In the theory of quadratic cycles, the invertible vector space $\cL$
 will be the space of local parameters (see e.g., \Cref{num:quad_div}).
 Then one has two interpretations
 of the elements of the twisted groups, in view of the preceding isomorphism:
\begin{itemize}
\item in the form $\alpha=\sigma \otimes l$, $\sigma$ is some coefficient,
 and $l$ is a choice of a local parameter;
\item in the form $\underline \alpha:\cL^\times \rightarrow \KMW*(K)$,
 we have a functional coefficient which to any choice of a local parametrization
 associates some symbol in a $K^\times$-equivariant way.
\end{itemize}
Both points of view are useful.
\end{rem}

\begin{ex}\label{ex:KMW_negative&GW}
Let $(K, \cL)$ be as above.
 Then the isomorphism of \Cref{prop:KMW&GW} induces \emph{canonical} isomorphisms:
\begin{align*}
\GW(K,\cL) & \xrightarrow{\ \simeq\ } \KMW0(K,\cL) \\
\W(K,\cL) & \xrightarrow{\ \simeq\ } \KMW n(K,\cL), \text{for } n<0,
\end{align*}
where the left-hand side was defined in \Cref{num:twisted_GW}.
 Indeed, it suffices to use the isomorphism: 
$$
\GW(K) \otimes_{\ZZ[K^\times]} \ZZ[\cL^\times] \rightarrow \GW(K,\cL), [\phi] \otimes l \mapsto [\phi.l].
$$
\end{ex}

\begin{rem}\label{rem:generalizedKMWtw}
We consider again the situation of Remark \ref{rem:generalizedKMW},
 and assume that $A$ is regular and semi-local (thus noetherian).
 Let $\cL$ be an invertible\footnote{\emph{i.e.} locally free of rank $1$} $A$-module.
 As $A$ is regular semi-local, $\cL$ is trivializable (in other words, $\Pic(A)=0$).
 We let $\cL^\times$ be the subset of $\cL$ made by generators (equivalently, bases)
 of the $A$-module $\cL$. Scalar multiplication gives an action of $A^\times$
 on $\cL^\times$.
 Moreover, the definition of $\tw a=1+\eta.[a]$ in $\KMW0(A)$
 (recall notation from \Cref{rem:generalizedKMW})
 makes sense for any unit $a \in A^\times$. 
 Thus we can define:
$$
\KMW n (A,\cL)=\KMW n (A) \otimes_{\ZZ[A^\times]} \ZZ[\cL^\times]
$$
\end{rem}

\begin{num}\textit{Basic operations on twisted Milnor-Witt K-theory}.\label{num:twKMW_basic}
We have the following structure on twisted Milnor-Witt K-theory:
\begin{enumerate}
\item Products:
$$
\KMW n(K,\cL) \otimes \KMW m(K,\cL') \rightarrow \KMW {n+m}(K,\cL \otimes \cL') ,
 (\sigma \otimes l,\tau \otimes l') \mapsto (\sigma.\tau,l \otimes l').
$$
\item First functoriality: given a morphism of fields $\varphi:K \rightarrow L$,
 one gets:
$$
\varphi_*:\KMW n(K,\cL) \rightarrow \KMW n(L,\cL \otimes_K L), (\sigma,l) \mapsto (\varphi_*(\sigma),l \otimes_K 1_L).
$$
\item Second functoriality: given an isomorphism of $K$-vector spaces $\theta:\cL \rightarrow \cL'$
 one gets:
$$
\theta_*:\KMW n(K,\cL) \rightarrow \KMW n(K,\cL'), (\sigma,l) \mapsto (\sigma,\theta(l))
$$
which is an isomorphism of abelian groups.
\end{enumerate}
\end{num}

\begin{rem}\label{rem:twisted_functoriality}
It is possible to unite the first and second functorialities. One considers the
 category of \emph{twisted fields} $\mathscr{TF}$ whose objects are pairs
 $(K,\cL)$ where $K$ is a field and $\cL$ is an invertible $K$-vector space.
 Morphisms are given by
$$
(\varphi,\Theta):(K,\cL) \rightarrow (L,\cL')
$$
where $\varphi:K \rightarrow L$ is a morphism of fields, and $\Theta:\cL \otimes_K L \rightarrow \cL'$
 is an isomorphism. Composition is defined in the obvious way.
 Then $\KMW *$ becomes a covariant functor from the category of twisted fields
 to the category of graded abelian groups.

The category of twisted fields is \emph{cofibred} over the category of fields $\mathscr{F}$.
 To interpret correctly the tensor product, via a \emph{symmetric} monoidal structure,
 one has to consider the graded category of twisted fields.
 This is obtained via the Grothendieck construction applied
 to the graded Picard category over fields (see \cite{Deligne} for this category
 and \cite{FaselLect} for the monoidal structure).
\end{rem}

\begin{ex}\label{ex:KM_trivial_tw}
Consider $(K,\cL)$ as above.
 Remark that the action of $K^\times$ on $\KMW*(K)/\eta$ via the map $u \mapsto \tw{u}$ is trivial: indeed, $\tw{u}=1 \mod \eta$.
 This implies that $\KMW*(K,\cL)/\eta$ is canonically isomorphic to $\KMW*(K)/\eta=\KM*(K)$,
 which we recall is just the Milnor K-theory of $K$.
\end{ex}

We further extend \Cref{df:KMW&KM} as follows.
\begin{df}\label{df:twistedKMW&KM}
Let $(K,\cL)$ be a twisted field. Then one defines the twisted forgetful (resp. hyperbolic) maps:
\begin{align*}
\forget&:\KMW * (K,\cL) \rightarrow \KM * (K), (\sigma \otimes l) \mapsto \forget(\sigma) \\
\hyper&:\KM * (K) \rightarrow \KMW * (K,\cL), \sigma \mapsto (h\sigma) \otimes l
\end{align*}
where the last formula does not depend on the choice of
 $l \in \cL^\times$, given \Cref{eq:GW_fixes_h}.
\end{df}
Obviously, the two relations of \ref{num:KMW&KM2} still hold with twists.

\begin{num}\label{num:twisted_I}
Consider a twisted field $(K,\cL)$.
 There exists an action of $K^\times$ on the graded algebra $\I *(K)$
 associated with the fundamental ideal $\I{}(K) \subset \W(K)$
 (see \Cref{rem:graded_I}), via its $\W(K)$-module structure. This allows us to define
$$
\I*(K,\cL):=\I*(K)\otimes_{\ZZ[K^\times]} \ZZ[\cL^\times]
$$
as in \Cref{df:KMW_twist}. In fact, one also has
 $\I*(K,\cL) \subset \W(K,\cL)[t,t^{-1}]$ where $t$ is a formal variable
 as in \Cref{cor:KMW_eta-inversed}.
 The isomorphism of \Cref{thm:KW&graded_I} induces an obviously defined twisted version,
 which is still an isomorphism (of $\ZZ$-graded $\W(K)$-modules):
$$
\KW*(K,\cL) \xrightarrow \psi \I*(K,\cL).
$$
As remarked in \Cref{rem:graded_I}, the action of $K^\times$ on the quotient
 algebra $\gI*(K,\cL)$ is trivial. Therefore, one deduces as in \Cref{ex:KM_trivial_tw}
 a canonical identification:
$$
\gI{*}(K,\cL):=\gI*(K)\otimes_{\ZZ[K^\times]}\ZZ[\cL^\times] \xrightarrow{\ \simeq\ } \gI*(K).
$$

These considerations allow to extend \Cref{cor:fundamental_square_KMW}
 as follows:
\end{num}
\begin{prop}\label{prop:twisted_fdl_square_KMW}
The following commutative square of $\ZZ$-graded abelian groups is Cartesian:
$$
\xymatrix@=14pt{
\KMW*(K,\cL)\ar^-F[r]\ar_{\mu'}[d] & \KM*(K)\ar^\mu[d] \\
\I*(K,\cL)\ar^\pi[r] & \gI*(K).
}
$$
Here $F$ is the twisted forgetful map (\Cref{df:KMW&KM}),
 $\mu'$ is the $\cL$-twisted version of the map defined in \Cref{num:Morel-Milnor_map},
 and $\mu$ is the map defined by Milnor (\Cref{thm:MilnorConj}).
\end{prop}

\subsection{Residues}

\begin{num}\label{num:residueMilnor}
Residues are a famous part of the functoriality of Milnor K-theory (see \cite[\textsection 4]{BT}).
 A \emph{discretely valued field} will be a pair $(K,v)$ of a field $K$ with a discrete valuation $v$.
 We let $\cO_v$ be its ring of integers, $\cM_v$ the maximal ideal of $\cO_v$
 and $\kappa_v=\cO_v/\cM_v$ its residue field.
 
 Given a valuation $v:K^\times \rightarrow \ZZ$, with residue field $\kappa_v$, one deduces for
 any $n>0$ a canonical morphism:
$$
\partial_v:\KM {n}(K) \rightarrow \KM{n-1}(\kappa_v)
$$
uniquely characterized by the property:
$$
\partial_v(\{u_1,\hdots,u_n\})=m.\{\overline{u_2},\hdots,\overline{u_n}\}
$$
for units $u_i \in K^\times$ such that $v(u_1)=m$ and for $i>1$, $v(u_i)=0$, $\overline{u_i}$ being the residue class of $u_i$.

The analogous construction exists on Milnor-Witt K-theory, but the twists are now necessary.
\end{num}
\begin{thm}\label{thm:residue}
Consider as above a discretely valued field $(K,v)$.
 The $\kappa_v$-space $\C_v:=\cM_v/\cM_v^2$ is the \emph{conormal cone} associated with $(K,v)$.
 It is an invertible $\kappa_v$-space (\emph{i.e.} of dimension $1$)
 and we let $\omega_v:=(\cM_v/\cM_v^2)^\vee$ be its $\kappa_v$-dual --
 in other words, the \emph{normal cone} associated with $(K,v)$.

Then for any integer $n\in \ZZ$, there exists a unique morphism of abelian groups:
$$
\partial_v:\KMW n(K) \rightarrow \KMW{n-1}(\kappa_v,\omega_v)
$$
satisfying the two following properties:
\begin{enumerate}
\item[(Res1)] $\partial_v(\eta.\sigma)=\eta.\partial_v(\sigma)$, for all $\sigma \in \KMW{n+1}(K)$.
\item[(Res2)] If $n>0$, for any uniformizer $\pi \in K$ and any units $u_1,\hdots,u_n \in K^\times$ such that 
 $u_1=v_1\pi^m$, $v(v_1)=0$, and $v(u_i)=0$ for $i>1$, one has:
$$
\partial_v([u_1,u_2,\hdots,u_{n}])=m_\epsilon\tw{\overline{v_1}}[\overline{u_2},\hdots,\overline{u_n}] \otimes \bar \pi^*
$$
where $\bar \pi^*$ is the dual vector of $\bar \pi$,
 where the latter is seen as a basis of the $\kappa_v$-vector space $\cM_v/\cM_v^2$.
\end{enumerate}
\end{thm}
\begin{proof} (See \cite[3.15, 3.21]{Mor}) We first choose some uniformizing parameter
 $\pi \in \cM_v$ of $v$. Then we introduce the following quotient ring of the indicated polynomial algebra:
$$
A_*=\KMW*(\kappa_v)[\xi]/(\xi^2-[-1].\xi)
$$
which we view as a graded ring by putting $\xi$ in degree $1$.
Then the proof reduces to showing that the canonical map:
$$
K^\times \rightarrow A_*, (u=a\pi^m) \mapsto [\bar a]+m_\epsilon\tw{\bar a}.\xi
$$
extends uniquely to a morphism of $\ZZ$-graded rings
$$
\Theta_\pi:\KMW *(K) \rightarrow A_*
$$
such that $\Theta_\pi(\eta)=\eta$.

Then given $\sigma \in \KMW n(K)$, one can write uniquely:
$$
\Theta_\pi(\sigma)=s_v^\pi(\sigma)+\partial_v^\pi(\sigma).\xi
$$
so that we get two maps
\begin{align*}
s_v^\pi:\KMW *(K) &\rightarrow \KMW*(\kappa_v) \\
\partial_v^\pi:\KMW *(K) &\rightarrow \KMW{*-1}(\kappa_v)
\end{align*}
such that $s_v^\pi$ is a homogeneous morphism of $\ZZ$-graded rings (obvious).

Both maps depend on the choice of $\pi$ in general. We then get the desired canonical map
 by the formula:
\begin{equation}\label{eq:specializeed_residue}
\partial_v(\sigma)=\partial_v^\pi(\sigma) \otimes \bar \pi^*.
\end{equation}
\end{proof}

\begin{df}\label{df:specialization}
Consider the notation of the above theorem.
 The homogeneous morphism of $\ZZ$-graded abelian groups 
$$\partial_v:\KMW*(K) \rightarrow \KMW*(\kappa_v,\omega_v)$$
 of degree $-1$ is called the residue map associated with the valued field $(K,v)$.

Given a prime $\pi$ of $(K,v)$, one also defines the residue map
 specialized at $\pi$ as the map
$$
\partial_v^\pi=\ev_{\bar \pi^*} \circ \partial_v:\KMW*(K) \rightarrow \KMW*(\kappa_v)
$$
 with the notation of \Cref{num:twists&functions}. Equivalently, this is the unique
 homogeneous morphism of $\ZZ$-graded abelian groups of degree $-1$
 such that relation \eqref{eq:specializeed_residue} holds.

Finally, one defines the specialization map
 associated with $(K,v,\pi)$
 as the  morphism of $\ZZ$-graded rings 
$$
s_v^\pi:\KMW *(K)  \rightarrow \KMW*(\kappa_v)
$$
defined in the above proof.
\end{df}

\begin{num}\label{num:basic_formula_res}
Let $(K,v)$ be a valued field, $\pi$ a prime of $v$ and $u \in \cO_v$ a unit.
 One can derive from the previous formula the following rule to compute residues
 for any symbol $\sigma \in \KMW*(K)$:
\begin{align*}
\partial_v^{u\pi}(\sigma)&=\tw{\bar u}\partial_v^{\pi}(\sigma) \\
\partial_v(\tw u\sigma)&=\tw{\bar u}\partial_v(\sigma) \\
\partial_v([u]\sigma)&=\epsilon[\bar u]\partial_v(\sigma)
\end{align*}
The first statement follows from \Cref{num:twists&functions},
 and the other ones follow by using the formula of the previous theorem
 and \Cref{prop:KMW_ab_pres}.

The specialization map can be computed by the following formulas (similar proof):
$$
s_v^\pi(\sigma)=\partial_v^\pi([\pi].\sigma)-[-1]\partial_v^\pi(\sigma)=-\epsilon\partial_v^\pi([-\pi].\sigma).
$$

Consider finally a morphism of rings $R \rightarrow \cO_v$, and let $\varphi:R \rightarrow K$, $\bar \varphi:R \rightarrow \kappa_v$
 be the induced morphisms. Then, for any symbol $\alpha \in \KMW q(R)$ (notation of \Cref{rem:generalizedKMW}),
 one deduces the relation (use the same argument as for the previous relations):
$$
\partial_v(\varphi_*(\alpha)\sigma)=\epsilon^q\bar \varphi_*(\alpha)\partial_v(\sigma).
$$
This implies that $\partial_v$ is $\KMW*(\ZZ)$-linear. In particular,
 it commutes with multiplication not only by $\eta$ but also by $\epsilon$ and $h$.
\end{num}

\begin{num}\label{num:residue}
Consider the assumptions of the previous theorem.
 One can further define, for any invertible $\mathcal O_v$-module $\cL$, a twisted version:
$$
\partial^\cL_v:\KMW n(K,\cL_K) \rightarrow \KMW{n-1}(\kappa_v,\omega_v \otimes_{\kappa_v} \cL_{\kappa_v})
$$
where $\cL_E=\cL \otimes_{\cO_v} E$ for $E=K, \kappa_v$.
 The procedure is a bit intricate:
 take an element $\sigma \otimes l$ on the left-hand side:
 $\sigma \in \KMW n(K)$ and $l \in (\cL_K)^\times$. 
 By definition, there exists a generator $l_0 \in \cL^\times$ of the  $\cO_v$-module $\cL$
 and an element $a \in K^\times$ such that $l=l_0 \otimes_{\cO_v} a$.
 Then one deduces by definition:
\begin{equation}\label{eq:renormalize}
\sigma \otimes l=(\tw{a}\sigma) \otimes (l_0 \otimes_{\cO_v} 1_K).
\end{equation}
One puts:
$$
\partial_v^\cL(\sigma \otimes l)=\partial_v(\tw{a}\sigma) \otimes (l_0 \otimes_{\cO_v} 1_{\kappa_v})
$$
or simply $\partial_v$ when $\cL$ is clear from the context.
\end{num}

\begin{rem}
The necessity to ``renormalize'' the parameter, as in \eqref{eq:renormalize},
when considering residues makes the computation
 in quadratic intersection theory sometimes quite cumbersome!
 Intuitively, we will be following a given orientation from open subschemes 
 to the complementary (reduced) closed subscheme
 (see for example \Cref{num:quad_div}).
\end{rem}

\begin{ex}\label{ex:Witt_res}
We can specialize the definition of the above residue map to negative degree. 
 Then according to \Cref{prop:KMW&GW}, we get a canonical residue map:
$$
\partial_v:\W(K) \rightarrow \W(\kappa_v,\omega_v)
$$
such that
$$
\partial_v(\tw u)=\begin{cases}
0 & v(u) \text{ even}, \\
\tw{\overline{u\pi^{-v(u)}}} \otimes \bar \pi^* & v(u) \text{ odd}, \pi \text{ any uniformizer.}
\end{cases}
$$
(Use \Cref{rem:quad_mult}, point (3)!)
The corresponding untwisted residue map is well-known in Witt theory:
 after the choice of a prime $\pi$, one has $\partial_v^\pi=\psi^1$ in the notation
 of \cite[IV, \textsection 1]{MH}, and it is called the \emph{second residue class
 morphism}.\footnote{The first residue class morphism is defined by the formula
 $\psi^0=\psi^1 \circ \gamma_{\tw{\pi}}$.}

Note also that in degree $0$, we get a more regular formula:
$$
\partial_v:\GW(K) \rightarrow \W(\kappa_v,\omega_v),
 \partial_v(\tw u)=m_\epsilon\tw{\bar a} \otimes \bar \pi^*,  
$$
for $u=a\pi^m$, $v(a)=0$, $v(\pi)=1$.
\end{ex}

\begin{rem}\label{rem:KMW&KM_residues}
Comparing the formulas in \Cref{num:residueMilnor} and \Cref{thm:residue}, it is clear
 that the residue in Milnor-Witt K-theory ``modulo $\eta$'' coincides with the residue map in Milnor K-theory.
 One can be more precise using the maps of \Cref{df:twistedKMW&KM}.
 Given a discretely valued field $(K,v)$, and an invertible $\cO_v$-module $\cL$,
 one gets a commutative diagram:
$$
\xymatrix@R=15pt@C=20pt{
\KM * (K)\ar^-\hyper[r]\ar_{\partial^M_v}[d] & \KMW * (K,\cL_K)\ar^-\forget[r]\ar^{\partial_v}[d] & \KM * (K)\ar^{\partial^M_v}[d] \\
\KM * (\kappa_v)\ar^-\hyper[r] & \KMW * (\kappa_v,\omega_v \otimes \cL_v)\ar^-\forget[r] & \KM * (\kappa_v) 
}
$$
where, for clarity, $\partial_v^M$ is the residue on Milnor K-theory.
 The commutativity of the right-hand square was just explained, while the
 second one follows from the formula $\partial_v(h.\sigma)=h.\partial^M_v(\sigma)$
 (indeed $h$ is unramified with respect to $v$).

Similarly, the second residue morphism on (twisted) Witt $K$-theory of the previous example
 obviously induces a canonical residue map:
$$
\partial_v^{\I{}}:\I n(K,\cL_K) \rightarrow \I {n-1}(\kappa_v,\omega_v \otimes \cL_{\kappa_v}).
$$
On the quotient ring, we get a canonical untwisted residue map:
 $$\partial_v^{\gI{}}:\gI n(K) \rightarrow \gI {n-1}(\kappa_v)$$
 (because of \Cref{num:twisted_I}).
 It is now a routine check to prove that all the maps of the square of \Cref{prop:twisted_fdl_square_KMW}
 are compatible with the corresponding residue maps.
\end{rem}

The following computation is an analogue of the Gersten exact sequence for Milnor K-theory
 (see \cite{Kerz}):
\begin{thm}\label{thm:local_Morel}
Let $(K,v)$ be a discretely valued field, and $\cL$ be an invertible $\cO_v$-module.
\begin{enumerate}
\item  Then the following sequence (see \Cref{rem:generalizedKMWtw} for the first term)
is exact:
$$
\KMW n(\cO_v,\cL) \xrightarrow{\nu_*} \KMW n(K,\cL_K)
\xrightarrow{\partial_v} \KMW{n-1}(\kappa_v,\omega_v \otimes_{\kappa_v} \cL_{\kappa_v}) \rightarrow 0
$$
where $\nu:\cO_v \rightarrow K$ is the obvious inclusion and $\nu_*$
is defined as in \Cref{num:twKMW_basic}(2).
\item If moreover the ring $\cO_v$ contains an infinite field of characteristic not $2$,
then the map $\nu_*$ is injective.
\end{enumerate}
\end{thm}
Idea of proof for point (1).
 The surjectivity of $\partial_v$ is obvious.
 Choose a uniformizer $\pi$ and a generator $l_0 \in \cL^\times$.
 Given any (abelian) generator 
$$\sigma=[\eta^r,v_1,\hdots,v_{n-1+r}] \otimes \bar \pi^*$$
 of the right-hand side group, where $r\geq 0$, $n-1+r\geq 0$ and $v_i\in \kappa_v^\times$
 (see \Cref{prop:KMW_ab_pres}), there exist lifts $u_i \in \cO_v^\times$ of $v_i$,
 along the epimorphism $\cO_v \rightarrow \kappa_v$.
 Then formulas (Res1) and (Res2) imply that
$$
[\eta^r,\pi,u_1,\hdots,u_{n-1+r}]
 \otimes (l_0 \otimes_{\cO_v}1_K)
$$
lifts the generator $\sigma$.

Next, (Res2) implies that $\partial_v\nu_*=0$. Therefore, one only needs to prove
 that the induced map $\Img(\nu_*) \rightarrow \Ker(\partial_v)$ is an isomorphism.
 This is the serious part! We refer the reader to the proof of \cite[Th. 3.22]{Mor}.

Point (2) is the \emph{Gersten conjecture for Milnor-Witt K-theory} and
 for the local ring $\cO_v$. This is due to Gille, Zhong and Scully:
 cf. \cite{GZS}. 
\section{A detour on Chow-Witt groups of Dedekind schemes}\label{sec:CHW}

\subsection{Chow-Witt groups, quadratic divisors and rational equivalence}\label{sec:quad_cycles}

\begin{num}\label{num:residues_RostSchmid}
We let $X$ be a connected noetherian $1$-dimensional scheme which is assumed to be normal (or equivalently regular).
 Let $\cL$ be an invertible sheaf over $X$.
 The main examples are smooth algebraic curves over a field and the spectrum of a Dedekind ring.

Let $\kappa(X)$ be the function field of $X$ and $\cL_\xi$ be the pullback to $\Spec(\kappa(X))$ seen as an invertible $\kappa(X)$-vector space.\footnote{We will also use the notation $\cL_{\kappa(X)}=\cL_\xi$ later.}
 We let $X^{(1)}$ be the set of points $x \in X$ which are closed (\emph{i.e.} of codimension $1$).
 This amounts to asking that the local ring $\cO_{X,x}$ is $1$-dimensional, and therefore a discrete valuation ring.
 In particular, $x$ uniquely corresponds to a valuation $v_x$ on $\kappa(X)$ and  we can consider the associated residue map (\Cref{thm:residue})
$$
\partial_x:\KMW *\big(\kappa(X),\cL_\xi\big) \rightarrow \KMW{*-1}(\kappa_x,\omega_{x/X} \otimes \cL_x)
$$
where $\cL_x$ is the restriction of the invertible $\cO_X$-module $\cL$ to $\kappa(x)$
 and $\omega_{x/X}$ is the normal sheaf of $(\kappa(X),v_x)$.\footnote{The notation $\omega_{x/X}$
 will take all its meaning in \Cref{df:can_sheaf}. See also \Cref{ex:can_sheaf}.}
 Explicitly:
$$
\omega_{x/X}:=\left(\cM_{X,x}/\cM_{X,x}^2\right)^\vee.
$$
Given an element $f \in \KMW*\big(\kappa(X)\big)$, we will interpret $\partial_x(f)$ as the $\KMW{}$-order of $f$ at $x$.
\end{num}
\begin{lm}\label{lm:FD_curve}
With the above notations, for any  $f \in \KMW n\big(\kappa(X)\big)$, the set:
$$
\{x \in X \mid \partial_x(f) \neq 0\}
$$
is finite.
\end{lm}
Given the definition of the residue map, and \Cref{prop:KMW_ab_pres},
 this directly follows from the (more classical) fact:
\begin{lm}
Let $u \in \kappa(X)^\times$ be a unit. Then the set $\{x \in X \mid v_x(u) \neq 0\}$ is finite.
\end{lm}
 Even in our generality, the finiteness is very classical.
 The alert reader will have recognized the support of the divisor associated with the rational function $u$ of $X$
 appearing in the previous lemma!

\begin{rem}\label{rem:FD_curve}
The fact that the scheme $X$ is noetherian is used here to obtain a finite set.
 However, if one only assumes that $X$ is locally noetherian, everything would still be fine,
 as we would obtain a locally finite subset of $X$.
 The theory of cycles, and quadratic cycles, would be fine as we will consider
 locally finite sums. This fits particularly well with the fact that Chow groups (as well as Chow-Witt groups)
 are a kind of \emph{Borel-Moore homology} in topology,
 and the latter is represented by the complex of locally finite singular chains
 (for suitable topological spaces).
\end{rem}

The following definition is a slight generalization of the classical definition
 of Chow-Witt groups
 (see \cite{FaselLect} for Chow-Witt groups of schemes separated and of finite type over a perfect base field).
 We refer the reader to \cite[\textsection 4.1, 4.2]{DFJ} for further developments.
\begin{df}\label{df:CHW_curves}
Consider the previous notation.
 We define the \emph{quadratic divisor class} map as the following sum:
$$
\tdiv_X=\sum_x \partial_x:\KMW *\big(\kappa(X),\cL_\xi\big) \rightarrow \bigoplus_{x \in X^{(1)}} \KMW{*-1}(\kappa_x,\omega_{x/X} \otimes \cL_x)
$$
which is well-defined according to \Cref{lm:FD_curve}.
 This is a homogeneous morphism of $\ZZ$-graded abelian groups of degree $-1$.

We then define the group $\Cq p (X,\cL)_q$ for $p=0$ (resp. $p=1$) as respectively the
 source (resp. target) of $\tdiv_X$ with $*=q$ (resp. $*=q+1$), and as $0$ otherwise.
 Therefore we have obtained a complex $\Cq * (X,\cL)_q$, concentrated in cohomological
 degree $0$ and $1$.
 We call it the \emph{(cohomological) Rost-Schmid complex} of $X$.

We define the \emph{Chow-Witt group} $\CHW p(X,\cL)_q$ of codimension $p$ and $\GG$-degree $q$
 as the cohomology in degree $p$ of this complex.
 When $q=0$, we call it simply the Chow-Witt group, written $\CHW p(X,\cL)$.
\end{df}

\begin{rem}
\begin{enumerate}
\item Beware that the differentials of the Rost-Schmid complex are homogeneous of degree $-1$
 with respect to the $\GG$-grading. There are other possible conventions for the bigrading
 of $\Cq * (X,\cL)_*$ but we will not use them here.
\item Even if we are mainly interested in the Chow-Witt groups,
 the other $\GG$-degrees for $q\neq 0$ will be crucial for computations.
 See \Cref{sec:localization}.
\item The groups $\CHW p(X,\cL)_q$ are analogues of the higher Chow groups.
 However, they do not deserve the name higher Chow-Witt groups as they only contribute to some part of the latter
 (that one can interpret as the Milnor-Witt motivic Borel-Moore homology; see \cite{BY,BCDFO}).
 In fact, while the latter are represented by a full ring spectrum $\mathrm H_{\mathrm{MW}}\ZZ$,
 the former are represented by a truncation
 of $\mathrm H_{\mathrm{MW}}\ZZ$.
 On the other hand, the groups just defined satisfy the same formalism as higher Chow groups.
\item If one replaces Milnor-Witt K-theory by Milnor K-theory,
 one obtains Rost's ($\GG$-)graded Chow groups
 $\CH^p(X)_q$ defined in \cite{Rost}. This was in fact the model for the previous definition.
 We refer the reader to \Cref{num:CHW&CH} for more discussion.
\end{enumerate}
\end{rem}

\begin{ex}
In codimension $0$, $\CHW 0(X,\cL)$ is the kernel of the map in degree $0$:
$$
\GW\big(\kappa(X),\cL_\xi\big) \rightarrow \bigoplus_{x \in X^{(1)}} \W(\kappa_x,\omega_{x/X} \otimes \cL_x).
$$
A virtual inner $\cL_\xi$-space over the function field $\kappa(X)$
 which is in the kernel of this map
 is said to be \emph{unramified} (with respect to the curve $X$).
\end{ex}

\begin{num}\textit{Quadratic divisors}.\label{num:quad_div}
Let us make explicit the above definition when $q=0$.
The abelian group $\CHW 1(X,\cL)$
 is the cokernel of the map in degree $1$:
$$
\tdiv:\KMW 1\big(\kappa(X),\cL_\xi\big) \rightarrow \bigoplus_{x \in X^{(1)}} \GW(\kappa_x,\omega_{x/X} \otimes \cL_x).
$$
The abelian group at the target will be called the group of \emph{quadratic divisors} (or $1$-codimensional cycles) of $(X,\cL)$.
 These are formal sums of the form
\begin{equation}
\sum_{i \in I} (\sigma_i \otimes \bar \pi_i^* \otimes l_i).x_i
\end{equation}
where:
\begin{itemize}
\item $x_i \in X$ is a closed point,
\item $\sigma_i \in \GW(\kappa_{x_i})$ is the class of an inner space over $\kappa_{x_i}$,
\item $\pi_i$ is a uniformizing parameter of the valuation ring $\cO_{X,x_i}$,
 (equivalently a \emph{local parameter} of the closed subscheme $\{x_i\} \subset X$),\footnote{the
 notation $\bar \pi_i^*$ reminds the reader that we consider the element in
 $\omega_{x_i/X}=(\cM_{X,x_i}/\cM_{X,x_i}^2)^\vee$ corresponding to $\pi_i$}
\item $l_i \in \cL_{x_i}$ is a nonzero element.
\end{itemize}
In practice, one can also view the coefficient $(\sigma_i \otimes \bar \pi_i^* \otimes l_i)$
 as a virtual inner $(\omega_{x_i/X} \otimes \cL_{x_i})$-space over $\kappa(x_i)$.
 Recall also from \Cref{num:twists&functions}
 the interpretation of this latter element as a $\kappa(x_i)^\times$-equivariant
 map from the space of nonzero elements $(\omega_{x_i/X} \otimes \cL_{x_i})^\times$
 to the Grothendieck-Witt group $\GW(\kappa_{x_i})$.

As in the classical case,
 quadratic divisors which are in the image of $\tdiv$ are said to be \emph{principal}.
 Two quadratic divisors are \emph{rationally equivalent} if their difference is principal.
\end{num}

\begin{ex}
In the case $X$ is in addition local, thus the spectrum of a discrete valuation ring $\cO_v$,
 \Cref{thm:local_Morel} implies in particular:
$$
\CHW p(\cO_v)=\begin{cases}
\GW(\cO_v) & p=0, \cO_v \supset k_0 \\
0 & p=1.
\end{cases}
$$
where $k_0$ is an infinite field of characteristic not $2$.
The vanishing of $\CHW 1(\cO_v)$ can be interpreted by saying that every quadratic divisor of $X$ is \emph{principal}.
\end{ex}

\begin{num}\textit{Quadratic order of vanishing}.
One can also make explicit the definition of the quadratic divisor class map.
 Let us fix a point $x \in X^{(1)}$, and let $v_x$ be the corresponding discrete valuation on $\kappa(X)$.
 We know that the abelian group $\KMW 1 (\kappa(X))$ is generated by elements $[f]$ for a unit $f \in \kappa(X)^\times$
 (see \Cref{cor:KMW_ab_pres} and \Cref{rem:KMW_ab_pres}).
 Given a rational function $f \in \kappa(X)^\times$ on $X$, we get with the notation of the above definition:
\begin{equation}\label{eq:quadratic_order}
\partial_x([f])=m_\epsilon.\tw{\bar u} \otimes \bar \pi_x^* \in \GW(\kappa_x,\omega_{x/X})
\end{equation}
where we have chosen a local parameter $\pi_x$ of $x$ in $X$ (\emph{i.e.} a uniformizing parameter of the
 valuation ring $\cO_{X,x}$), $m=v_x(f)$ is the classical order of vanishing of $f$ at $x$,
 and $u=f.\pi_x^{-m}$ and $\bar u$ is its class in $\kappa_x=\cO_{X,x}/\cM_{X,x}$.
 The formula, as well as the fact this element does not depend on the particular choice of $\pi_x$,
 directly follows from \Cref{thm:residue}.
\end{num}
\begin{df}
Consider the above assumptions.
 We define the quadratic order of vanishing of a rational function $f \in \kappa(X)^\times$
 as the element $\tord_x(f)=\partial_x([f])$ in $\GW(\kappa_x,\omega_{x/X})$.
\end{df}
One can rewrite the definition of the divisor class map when $q=0$ in more classical terms:
$$
\tdiv([f])=\sum_{x \in X^{(1)}} \tord_x(f).x.
$$

\begin{rem}
One should be careful that the quadratic order of vanishing, as well as the quadratic divisor class map,
 is only additive in $f$ with respect to the addition of $\KMW 1 (\kappa(X))$, which in general
 differs from the group law of $\kappa(X)^\times$ (see \Cref{rem:KMW_ab_pres}).
\end{rem}

\begin{num}\label{num:CHW&CHbasic}
Let us consider the previous definitions modulo $\eta$. Then we get in degree $0,1$ a map, independent of $\cL$:
$$
\kappa(X)^\times=\KM 1\big(\kappa(X)) \xrightarrow{\tdiv_X \text{ mod } \eta} \bigoplus_{x \in X^{(1)}} \KM 0(\kappa_x)=Z^1(X)
$$
where the right-hand side is the group of (ordinary!) $0$-cycles of $X$. This is precisely the divisor class map: in fact, one obviously 
 has the formula
$$
\ord_x(f)=\tord_x(f) \text{ mod } \eta,
$$
which amounts to saying that the rank of the underlying inner space of $\tord_x(f)$
 is the classical order of vanishing $\ord_x(f)$ of $f$ at $x$ (use Formula \Cref{eq:quadratic_order}).
In particular, we get:
$$
\CHW p(X,\cL)/(\eta)=\begin{cases}
\ZZ^{\pi_0(X)} & p=0 \\
\Pic(X) & p=1.
\end{cases}
$$
Moreover, one can describe explicitly the image of the map:
$$
\CHW p (X,\cL) \rightarrow \CHW p (X,\cL)/(\eta) \simeq \CH^p(X)
$$
It is just induced by the rank map: in degree $0$, it sends an unramified inner $\cL$-space $\sigma$ over $\kappa(X)$
 to its rank $\rk(\sigma)$. In degree $1$, it sends a quadratic $0$-cycle
$$
\sigma=\sum_{i \in I} \sigma_i \otimes \bar \pi_i^* \otimes l_i.x_i
$$
to the $0$-cycle:
$$
\rk(\sigma)=\sum_{i \in I} \rk(\sigma_i).x_i.
$$
\end{num}

\begin{num}\label{num:CHW&CH}
We can be more precise about the relation between Chow and Chow-Witt groups,
 using the definitions of \Cref{df:twistedKMW&KM}. Indeed, \Cref{rem:KMW&KM_residues} implies
 that the following diagram is commutative:
$$
\xymatrix@R=15pt@C=25pt{
\KM 1\big(\kappa(X)\big)\ar^-{\hyper_\xi}[r]\ar_{\Div_X}[d] & \KMW 1\big(\kappa(X),\cL_\xi\big)\ar^-{\forget_\xi}[r]\ar^{\tdiv_X}[d]
 & \KM 1\big(\kappa(X)\big)\ar_{\Div_X}[d] \\
Z^1(X)\ar^-{\sum_x \hyper_x}[r] & \Cq 1 (X,\cL)\ar^-{\sum_x \forget_x}[r] & Z^1(X)
}
$$
where $Z^1(X)$ denotes the group of codimension $1$ algebraic cycles
 (\emph{i.e.} the Weil divisors) of $X$.
Taking cokernel, one gets well-defined maps:
$$
\CH^1(X) \xrightarrow \hyper \CHW 1 (X,\cL) \xrightarrow \forget \CH^1 (X)
$$
whose composite is multiplication by $2$.
 We still call them respectively the \emph{hyperbolic}
 and \emph{forgetful} maps.
\end{num}

\subsection{Homotopy invariance over a field}

Our next result was first proved for Milnor K-theory by Milnor: see \cite[Th. 2.3]{Milnor}
 (and also \cite[5.2]{BT}). It was generalized by Morel in \cite[Th. 3.24]{Mor}.
\begin{thm}[Morel]\label{thm:htp_inv_fields}
Let $k$ be an arbitrary field, $X=\AA^1_k$ with function field $k(t)=\kappa(X)$.
 Let $\varphi:k \rightarrow k(t)$ be the obvious inclusion.

Then the quadratic divisor class map of $X$ fits into the following sequence
$$
0 \rightarrow \KMW q(k) \xrightarrow{\ \varphi_*\ } \KMW q\big(k(t)\big) \xrightarrow{\ \tdiv_X\ } \bigoplus_{x \in X^{(1)}} \KMW {q-1}(\kappa_x,\omega_{x/X}) \rightarrow 0
$$
which is split exact.

In particular,
$$
\CHW p(\AA^1_k)_q=\begin{cases}
\KMW q(k) & p=0 \\
0 & p=1.
\end{cases}
$$
\end{thm}
Note that a splitting is easy to get: considering the $t$-adic valuation $v_t$ on $k(t)$,
 the specialization map $s_{v_t}^t$ (\Cref{df:specialization}) gives a splitting.
 More generally, any discrete valuation $v$ on $k(t)$ trivial on $k$, with residue field $k$ and uniformizing parameter
 $\pi$, will give a splitting $s_v^\pi$.

The proof of this theorem uses the same trick as in Milnor's proof,
 and argues inductively on the degree in $t$. The idea is to filter $\KMW*(k(t))$
 by the subring $L_d$ generated by $\eta$ and symbols of the form $[P(t)]$ where
 $P(t)$ is a polynomial of degree less or equal to $d$. We can then argue inductively
 on the $\ZZ$-graded $\KMW *(k)$-rings $L_d$ using an explicit presentation
 of the $\ZZ$-graded $\KMW *(k)$-module $L_d/L_{d-1}$.

As an example, the reader is encouraged to work out for himself
 the case of $L_1$. The hint is to use the (obvious!) exact sequence:
$$
0 \rightarrow k^\times \xrightarrow{\ \varphi_*\ } k(t)^\times  \xrightarrow{\sum_x v_x} Z^1(\AA^1_k) \rightarrow 0
$$

Given that invertible sheaves on $\AA^1_k$ are trivializable,
 one immediately deduces the twisted version of the previous theorem.
\begin{cor}\label{cor:htp_inv_fields}
Consider the notation of the previous theorem, and let $\cL$ be an
 invertible sheaf on $\AA^1_k$.
 Then the following sequence of abelian groups is exact:
$$
0 \rightarrow \KMW q(k,\cL_0) \xrightarrow{\ \varphi_*\ } \KMW q\big(k(t),\cL_{k(t)}\big)
 \xrightarrow{\ \tdiv_X\ } \bigoplus_{x \in X^{(1)}} \KMW {q-1}(\kappa_x,\omega_{x/X} \otimes \cL_x)
 \rightarrow 0
$$
where $\cL_0$ (resp. $\cL_x$) is the fiber of $\cL$ over the point $0$ (resp. a closed point $x$).
 In particular, $\cL_0=\cL \otimes_{k[t]} k$ and the morphism $\varphi_*$ is defined on twists
 using the unique isomorphism:
$$
\cL\xrightarrow{\sim}\cL_0\otimes_k k[t]
$$
whose reduction modulo $t$ is the identity of $\cL_0$. Explicitly, if $e$ is any generator of the $k[t]$-module $\cL$
 and $e_0$ denotes its image in $\cL_0$, this isomorphism is given by
$$
f(t)e\longmapsto e_0\otimes f(t).
$$
\end{cor}

\subsection{Localization exact sequences}\label{sec:localization}

In this section, we will illustrate the usefulness of considering the $\GG$-grading
 of the Rost-Schmid complex (\Cref{df:CHW_curves}).
 The aim is to compute the Chow-Witt groups of the projective line.
\begin{num}
Let again $X$ be a normal connected $1$-dimensional scheme, $\cL$ an invertible sheaf on $X$.

Consider in addition a finite subset $Z \subset X$ of closed points of $X$, seen as a reduced closed subscheme, and let $i:Z \rightarrow X$ be the closed immersion.
 Let $\omega_{Z/X}=(\cI(Z)/\cI(Z)^2)^\vee$ be the normal sheaf of $i$, where $\cI(Z) \subset \cO_X$ is the ideal sheaf.
 Let $U=X-Z$, and $j:U \rightarrow X$ the open immersion.

There is an obvious split epimorphism:
$$
j^*:\Cq 1(X,\cL)_q \rightarrow \Cq 1(U,\cL)_q
$$
whose kernel is the finite sum:
$$
\CHW 0(Z,\omega_{Z/X} \otimes \cL_Z)_q:=\oplus_{z \in Z} \KMW q\big(\kappa_z,\omega_{z/X} \otimes \cL_z\big).
$$
Remark that this notation fits in with the previous considerations as for any point $z \in Z$, we have a \emph{canonical} isomorphism
 (this can be checked directly, or see \eqref{eq:rel_canonical})
 of invertible $\kappa_z$-vector spaces:
$$
\omega_{z/X} \otimes \cL_z \simeq (\omega_{z/Z} \otimes \omega_{Z/X}|_z) \otimes \cL_z \simeq \omega_{z/Z} \otimes (\omega_{Z/X} \otimes \cL_Z)_z.
$$
Assembling all this, we get a commutative diagram whose lines are exact:
$$
\xymatrix@=20pt{
& 0\ar[r]\ar[d] & \Cq 0(X,\cL)_{q+1}\ar@{=}[r]\ar_{\tdiv_X}[d] & \Cq 0(U,\cL)_{q+1}\ar[r]\ar^{\tdiv_U}[d] & 0 \\
0\ar[r] &\CHW 0(Z,\omega_{Z/X} \otimes \cL_Z)_q\ar^-{i_*}[r] &  \Cq 1(X,\cL)_q\ar^-{j^*}[r] & \Cq 1(U,\cL)_q\ar[r] & 0
}
$$
\end{num}
\begin{df}
Consider the previous notation. The exact sequence obtained by applying the snake lemma to the
 preceding commutative diagram:
\begin{align*}
0 \rightarrow \CHW 0(X,\cL)_{q+1} &\xrightarrow{j^*} \CHW 0(U,\cL_U)_{q+1}
 \xrightarrow{\partial_{Z/X}} \CHW 0(Z,\omega_{Z/X} \otimes \cL_Z)_q \\
 & \xrightarrow{\ i_*\ }  \CHW 1(X,\cL)_q\xrightarrow{j^*} \CHW 1(U,\cL_U)_q \rightarrow 0
\end{align*}
is called the \emph{localization exact sequence} associated with $i$.

The connecting map $\partial_{Z/X}$ is called the \emph{residue map} associated with $i$.
 It is induced by the following restriction and corestriction of the quadratic divisor class map $\tdiv_X$:
$$
\sum_{z \in Z} \partial_z:\KMW {q+1} \big(\kappa(X),\cL_\xi\big) \longrightarrow \oplus_{z \in Z} \KMW q\big(\kappa_z,\omega_{z/X} \otimes \cL_z\big).
$$
\end{df}

\subsection{Twisted Chow-Witt groups of the projective line}

\begin{num}
We now illustrate the usage of the localization exact sequence defined in the previous section. Let $k$ be an arbitrary field.

Let $\PP^1_k=\Proj(k[x,y])$ be the projective line, $\infty=[1:0]$ be the point at infinity with complementary open subscheme $\AA^1_k=\Spec(k[x])$.
 We let $i^\infty:\{\infty\} \rightarrow \PP^1_k$ be the natural closed immersion, and $j:\AA^1_k \rightarrow \PP^1_k$ the complementary open immersion.
 We fix a line bundle $\cL$ over $\PP^1_k$,  which is therefore determined up to isomorphism by its degree, $\deg(\cL)$.
 We let $\cL'$ be the restriction of $\cL$ to $\AA^1_k$.

Then the localization exact sequence of $i^\infty$, with $(X,Z)=(\PP^1_k,\{\infty\})$,
 together with Morel's homotopy invariance theorem
 (see \Cref{cor:htp_inv_fields}) gives us the following exact sequence:
\begin{align*}
0 \rightarrow \CHW 0(\PP^1_k,\cL)_{q+1}& \xrightarrow{j^*} \KMW{q+1}(k,\cL_0)
 \xrightarrow{\partial_{\infty/\PP^1}} \KMW q(k,\omega_\infty \otimes \cL_\infty) \\
 &\xrightarrow{\ i^\infty_*\ }  \CHW 1(\PP^1_k,\cL)_q \rightarrow 0
\end{align*}
where we have denoted by $\cL_0$ the restriction of $\cL'$ to the point $0$ in $\AA^1_k$,
 and put $\omega_\infty=\omega_{\infty/\PP^1_k}$ with the notation of \Cref{num:residues_RostSchmid}.
 The main problem is to determine the boundary map $\partial_{\infty/\PP^1}$.
\end{num}
\begin{lm}
Consider the above assumptions and notations.

Then if $\deg(\cL)$ is even, $\partial_{\infty/\PP^1}=0$.
 If $d=\deg(\cL)$ is odd, after choosing an isomorphism $\cL \simeq \cO(d)$,
 and using the isomorphism $\omega_\infty \simeq k$ given by the uniformizing parameter $1/x$,
 we get the following commutative diagram:
$$
\xymatrix@C=40pt@R=20pt{
\KMW{q+1}(k,\cL_0)\ar^-{\partial_{\infty/\PP^1}}[r]\ar_\sim[d] & \KMW q(k,\omega_\infty \otimes \cL_\infty)\ar^\sim[d] \\
\KMW{q+1}(k)\ar^{\gamma_\eta}[r] & \KMW q(k).
}
$$
\end{lm}
\begin{proof}
One reduces to the case $\cL=\cO(d)$.
 We consider $U=U_\infty$ (resp. $U_0$) the open complement of $\infty$ (resp. $0$) in $\PP^1_k$,
 so that $U_\infty=\Spec(k[x])$ and $U_0=\Spec(k[y])$.
 The gluing map $U_0 \cap U_\infty \rightarrow U_\infty \cap U_0$ is given by
 mapping $x$ to $y^{-1}$.
 Then the line bundle $\cL=\cO(d)$ is given on $U_\infty$ (resp. $U_0$) by a free module
 $\cL'_\infty=k[x].u$ (resp. $\cL'_0=k[y].v$) with a gluing map $u \mapsto y^d.v$.

Note that one has preferred isomorphisms:
 $\cL_0 \simeq_u k$ and $\cL_\infty \simeq_v k$. Therefore we deduce a canonical map
$$
\KMW{q+1}(k) \simeq_{u^{-1}_*} \KMW{q+1}(k,\cL_0)
 \xrightarrow{\partial_{\infty/\PP^1}} \KMW q(k,\omega_\infty \otimes \cL_\infty) \simeq_{y_* \otimes v_*} \KMW q(k)
$$
denoted by $\partial'_{\infty/\PP^1}$.

We compute the image of $\sigma \in \KMW{q+1}(k)$ under $\partial'_{\infty/\PP^1}$.
 First, $u^{-1}_*(\sigma)=\sigma \otimes u$.
 Then we need to use the map $\varphi_*$ of \Cref{cor:htp_inv_fields}, which sends the latter to
$$
\sigma \otimes (u \otimes 1) \in \KMW{q+1}\big({k(t),\cL_0 \otimes_k k(t)}\big).
$$
In order to compute its residue at $\infty$, one needs to write it as an element
 of $\KMW{q+1}\big(k(t),\cL_\infty \otimes_k k(t)\big)$. Therefore, one uses the above change of variables:
$$
\sigma \otimes (y^dv \otimes 1)=(\tw{y^d}\sigma) \otimes (v \otimes 1).
$$
Now if $d$ is even, $\tw{y^d}=1$ and we get: $\partial_\infty^y(\tw{y^d}\sigma)=0$
 as $\sigma$ comes from $\KMW*(k)$. Thus $\partial'_{\infty/\PP^1}(\sigma)=0$.

If on the contrary, $d$ is odd, $\tw{y^d}=\tw y$.
 Therefore 
$$
\partial_\infty^y(\tw{y^d}\sigma)=\partial_\infty^y(\tw{y}\sigma)=\eta.\sigma
$$
and one deduces that $\partial'_{\infty/\PP^1}(\sigma)=\eta.\sigma$ as expected.
\end{proof}

\begin{num}\label{num:PB1}
Let $d=\deg(\cL)$. The lemma and the localization exact sequence give the following possibilities:
\begin{enumerate}
\item if $d$ is even, one gets isomorphisms:
\begin{align*}
j^*&:\CHW 0(\PP^1_k,\cL)_{q} \xrightarrow{\ \sim \ } \CHW 0(\AA^1_k,\cL)_{q} \simeq \KMW{q}(k,\cL_0) \\
i^\infty_*&:\KMW q(k,\omega_\infty \otimes \cL_\infty) \xrightarrow{\ \sim \ }  \CHW 1(\PP^1_k,\cL)_q.
\end{align*}
\item If $d$ is odd, and after the choices indicated in the above lemma, we get an exact sequence:
\begin{align*}
0 \rightarrow \CHW 0(\PP^1_k,\cL)_{q+1}& \xrightarrow{j^*} \KMW{q+1}(k)
 \xrightarrow{\gamma_\eta} \KMW q(k) \xrightarrow{\ i^\infty_*\ }  \CHW 1(\PP^1_k,\cL)_q \rightarrow 0
\end{align*}
\end{enumerate}
\end{num}
Recall that the cokernel of $\gamma_\eta$ is $\KM q(k)$ (see \Cref{num:KMW&KM}),
 and its kernel is $2.\KM{q+1}(k)$, that is the $(q+1)$-st graded part of the ideal
 generated by $2$ in the ring $\KM*(k)$ (see \Cref{cor:Kereta}).
To summarize, we have obtained the following computation of (graded) Chow-Witt groups,
 first proved by Jean Fasel for a perfect base field of characteristic not $2$ (see \cite{FaselPB}):
\begin{thm}\label{thm:PB1}
Consider the above assumption: $k$ is an arbitrary field, $\cL$ an invertible sheaf over $\PP^1_k$
 of degree $d$. Then
\begin{align*}
\CHW 0(\PP^1_k,\cL)_q &\simeq \begin{cases}
\KMW{q}(k,\cL_0) & \qquad \quad d \text{ even} \\
2\KM q(k) & \qquad\quad d \text{ odd}
\end{cases} \\
\CHW 1(\PP^1_k,\cL)_q &\simeq \begin{cases}
\KMW q(k,\omega_\infty \otimes \cL_\infty) & d \text{ even} \\
\KM q(k) & d \text{ odd}
\end{cases}
\end{align*}
Recall finally from \Cref{cor:Kereta} and \Cref{num:quad_mult} that: $\Ker(\eta)=2_\epsilon\KM q (k)$ when $\car(k) \neq 2$.
\end{thm}
Let us draw the picture for Chow-Witt groups:
$$
\CHW p(\PP^1_k,\cL)=\begin{cases}
\GW(k,\cL_0) & p=0, \deg(\cL)=2r, \\
\GW(k,\omega_\infty \otimes \cL_\infty) & p=1, \deg(\cL)=2r, \\
\ZZ & p=0, 1, \deg(\cL)=2r+1.
\end{cases}
$$
In particular, these groups depend on the twist $\cL$ when $\GW(k)$ is non-trivial! 
\section{Transfers}\label{sec:transfers}

\subsection{Cotangent complexes and canonical sheaves}\label{sec:cotangent}

Recall for convenience (and completeness) the following definition.
\begin{df}\label{df:lci}
Let $f:X \rightarrow S$ be a morphism of schemes.
\begin{enumerate}
\item $f$ is \emph{smoothable} if there exists a factorization 
$$
f:X \xrightarrow i P \xrightarrow p S
$$
 such that $p$ is smooth and $i$ is a closed immersion.
\item $f$ is a \emph{complete intersection} if there exists a factorization 
$$
f:X \xrightarrow i P \xrightarrow p S
$$
 such that $p$ is smooth and $i$ is a regular closed immersion.
\item $f$ is a \emph{local complete intersection} if any point $x \in X$ admits an open neighborhood $V$ such that
 the restriction $f|_V$ is a complete intersection. \\
Following a classical abuse, we will simply say that $f$ is lci.
\end{enumerate}
\end{df}

\begin{rem}\begin{enumerate}[wide]
\item The second definition first appeared in \cite[VIII, 1.1]{SGA6}.
 For the first definition, we also refer the reader to \cite[Tag 068E]{Stack}.
\item A morphism $f$ is a complete intersection if and only if it is smoothable and lci (see \cite[VIII, 1.2]{SGA6}).
\end{enumerate} 
\end{rem}

\begin{num}\label{num:cotangent&canonical}\textit{Cotangent Complexes}.
For a scheme $X$, we let $\Der(\cO_X)$ be the derived category of $\cO_X$-modules.
 It can be obtained as the homotopy category of a stable $\infty$-category (see \cite[Section 1]{LurieHA}).
 However, we will not use this higher structure and we will only need
 its canonical triangulated structure.

 Let $f:X \rightarrow S$ be a morphism of schemes.
 Recall that one can associate to $f$ its cotangent complex $\cL_{X/S}$ (see \cite[III, 1.2.3]{Illusie}),
 a canonically defined object of $\Der(\cO_X)$ --- 
 it is the derived functor of the K\"ahler differential functor evaluated at $\cO_X/f^{-1}\cO_S$.

If $f$ is a complete intersection, choosing a factorization as in \Cref{df:lci}(2),
 one can explicitly compute its cotangent complex $\cL_{X/S}$.
 It is quasi-isomorphic to a complex concentrated in two degrees
$$
\C_{X/P} \rightarrow \Omega_{P/S}|_X
$$
where $\C_{X/P}=\cI_i/\cI_i^2$ is the \emph{conormal sheaf} associated with $i$, placed in homological degree $+1$,
 and $\Omega_{P/S}$ is the \emph{cotangent sheaf} of $P/S$ (the conormal sheaf of the diagonal of $P/S$) placed in degree $0$
 (see \cite[VIII, 3.2.7]{IllusieII}).
 This obviously implies that if $f:X \rightarrow S$ is only assumed to be lci,
 then its cotangent complex is Zariski locally in $X$ quasi-isomorphic
 to a complex concentrated in degree $0$ and $1$ and whose terms are free.\footnote{One says
 that $\cL_{X/S}$ has perfect homological amplitude in $[0,1]$.} In particular, $\cL_{X/S}$ is perfect.\footnote{A complex $\mathcal K$ of $\cO_X$-modules
 is perfect if any point of $X$ admits an open neighborhood $U$ such that $\mathcal K|_{U}$ is quasi-isomorphic to a bounded complex $\mathcal L$
 such that for all integers $n$ the $\cO_U$-module $\mathcal L^n$ is a direct factor of a finite free $\cO_U$-module.
 See \cite[Tags 08C4 and 08CM]{Stack}.}

The interest of the cotangent complex lies in being compatible with composition
 in the following sense.
 Consider a commutative diagram
\begin{equation}\label{eq:triangle_morph_schemes}
\xymatrix@R=12pt@C=18pt{
X\ar^f[rr]\ar[rd] &\ar@{}|/-4pt/\Delta[d]& Y\ar^g[ld] \\
& S &
}
\end{equation}
of morphisms of schemes. Then one has a canonical distinguished triangle in $\Der(\cO_X)$
 (induced by an exact sequence in the underlying stable $\infty$-category):
\begin{equation}\label{eq:cotangent_htp_seq}
(\derL f^{*}\cL_{Y/S}) \rightarrow \cL_{X/S} \rightarrow \cL_{X/Y}.
\end{equation}
\end{num}

\begin{num}
Recall from \cite[Ex. 4.13]{Deligne} that one associates to a perfect complex $C$ of $\cO_X$-modules
 its \emph{rank} $\rk(C)$ which is a locally constant function $X \rightarrow \ZZ$
 and its determinant  $\det(C)$ which is a well-defined invertible sheaf over
 $X$.
\footnote{The couple $(\det,\rk)$ is actually the left Kan extension, as an 
 $\infty$-functor, from the $\infty$-category of perfect complexes
 to the $\infty$-groupoid of graded line bundles,
$$
\Perf(X) \rightarrow \cPic(X)
$$
 of the functor sending a locally free $\cO_X$-module to its rank and its maximal exterior power
 (see \cite[\textsection 5]{LODef}).
 It is also obtained by restriction of the canonical functor from the Thomason-Trobaugh $K$-theory space $K(X)$
 to $\cPic(X)$ (see \cite{BS} in the affine case).}
\end{num}
\begin{df}\label{df:can_sheaf}
Let $f:X \rightarrow S$ be a morphism whose cotangent complex is perfect (e.g., lci).
 One associates to $f$ its \emph{canonical sheaf}:
$$
\omega_{X/S}=\det(\cL_{X/S}).
$$
We will also say that $f$ is of (virtual) relative dimension $d=\rk(\cL_{X/S})$.
\end{df}
When $X/S$ is the spectrum of a ring extension $B/A$,
 the canonical sheaf $\omega_{X/S}$ is determined by its global sections.
 We will denote by $\omega_{B/A}$ the $B$-module of its global sections,
 and call it the \emph{canonical module} associated with $B/A$.

\begin{ex}\label{ex:can_sheaf}
\begin{enumerate}[wide]
\item If $f:X \rightarrow S$ is smooth, the above definition coincides with the classical definition
 of the canonical sheaf: the cotangent sheaf of $f$ is locally free $\Omega_{X/S}$,
 and $\omega_{X/S}$ is the maximal exterior power of $\Omega_{X/S}$ as a $\cO_X$-module.

Note in particular that when $f$ is \'etale, one has an equality: $\omega_{X/S}=\cO_X$.
 This is really an identity, and not just an isomorphism.
\item If $f=i:Z \rightarrow X$ is a regular closed immersion of pure codimension $1$,
 then $\omega_{Z/X}=\C_{Z/X}^\vee$, the dual of the conormal sheaf.
\item A morphism $f:X \rightarrow S$ of schemes which is flat, of finite presentation and lci is called
 \emph{syntomic} after Fontaine and Messing. Syntomic morphisms are stable under composition and base change
 (\cite[Tags 01UH, 01UI]{Stack}). In this case the virtual relative dimension function of $X/S$ equals the
 function which to a point $x \in X$ associates $\dim_x(X_{f(x)})$.
 This can be seen by reducing to the case where $S$ is the spectrum of a field as the cotangent complex of $f$
 is stable under (naive) pullbacks.
\end{enumerate}
\end{ex}

\begin{num}\label{eq:can_iso_can_sheaf}
Let us consider a commutative diagram \eqref{eq:triangle_morph_schemes}
 such that the cotangent complexes of all three morphisms are perfect (for example, $f$ and $g$ are lci).
 Then the above homotopy exact sequence translates into a canonical isomorphism
 of invertible sheaves over $X$:
\begin{equation}\label{eq:rel_canonical}
\psi_\Delta:\omega_{X/S} \simeq \omega_{X/Y} \otimes (f^{*}\omega_{Y/S})
\end{equation}
\end{num}

\begin{rem}\label{rem:can_iso_can_sheaf}
It is also useful to consider commutative squares:
$$
\xymatrix@=15pt{
Y\ar^g[r]\ar_q[d]\ar@{}|\Theta[rd] & X\ar^p[d] \\
T\ar_f[r] & S.
}
$$
Dividing the square into two commutative triangles, and applying
 the preceding isomorphism for both triangles, one gets a canonical isomorphism:
$$
\psi_\Theta:\omega_{Y/X} \otimes (g^*\omega_{X/S}) \simeq \omega_{Y/T} \otimes (q^*\omega_{T/S}).
$$
When the preceding square is affine corresponding to a commutative square of rings:
$$
\xymatrix@=15pt{
D\ar@{}|\Theta[rd] & C\ar[l] \\
B\ar[u] & A\ar[u]\ar[l].
}
$$
one gets the following simpler form, an isomorphism of invertible $D$-modules:
$$
\psi_\Theta:\omega_{D/C} \otimes_C \omega_{C/A} \simeq \omega_{D/B} \otimes_B \omega_{B/A}
$$
where the tensor product on the left (resp. right) is taken with respect to
 the induced structure of $C$-module on $\omega_{D/C}$ (resp. $B$-module on $\omega_{D/B}$).
\end{rem}

\begin{ex}\label{ex:affine_complete_intersection}
Let us consider a finitely generated lci $A$-algebra $B$.
 We assume that there exists a smooth $A$-algebra $R$ and a regular ideal $I \subset R$
 such that $B \simeq R/I$ as an $A$-algebra
 so that we get an epimorphism $\varphi:R \rightarrow B$.\footnote{This will automatically be the case if $A$ is a field,
 as we can choose $R$ to be a polynomial $A$-algebra.}

Assume $\Spec(A)$ and $\Spec(R)$ are irreducible and let $n$ be the rank of the locally free $R$-module $\Omega_{R/A}$,
 $m$ be the height of $I$.
Then one can compute the canonical module of $B/A$ as:
$$
\Theta:\omega_{B/A} \simeq \omega_{B/R} \otimes_B (\omega_{R/A} \otimes_R B)
 \simeq \Lambda^m_B(I/I^2)^\vee \otimes_R \Omega^n_{R/A}.
$$
Indeed, as $I$ is regular, $I/I^2$ is a locally free $B$-module of
 constant rank $m$.

In general, we have an exact sequence of $B$-modules:
\begin{equation}\label{eq:differential_seq}
0 \rightarrow N \rightarrow I/I^2 \xrightarrow \psi \Omega_{R/A} \otimes_R B
 \xrightarrow{\varphi_*} \Omega_{B/A} \rightarrow 0
\end{equation}
where $N=\Ker(\psi)$, $\psi$ is induced by the composition
$$
I \hookrightarrow R \xrightarrow{d_R} \Omega_{R/A} \rightarrow \Omega_{R/A} \otimes_R B
$$
and the last map is induced by $\varphi:R \rightarrow B$.
 As recalled in \Cref{num:cotangent&canonical}, the cotangent complex
 $\cL_{B/A}$, that we view as a complex of $B$-modules,
 is concentrated in homological degree $[0,1]$ and one deduces from the above
 exact sequence and the homotopy exact sequence \eqref{eq:cotangent_htp_seq}
 the following isomorphisms:
\begin{align*}
H_0(\cL_{B/A}) & \simeq \Omega_{B/A}, \\
H_1(\cL_{B/A}) & \simeq N.
\end{align*}

When $B/A$ is \'etale, one gets $\Omega_{B/A}=N=0$, and $n=m$.
 So $\omega_{B/A}=B$ (this is really an identity),
 and the isomorphism 
 $\Theta:B \simeq \Lambda^n(I/I^2)^\vee \otimes_B \Lambda^n(\Omega_{R/A} \otimes_R B)$
 is obtained by transposing the isomorphism $\psi$.
\end{ex}

\begin{ex}\label{ex:field_extension_complete_inter}
 We consider a particular case of the preceding example,
 that of a finite field extension $L/K$.
 We can choose a set of generators $(\alpha_1,...,\alpha_n)$, $L=K[\alpha_1,...,\alpha_n]$.
 If we consider the polynomial $K$-algebra $R=K[t_1,...,t_n]$, then one can write
 $L=R/I$, and $I=(f_1,...,f_n)$
 where $f_i$ is a polynomial in the variables $t_1,...,t_i$, monic in $t_i$,
 which is a lift of the characteristic polynomial of $\alpha_i$
 in the simple extension
 $K[\alpha_1,...,\alpha_i]/K[\alpha_1,...,\alpha_{i-1}]$. Thus, $I$ is regular.
 Then we get from the previous example a canonical isomorphism:
\begin{equation}\label{eq:canonical_mod_field_ext}
\Theta:\omega_{L/K} \simeq \Lambda^n_L(I/I^2)^\vee \otimes_R \Omega^n_{R/K}.
\end{equation}
We then get an explicit basis of the invertible $L$-module $\omega_{L/K}$,
 given by the element:
\begin{equation}\label{eq:canonical_base_field_ext}
(\bar f_1 \wedge \hdots \wedge \bar f_n)^* \otimes (dt_1 \wedge \hdots \wedge dt_n).
\end{equation}

If $L/K$ is \underline{separable}, as explained at the end of the previous example,
 $\Omega^1_{L/K}=0$, so $\omega_{L/K}=L$. According to the description
 of $\psi$, one obtains that the element
 \eqref{eq:canonical_base_field_ext} goes under $\Theta^{-1}$ to the unit:
$$
\big(f'_1(\alpha_1)f'_2(\alpha_1,\alpha_2)
 \hdots f'_n(\alpha_1,\hdots,\alpha_n)\big)^{-1} \in L^\times.
$$
We have to take the inverse of the obvious element as $\Theta$ is obtained after transposition,
 as seen in the end of the previous example.

Let us assume on the contrary that $L/K$ is \underline{totally inseparable
 and modular}\footnote{\emph{i.e.}, a tensor product of monogenic totally inseparable extensions over $K$}.
 Let $p>0$ be the characteristic of $K$.
 Then $\alpha_i=(a_i)^{1/q_i}$, $a_i \in K-K^p$. Moreover, in the sequence
 \eqref{eq:differential_seq} with $B/A=L/K$, one obtains that $\psi=0$.
 In other words, one gets isomorphisms:
\begin{align*}
\varphi_*:\Omega_{R/K} \otimes_R L \xrightarrow \sim \Omega_{L/K}, \\
N \simeq I/I^2.
\end{align*}
In particular, $(d\alpha_1,...,d\alpha_n)$, which is the image of $(dt_1,...,dt_n)$ by 
the isomorphism $\varphi_*$, is an $L$-basis of $\Omega_{L/K}$,
 which can be identified with $H_0(\cL_{L/K})$.
 Similarly, $N \simeq I/I^2$ can be identified with $H_1(\cL_{L/K})$,
 and an $L$-basis is given by $(\bar f_1,...,\bar f_n)$
 --- each $\bar f_i$ goes to zero in $\Omega_{R/K} \otimes_R L$.
\end{ex}

\begin{rem}
In the case of a totally inseparable extension $L/K$, $\FF_p \subset K$,
 one defines the \emph{imperfection module} $\Gamma_{L/K}$ of $L/K$
 by the following exact sequence:
$$
0 \rightarrow \Gamma_{L/K} \rightarrow \Omega_{K/\FF_p} \otimes_K L
 \rightarrow \Omega_{L/\FF_p} \rightarrow \Omega_{L/K} \rightarrow 0.
$$
One deduces that $H_1(\cL_{L/K}) \simeq \Gamma_{L/K}$. In particular,
 with the notations of the previous paragraph, the imperfection module $\Gamma_{L/K}$
 is an $n$-dimensional $L$-vector space which is isomorphic
 to $I/I^2$.
\end{rem}

\subsection{The quadratic degree map}\label{sec:quad-deg+traces}

\begin{num}\label{num:homological_quad-cycle_P1}
We will now come back to \Cref{thm:PB1} and give its fundamental application to build transfers
 on Milnor-Witt K-theory.

Let $k$ be an arbitrary field,
 and $\omega=\omega_{\PP^1_k/k}$ be the canonical sheaf on $\PP^1_k$ (\Cref{df:can_sheaf}),
 and let $\infty$ (resp. $\eta$) be the point at infinity (resp. generic point) of $\PP^1_k$.
 We first rewrite the quadratic divisor class map in \emph{homological} conventions.
 Consider a point $x \in \PP^1_k$ with residue field $\kappa_x$.
 Note that $\kappa_x/k$ is not necessarily separable so the canonical sheaf $\omega_{\kappa_x/k}$ can be non-trivial. 
 Nevertheless, the commutative diagram
$$
\xymatrix@=5pt{
\Spec(\kappa_x)\ar^{x}[rr]\ar[rd] && \PP^1_k\ar[ld] \\
& \Spec(k) &
}
$$
gives a canonical isomorphism $\psi^x:\omega_{\kappa_x/k} \simeq \omega_{x/\PP^1_k} \otimes \omega|_x$ --- see \eqref{eq:rel_canonical}.

In particular, the quadratic divisor class map for $\PP^1_k/k$ in $\GG$-degree $q \in \ZZ$
 and with twists $\omega$ can be rewritten as:
$$
\tdiv:\KMW{q+1}\big(k(t),\omega_{k(t)/k}\big)
 \longrightarrow \bigoplus_{x \in \PP^1_{k,(0)}} \KMW q(\kappa_x,\omega_{\kappa_x/k})=:\tilde C_0(\PP^1_k)_q.
$$
Recall that $\tdiv$ is the sum of the residue maps
 $\partial_x:\KMW{q+1}\big(k(t),\omega_{k(t)/k}\big) \rightarrow \KMW q(\kappa_x,\omega_{\kappa_x/k})$
 for $x$ a closed point in $\PP^1_k$, corresponding to a valuation $v_x$ on $k(t)$ with residue field $\kappa_x$.

The cokernel of $\tdiv$ equals the Chow-Witt group $\CHW 1(\PP^1_k,\omega)_q$ and,
 as the line bundle $\omega$ has even degree, \Cref{thm:PB1} and paragraph \ref{num:PB1}
 tell us that the pushforward map
$$
i^\infty_*:\KMW q(k) \rightarrow \CHW 1(\PP^1_k,\omega)_q
$$
is an isomorphism. Let us introduce the following definition.\footnote{This is the mother case
 of the degree map on Chow-Witt groups. See \cite[1.4.1, Ex. 4.1.6]{DFJ}.}
\end{num}
\begin{df}\label{df:quad_deg}
Using the above notation, we denote by
$$
\tdeg_q:\CHW 1 (\PP^1_k,\omega)_q \rightarrow \KMW q(k)
$$
the inverse of the isomorphism $i^\infty_*$
 and call it the \emph{quadratic degree map} in $\GG$-degree $q$ (associated with $\PP^1_k$).
\end{df}
In degree $q=0$, we therefore get a map:
$$
\tdeg:\CHW 1 (\PP^1_k,\omega) \rightarrow \GW(k).
$$

Following Bass and Tate (\cite[I.5.4]{BT}) and Morel (\cite[\textsection 4.2]{Mor}),
 we can be more precise about this notion of quadratic degree.
\begin{prop}\label{prop:BT}
Consider the above assumptions and notation.
 Then there exists a \emph{unique} family of maps
$$
\Tr^{MW}_{\kappa_x/k}:\KMW q(\kappa_x,\omega_{\kappa_x/k}) \rightarrow \KMW q (k), x \in \PP^1_{k,(0)}
$$
which fits into the following commutative diagram
$$
\xymatrix@R=16pt@C=80pt{
& \KMW q(k)\ar@{^(->}_{i^\infty_*}[d]\ar@{=}[rd] & \\
\KMW{q+1}\big(k(t),\omega_{k(t)/k}\big)\ar_-{\tdiv}[r] & \tilde C_0(\PP^1_k)_q\ar_-{\sum_x \Tr^{MW}_{\kappa_x/k}}[r]
 & \KMW q(k) 
}
$$
in such a way that the composition of the horizontal maps is zero.
\end{prop}
In particular, the quadratic degree map is defined at the level of cycles:
$$
\tdeg_q:=\sum_{x \in \PP^1_{k,(0)}} \Tr^{MW}_{\kappa_x/k}:\tilde C_0(\PP^1_k)_q \rightarrow \KMW q(k).
$$
The last condition in the above statement can be translated by saying that
 the quadratic $0$-cycles of degree $0$ on $\PP^1_k$
 are exactly the principal (\emph{i.e.} rationally trivial) quadratic divisors
 (using the terminology of \Cref{num:quad_div}).

\begin{rem}\label{rem:Weil_RC}
\begin{enumerate}[wide]
\item Note that the commutative triangle corresponds to the normalization
 property: $\Tr^{MW}_{\kappa_\infty/k}=\Id$.
\item The formula $\tdeg \circ \tdiv=0$ is the quadratic analogue of the \emph{Weil reciprocity formula}.
 Given the preceding normalization property, it can be restated
 as the following equation:
$$
\sum_{x \in (\AA^1_k)_{(0)}} \Tr^{MW}_{\kappa_x/k} \circ \partial_x=-\partial_\infty.
$$
\end{enumerate}
\end{rem}

\begin{num}\label{num:quad&classic_deg}\textit{Quadratic and classical degree}.
As the pushforward morphism $i^\infty_*$ is compatible with the forgetful and hyperbolic maps of \Cref{num:CHW&CH},
 we get by definition the following commutative diagram:
$$
\xymatrix@R=14pt@C=28pt{
\CH^1(\PP^1_k)\ar^-\hyper[r]\ar_{\deg}[d] & \CHW 1 (\PP^1_k,\omega) \ar^-\forget[r]\ar^{\tdeg}[d] & \CH^1(\PP^1_k)\ar^{\deg}[d] \\
\ZZ\ar^-\hyper[r] & \GW(k)\ar^-{\rk}[r] & \ZZ.
}
$$
Specializing at a point $x \in \PP^1_k$ with residue field $\kappa_x$ as before,
 one gets the following computation:
\begin{align}
&\forall n \in \ZZ, \Tr^{MW}_{\kappa_x/k}(n.h)=(d_xn).h, \\
&\forall \sigma \in \GW(\kappa_x,\omega_{\kappa_x/k}), \rk\big(\Tr^{MW}_{\kappa_x/k}(\sigma)\big)=d_x.\rk(\sigma),
\end{align}
where $d_x=[\kappa_x:k]$.
\end{num}

\begin{rem}
Notice in particular that every quadratic cycle
 which comes from the hyperbolic map will have a degree of the form $n.h$
 for $n \in \ZZ$ and $h$ the hyperbolic form.
\end{rem}

\begin{num}\label{num:MWtr-monogenic}\textit{Transfers in the monogenic case}.
Let $E/k$ be a monogenic finite field extension. Giving a generator $\alpha \in E$ is equivalent to giving a closed embedding
 $x:\Spec(E) \rightarrow \AA^1_k$, corresponding to the (monic) minimal polynomial of $\alpha$ over $k$.
 Therefore, the preceding proposition gives for any integer $q \in \ZZ$ a well-defined transfer map:
$$
\Tr^{MW,\alpha}_{E/k}:\KMW q(E,\omega_{E/k}) \rightarrow \KMW q (k),
$$
which a priori depends on the chosen parameter $\alpha$.

We also define an $\cL$-twisted version, for an invertible $k$-vector space $\cL$, as follows:
$$
\begin{array}{rcl}
\Tr^{MW,\alpha}_{E/k}:\KMW q(E,\omega_{E/k} \otimes \cL_E) &\rightarrow &\KMW q(k,\cL), \\
 \sigma \otimes w \otimes l &\mapsto& \Tr^{MW,\alpha}_{E/k}\big(\tw u\sigma \otimes w\big) \otimes l'
\end{array}
$$
where $\sigma \in \KMW q(E)$,
 $w \in \omega_{E/k}^\times$, $l \in \cL_E^\times$ and we have written: $l=l' \otimes u$ for $l' \in \cL^\times$, $u \in E^\times$
 (as according to our notation $\cL_E=\cL \otimes_k E$).
\end{num}

\begin{rem}
\begin{enumerate}
\item We will see in \Cref{prop:comparison_MW-transfers} that the above transfers are independent of the generator $\alpha$,
 and we will extend their definition to finite extensions which are not necessarily monogenic.
\item Our construction is a variation on Morel's one, as done in \cite[\textsection 4.2, 5.1]{Mor}.
 The main difference is that one uses appropriate twists (by canonical sheaves)
 which allow us to work over an arbitrary base field $k$, in particular allowing inseparable extensions from the start.
\end{enumerate}
\end{rem}

\begin{ex}\label{ex:algorithm-norms}
One can derive from \Cref{prop:BT} the following way to compute the above trace map,
 for a monogenic extension $E/k$ and an explicit presentation $E=k[t]/(f)$,
 where $f$ is the minimal polynomial of the chosen generator $\alpha$.
 Let $v_f$ be the valuation on $k(t)$ corresponding to $f$.

Consider an element $\sigma \in \KMW n(E,\omega_{E/k})$.
 According to \Cref{thm:htp_inv_fields}, there exists an element $\varphi \in \KMW{n+1}(k(t),\omega_{k(t)/k})$
 such that for any maximal ideal of $k[t]$, corresponding to a valuation $v$,
$$
\partial_v(\varphi)=\begin{cases}
\sigma & v=v_f, \\
0 & \text{otherwise.}
\end{cases}
$$
Then, one deduces from \Cref{rem:Weil_RC}(2) that
$$
\Tr^{MW,\alpha}_{E/k}(\sigma)=-\partial_\infty(\varphi)
$$
where $\partial_\infty$ is the residue map corresponding
 to the place at infinity of $k(t)$.
\end{ex}

\begin{num}
The method of Bass and Tate, already mentioned, was applied to Milnor K-theory (\cite[I.5.4]{BT}).
 They constructed the transfer map on Milnor K-theory for monogenic finite extensions,
 and later, Kato proved that these transfers extend to arbitrary finite extensions $E/k$
 (\cite[\textsection 1.7, Prop. 5]{KatoNorm}),
 giving a transfer map\footnote{Kato called this the norm homomorphism}:
$$
\Tr^M_{E/k}:\KM *(E) \rightarrow \KM *(k).
$$
In particular, when $E/k$ is monogenic, this map coincides with the Bass-Tate morphism
 for any choice of generator $\alpha$ of $E/k$.
 As the (twisted) hyperbolic and forgetful maps
 (\Cref{df:twistedKMW&KM}) are compatible with residues (\Cref{rem:KMW&KM_residues})
 we easily derive from the above construction the following compatibility lemma
 (extending \Cref{num:quad&classic_deg}).
\end{num}
\begin{lm}\label{lm:mono_transf_F&H}
Let $E/k$ be a monogenic finite extension with generator $\alpha \in E$.
 Then the following diagram is commutative:
$$
\xymatrix@R=22pt@C=30pt{
\KM*(E)\ar^-{\hyper}[r]\ar_{\Tr^M_{E/k}}[d]
 & \KMW*(E,\omega_{E/k}) \ar^-{\forget}[r]\ar^{\Tr^{MW,\alpha}_{E/k}}[d]
 & \KM*(E)\ar^{\Tr^M_{E/k}}[d] \\
\KM*(k)\ar^-{\hyper}[r] & \KMW*(k)\ar^-{\forget}[r] & \KM*(k).
}
$$
\end{lm}

\subsection{A variation on Scharlau's quadratic reciprocity property}

\begin{num}\label{num:differential_MW_transfer}
Let $E/k$ be a finite extension field.
 Recall that Scharlau has defined in \cite{SchRec} (Definition p. 79) a notion of transfer maps for Witt groups,
 depending on the choice of a non-zero $k$-linear map $s:E \rightarrow k$.\footnote{These maps are denoted by $s^*:W(E) \rightarrow W(k)$
 in \emph{loc. cit.}, but we will prefer the notation $s_*$ (for obvious reasons).
 Scharlau originally considered fields of characteristic not $2$ but the definition makes sense in arbitrary characteristic.
 Moreover, one can replace non-degenerate quadratic forms by non-degenerate symmetric bilinear forms.}

Using the differential trace map $\Tr^\omega_{E/k}:\omega_{E/k} \rightarrow k$
 (see \Cref{df:diff_trace}), it is possible to give a uniform definition, which does not
 depend on such a choice. Moreover, we will see that it coincides with the trace maps $\Tr^{MW,\alpha}_{E/k}$
 just defined in degree $q \leq 0$.

The definition is very similar to Scharlau's definition, but motivated by the form
 of MW-transfers (see \ref{num:MWtr-monogenic}), we use $\mathcal L$-valued inner product spaces (see \Cref{num:twisted_GW}).
  Given an arbitrary $\omega_{E/k}$-valued inner product space $\Phi:V \otimes_E V \rightarrow \omega_{E/k}$,
 one can consider the composite map 
$$
\Tr^\omega_{E/k} \circ \Phi:V \otimes_k V
 \rightarrow V \otimes_E V \xrightarrow{\ \Phi\ } \omega_{E/k}
 \xrightarrow{\Tr^\omega_{E/k}} k,
$$
which is again a non-degenerate symmetric bilinear $k$-form.
 It is compatible with isomorphisms and orthogonal sums,
 therefore it induces a well-defined map:
$$
\begin{array}{rcl}
\Tr^\omega_{E/k*}:\GW(E,\omega_{E/k}) & \rightarrow & \GW(k) \\
\lbrack \Phi \rbrack & \mapsto & \lbrack \Tr^\omega_{E/k} \circ \Phi \rbrack.
\end{array}
$$
As, by definition, the map $\Tr^\omega_{E/k}:\omega_{E/k} \rightarrow k$ is $k$-linear,
 one deduces that $\Tr^\omega_{E/k*}$ is $\GW(k)$-linear (recall the $\GW(k)$-action
 on both sides from \Cref{num:twisted_GW}). 
\end{num}
\begin{df}\label{df:GW-transfers}
Let $E/k$ be an arbitrary finite field extension.
 We call the $\GW(k)$-linear morphism $\Tr^\omega_{E/k*}:\GW(E,\omega_{E/k}) \rightarrow \GW(k)$ just defined the \emph{(differential) GW-transfer map}.

Modding out by the ideal $(h)$, one gets a (differential) W-transfer map 
 that we still denote: $\Tr^\omega_{E/k*}:\W(E,\omega_{E/k}) \rightarrow \W(k)$.
\end{df}

\begin{ex}\label{ex:GW_tr_separable}
If $E/k$ is separable, then $\omega_{E/k}=E$ and $\Tr^\omega_{E/k}$
 is just the usual trace map: $\Tr_{E/k}:E \rightarrow k$ (see \Cref{cor:Tr^omega_etale}).
 In particular, $\Tr^\omega_{E/k*}=\Tr_{E/k*}:\GW(E) \rightarrow \GW(k), \W(E) \rightarrow \W(k)$
 is the usual Scharlau transfer associated with the trace ``form'' $\Tr_{E/k}$.

In the inseparable case on the contrary, $\Tr_{E/k}=0$. The link
 with Scharlau traces will be explained in \Cref{rem:comp_Scharlau_trace}.
\end{ex}

\begin{ex}\label{ex:GW-diff&Tate_traces}
One can compute the GW-differential transfer maps more explicitly.

Consider a monogenic field extension $E/k$ of degree $d$,
 written as $E=k[\alpha]$.
 Let $f$ be the minimal polynomial of $\alpha$, so that for $I=(f)$, $E=k[t]/I$.
 Then, as explained in \Cref{ex:field_extension_complete_inter},
 $\omega_{E/k} \simeq (I/I^2)^\vee \otimes \omega_{k[t]/k}$.
 In particular, the invertible $E$-vector space $\omega_{E/k}$ admits an explicit base
 given by the element $\bar f^* \otimes dt$.
 In particular, any $\omega_{E/k}$-valued inner product space 
$$
\Phi:V \otimes_E V \rightarrow \omega_{E/k}
$$
can be written as $(x,y) \mapsto \phi(x,y).(\bar f^* \otimes dt)$
 where $\phi:V \otimes_E V \rightarrow E$ is an inner product space.

With this notation, \Cref{cor:compute_residues_SS} gives the following computation:
$$
\Tr^\omega_{E/k} \circ \big(\phi \otimes (\bar f^* \otimes dt)\big)=
 \tau_{E/k}^\alpha \circ \phi=:\tau_{E/k*}^\alpha(\phi),
$$
where we recall that the \emph{Tate trace map} $\tau_{E/k}^\alpha:E \rightarrow k$ is the $k$-linear form
 associated with the element $\alpha^{d-1}$ of the $k$-base $(1,\alpha,\hdots,\alpha^{d-1})$ of $E$.

When $E/k$ is non monogenic, one writes $E=k[\alpha_1,\hdots,\alpha_n]=k[t_1,\hdots,t_n]/(f)$,
 where $f=(f_1,...,f_n)$ for monic polynomials $f_i \in k[t_1,\hdots,t_n]$.
 Then combining the notation \Cref{ex:field_extension_complete_inter} and \Cref{prop:compute_residues_SS},
 one gets the formula:
$$
\Tr^\omega_{E/k} \circ \Big(\phi \otimes \big((\bar f_1 \wedge \hdots \wedge \bar f_n)^* \otimes (dt_1 \wedge \hdots \wedge dt_n)\big)\Big)=
 \tau_f \circ \phi,
$$
where $\tau_f$ is the Scheja-Storch trace map (\Cref{df:SS_trace}) associated with the presentation $f$ of $E/k$.
\end{ex}

\begin{rem}\textit{Comparison with Scharlau transfer}.\label{rem:comp_Scharlau_trace}
A particular case of Grothendieck duality (see \Cref{num:trivial_duality}) gives the following
 isomorphism:
$$
\omega_{E/k} \xrightarrow \sim \Hom_k(E,k), w \mapsto s_w:=\Tr^\omega_{E/k}(-.w).
$$
According to \Cref{ex:KMW_negative&GW}, one has a canonical identification
 $\GW(E,\omega_{E/k}) \simeq \GW(E) \otimes_{\ZZ[E^\times]} \ZZ[\omega_{E/k}^\times]$.
 With this notation, one can see that the above transfers incorporate all Scharlau's transfer maps
 at once: for $ \sigma \in \GW(E)$ and a non-zero $w \in \omega_{E/k}$, one gets:
$$
\Tr^\omega_{E/k*}(\sigma \otimes w)=s_{w*}(\sigma).
$$
\end{rem}

\begin{num}\label{num:basic_GW-transfer}
One easily derives from the previous definition the following basic properties of
 the differential GW-trace map, for a finite extension $\varphi:k \rightarrow E$ of degree $d$:
\begin{enumerate}[wide]
\item For any $\sigma \in \GW(E,\omega_{E/k})$, one has
 $\rk(\Tr^\omega_{E/k*}(\sigma))=d.\rk(\sigma)$.
\item If $L/E$ and $E/k$ are finite extensions, $\Tr^\omega_{E/k*} \circ \Tr^\omega_{L/E*}=\Tr^\omega_{L/k*}$
 where we have hidden the canonical isomorphism $\omega_{L/k} \simeq \omega_{L/E} \otimes_L \omega_{E/k}$
 (\Cref{eq:can_iso_can_sheaf}).\footnote{This follows from the functoriality of the differential trace map.}
\item For $\sigma \in \GW(E,\omega_{E/k})$, $\sigma' \in \GW(k)$, one has:
 $\Tr^\omega_{E/k*}(\sigma.\varphi_*(\sigma'))=\Tr^\omega_{E/k*}(\sigma).\sigma'$.\footnote{This is generically
 called the projection formula, and more specifically \emph{Frobenius reciprocity}
 in the theory of quadratic forms (\cite[p. 80]{SchRec}).}
\end{enumerate}
\end{num}

The main result for the GW-differential transfer map is the following \emph{quadratic reciprocity formula}
 which extends to the Milnor-Witt case a formula due to Scharlau first proved in \cite[Th. 4.1]{SchRec}, with a similar proof.
\begin{thm}\label{thm:KMW_GW-differential_Reciprocity}
Let $k$ be an arbitrary field.
 Then the following formula holds:
\begin{equation}\label{eg:KMW_GW-differential_Reciprocity}
\sum_{x \in (\PP^1_k)_{(0)}} \Tr^\omega_{\kappa_x/k*} \circ \partial_x=0,
\end{equation}
as maps $\KMW 1 \big(k(t),\omega_{k(t)/k}\big) \rightarrow \GW(k)$.
 Here, the map $$\partial_x:\KMW{1}\big(k(t),\omega_{k(t)/k}\big) \rightarrow \GW(\kappa_x,\omega_{\kappa_x/k})$$
 stands for the residue map associated with the discrete valuation on $k(t)$ associated to the closed point $x \in \PP^1_k$
 (see \Cref{num:homological_quad-cycle_P1}).
\end{thm}
\begin{proof}
The abelian group $\KMW 1 \big(k(t),\omega_{k(t)/k}\big)$ is generated by elements
 of the form $[f] \otimes dt$ where $f \in k(t)^\times$ is a rational function on $\PP^1_k$.
 So we need only to check the vanishing on these particular elements.

Consider the prime decomposition of $f$:
\begin{equation}\label{eq:KMW_GW-differential_Reciprocity1}
f=u.\pi_1^{m_1}\hdots \pi_r^{m_r},
\end{equation}
where $u \in k^\times$, $m_i \in \ZZ$ and $\pi_i$ is an irreducible monic polynomial in $k[t]$.
 Each polynomial $\pi_i$ corresponds to a closed point $x_i$ in $\AA^1_k \subset \PP^1_k$,
 with residue field $\kappa_i=\kappa(x_i)=k[t]/(\pi_i)$. With this notation,
 we will write $\alpha_i \in \kappa_i$ for the obvious generator of $\kappa_i/k$
 (\emph{i.e.} corresponding to $t$).

We first remark that, computing the quadratic order of vanishing of $f$ at $\infty$ using the uniformizer
 $1/t$, we find in $\GW(k)$:\footnote{Note that $\omega_{\kappa_\infty/k}=k$ so that
 we identify $\GW(\kappa_\infty,\omega_{\kappa_\infty/k})$ with $\GW(k)$. With this identification,
 the GW-differential trace map $\Tr_{\kappa_\infty/k}^\omega$ is just the identity.}
$$
\partial_\infty([f] \otimes dt)=-d_\epsilon\tw u \in \GW(k)
$$
where $d=\deg_t(f)=\sum_i m_i.\deg(\pi_i)$.

Let us write $f_i=\prod_{j \neq i} \pi_j^{m_j}$, so that $f=uf_i.\pi_i^{m_i}$.
 Applying formula \eqref{eq:quadratic_order} (with the added twist $dt$), one gets:
$$
\sum_{x \in (\PP^1_k)_{(0)}} \Tr^\omega_{\kappa_x/k*} \circ \partial_x([f] \otimes dt)=
 \sum_{i=1}^r (m_i)_\epsilon\tw u\Tr^\omega_{\kappa_i/k*}\big(\tw{f_i(\alpha_i)} \otimes dt \otimes \bar \pi_i^*\big)
 -d_\epsilon\tw u.
$$
Let us denote by $(*)$ the right-hand side,
 so that we need to show that $(*)$ is $0$ in $\GW(k)$.
 The (virtual) rank of $(*)$ is
$$
\sum_{i=1}^r m_i.\deg(\pi_i)-d
$$
which is obviously zero --- according to relation \eqref{eq:KMW_GW-differential_Reciprocity1}.

Therefore, one needs only to show that the class of $(*)$ is zero in $\W(k)$.
 Obviously, one can assume that $u=1$. Moreover, as $n_\epsilon=0$ for $n$ even in $\W(k)$,
 one can assume that $m_i=1$ for all $i$.

Let us consider the monogenic $k$-algebra $A=k[t]/(f)$, and write $\alpha$ its generator.
 Recall that $A$ is a finite $k$-vector space with basis $\mathcal B=(1,\alpha,\hdots,\alpha^{d-1})$.
 The Chinese remainder lemma gives an isomorphism of $k$-algebras:
$$
\Theta:A \xrightarrow{\ \sim\ } \prod_{i=1}^r \kappa_i, g \mapsto \big(g(\alpha_i)\big)_{1 \leq i \leq r}.
$$
Under the induced isomorphism of canonical modules, the element $\bar f^* \otimes dt$
 maps on the $i$-th factor to $f_i(\alpha_i)^{-1}\bar \pi_i^* \otimes dt$.
 Applying \Cref{rem:diff_trace}, and using
 $\tw{f_i(\alpha_i)^{-1}}=\tw{f_i(\alpha_i)}$, one deduces that
$$
\sum_{i=1}^r \Tr^\omega_{\kappa_i/k*}\big(\tw{f_i(\alpha_i)} \otimes dt \otimes \bar \pi_i^*\big)
=\Tr^\omega_{A/k*}\big(\tw 1 \otimes dt \otimes \bar f^*\big).
$$
We are now reduced to show the following equality in $\W(k)$:
\begin{equation}\label{eq:KMW_GW-differential_Reciprocity2}
\Tr^\omega_{A/k*}\big(\tw 1 \otimes (dt \otimes \bar f^*)\big)=d_\epsilon.
\end{equation}
One can apply \Cref{cor:compute_residues_SS} (see \Cref{ex:GW-diff&Tate_traces}) to compute
 the left-hand side: if one denotes by $\tau_{A/k}^\alpha:A \rightarrow k$
 the Tate trace map associated with $A/k$ and its generator $\alpha$,
 --- that is the linear form associated with $\alpha^{d-1}$ in the basis $\mathcal B$ ---
 this inner product on the $k$-vector space $A$ is given by the formula:
$$
A \otimes_k A \rightarrow k, g \otimes g' \mapsto \tau_{A/k}^\alpha(gg').
$$
One easily computes the form of the symmetric $(d \times d)$-matrix of this symmetric bilinear form in the basis $\mathcal B$ as:
$$
\left(\raisebox{0.5\depth}{\xymatrix@=10pt{
\ar@{}|0[rd] && 1\ar@{-}[lldd] \\
 && \\
1 && \ar@{}|{*}[lu]
}}
\right).
$$
But the class of the corresponding inner product space in $\GW(k)$ is $d_\epsilon=\tw{1,-1,\hdots}$
 as it has a totally isotropic subspace of rank $n$ spanned by $(1,\hdots,\alpha^{n-1})$
 if $d=2n$ or $d=2n+1$, and its determinant is $(-1)^{d(d-1)/2}$.
 This proves \eqref{eq:KMW_GW-differential_Reciprocity2}.
\end{proof}

\begin{rem}\label{rem:degree_diff_tr-mono}
It is interesting to note that the end of the previous proof also shows the following degree formula,
 for any monogenic finite extension $E=k[\alpha]$ of degree $d$, with minimal polynomial $f$:
$$
\Tr^\omega_{E/k*}\big(\tw 1 \otimes (dt \otimes \bar f^*)\big)=d_\epsilon.
$$
\end{rem}

\begin{rem}
Multiplying by $\eta$, and looking modulo the hyperbolic form $h$
 (granted the $\GW(k)$-linearity of each involved operator),
 the equation \eqref{eg:KMW_GW-differential_Reciprocity} gives a twisted version of Scharlau's quadratic
 reciprocity formula: for any class $\sigma \in \W(k(t),\omega_{k(t)/k})$
 of a $\omega_{k(t)/k}$-valued  inner product space over $k(t)$, one has:
$$
\sum_{x \in (\PP^1_k)_{(0)}} \Tr^\omega_{\kappa_x/k*}(\partial_x(\sigma))=0.
$$
In fact, using \Cref{ex:GW-diff&Tate_traces} and applying this equality to $\sigma=\sigma_0 \otimes dt$,
 one gets back precisely Scharlau's formula (see also \cite[\textsection 2, Satz]{GHKS}
 in the characteristic not $2$ case).
\end{rem}

The main application of the previous theorem,
 taking into account the uniqueness statement of \Cref{prop:BT} is the following comparison result
 between the two transfer maps we have introduced.
\begin{cor}\label{cor:trMW&omega_monogen}
Let $E/k$ be a monogenic finite extension field.
 Then for any generator $\alpha$ of $E/k$, 
 and any $q<0$, one has commutative diagrams
$$
\xymatrix@R=10pt@C=32pt{
\KMW 0(E,\omega_{E/k})\ar^-{\Tr^{MW,\alpha}_{E/k}}[r] & \KMW 0 (k)
& \KMW q (E,\omega_{E/k})\ar^-{\Tr^{MW,\alpha}_{E/k}}[r] & \KMW q (k) \\
\GW(E,\omega_{E/k})\ar_-{\Tr^\omega_{E/k*}}[r]\ar^\sim[u]  & \GW(k)\ar_{\sim}[u]
 & \W(E,\omega_{E/k})\ar_-{\Tr^\omega_{E/k*}}[r]\ar^\sim[u]  & \W(k)\ar_{\sim}[u]
}
$$
where the vertical isomorphisms come from \Cref{ex:KMW_negative&GW}.
\end{cor}

\subsection{General trace maps}\label{sec:general-traces}

\begin{num}
Let $E/k$ be a finite extension  with canonical module $\omega_{E/k}$.

We have already seen (\Cref{df:GW-transfers}) how the Grothendieck differential trace map
 induces a transfer map on twisted Grothendieck-Witt and Witt groups.
 We now show how to extend these transfers to Milnor-Witt K-theory
 using Morel's fundamental square from \Cref{cor:fundamental_square_KMW}.
 We first need a lemma.
\end{num}
\begin{lm}
Consider the above notation.
 For any integer $n \in \ZZ$, one has:
$$
\Tr^\omega_{E/k*}\big(\I n(E,\omega_{E/k})\big) \subset \I n(k)
$$
where we have used the notation of \Cref{num:twisted_I} for $\I n$
 and the transfer map on Witt groups was defined in \Cref{df:GW-transfers}.
\end{lm}
Using \Cref{rem:comp_Scharlau_trace}, the lemma follows from \cite[Satz 3.3]{Arason}.
 At this point, one can easily deduce it from our earlier computations so we give
 a proof for completeness.
\begin{proof}
The cases $n \leq 0$ are trivial. We note the case $n=1$ is easy
 (use \Cref{num:basic_GW-transfer}(1)).
 For the other cases, using the functoriality of GW-transfers \Cref{num:basic_GW-transfer}(2),
 one reduces to the case where $E/k$ is monogenic,
 with say a fixed generator $\alpha$.
 This case now follows from \Cref{cor:trMW&omega_monogen},
 \Cref{thm:KW&graded_I} and the fact $\Tr^{MW,\alpha}_{E/k}$
 (defined in \Cref{num:MWtr-monogenic}) commutes with multiplication
 by $\eta$ and $h$.
\end{proof}

In particular we get well-defined transfer maps on the algebra functor $\I*$.
 As an intermediate step, we show that these transfers are compatible
 with the monogenic transfers obtained so far on the Milnor-Witt K-theory
 functor (\Cref{num:MWtr-monogenic}).
\begin{lm}\label{lm:mono_transf_MW&I}
Let $E/k$ be a monogenic finite extension, with a generator $\alpha \in E$.
 Then the following diagram is commutative:
$$
\xymatrix@R=18pt@C=40pt{
\KMW*(E,\omega_{E/k}) \ar_{\mu'_E}[d]\ar^-{\Tr^{MW,\alpha}_{E/k}}[r]
 & \KMW*(k)\ar^{\mu'_k}[d] \\
\I*(E,\omega_{E/k})\ar^-{\Tr^\omega_{E/k*}}[r] & \I*(k).
}
$$
\end{lm}
Given the previous lemma, and the construction of the morphism $\mu'$ (see \Cref{num:Morel-Milnor_map})
 this statement reduces to \Cref{cor:trMW&omega_monogen}.

\begin{num}
According to \Cref{prop:twisted_fdl_square_KMW},
 $\KMW n (E,\omega_{E/k})$ can be identified with the abelian group made
 of pairs $(\sigma,\tau) \in \I n(E,\omega_{E/k}) \times \KM n (E)$ such that
 $\pi(\sigma)=\mu(\tau)$. The following lemma is the last step needed to define
 the transfer map associated with $E/k$ on Milnor-Witt K-theory.
\end{num}
\begin{lm}
Consider the above notation. Then one has the following equality in $\gI n(k)$:
$$
\pi\big(\Tr^\omega_{E/k*}(\sigma)\big)
 = \mu\big(\Tr^M_{E/k}(\tau)\big).
$$
\end{lm}
\begin{proof}
By functoriality of the differential trace map (\Cref{rem:diff_trace_composition})
 and of Kato's transfer map on Milnor K-theory,
 one reduces to the case of finite monogenic extensions
 $E=k[\alpha]$.
 Then the result follows from the existence of the trace map $\Tr^{MW,\alpha}_{E/k}$,
 and its compatibility with both Kato's transfer (\Cref{lm:mono_transf_F&H})
 and the differential transfer on $\I*$ (\Cref{lm:mono_transf_MW&I}).
\end{proof}

We finally obtain the main definition of this section.
\begin{df}\label{df:MW-transfers}
Let $E/k$ be a finite extension with canonical module $\omega_{E/k}$.
 One defines the transfer map on Milnor-Witt K-theory by the following formula:
$$
\xymatrix@C=60pt@R=-4pt{
\KMW*(E,\omega_{E/k})\ar^{\Tr^{MW}_{E/k}}[r]\ar_{(\mu'_E,\F_E)}^\sim[ddddddddd]
 & \KMW*(k)\ar^{(\mu'_k,\F_k)}_\sim[ddddddddd] \\
& \\& \\& \\& \\& \\& \\& \\& \\
\I*(E,\omega_{E/k}) \times_{\gI*(E)} \KM*(E) & \I*(k) \times_{\gI*(k)} \KM*(k) \\
(\sigma,\tau)\ar[r] & \big(\Tr^\omega_{E/k*}(\sigma),\Tr^M_{E/k}(\tau)\big)
}
$$
well-defined according to the previous lemma. The vertical isomorphisms
 come from \Cref{prop:twisted_fdl_square_KMW}.

As in the end of \Cref{num:MWtr-monogenic}, one also defines for an invertible $k$-vector space $\cL$,
 an $\cL$-twisted transfer:
$$
\Tr^{MW}_{E/k}:\KMW q(E,\omega_{E/k} \otimes \cL_E) \rightarrow \KMW q(k,\cL).
$$
\end{df}
When we denote by $\varphi:k \rightarrow E$ the structural map of the extension $E/k$,
 it is customary to use the notation $\varphi^*=\Tr^{MW}_{E/k}$.
 We also call it occasionally the trace map.\footnote{Other terminologies that we prefer
 to avoid are the norm (Kato) and corestriction (Rost).}

\begin{rem}\label{rem:MW-transfer_ppties}
This trace map has all the good properties of its analog on Milnor K-theory.
 It is compatible with composition (as this is the case for $\Tr^\omega_{E/k}$ and $\Tr^M_{E/k}$).
 It satisfies the so-called projection formula:
 for $(\sigma,\beta) \in \KMW*(E,\omega_{E/k} \otimes \cL_E) \times \KMW*(k,\cM)$, one has in $\KMW*(k,\cL \otimes_k \cM)$:
$$
\varphi^*(\sigma.\varphi_*(\beta))=\varphi^*(\sigma).\beta.
$$
This follows from \Cref{num:basic_GW-transfer}(3) and the corresponding formula for Milnor K-theory
 (see \cite[formula (5), p. 378]{BT}).

Finally, we note that from a geometric point of view,
 if one denotes by $f:\spec E \rightarrow \spec k$
 the induced morphism, one can also denote: $\varphi^*=f_*$ and $\varphi_*=f^*$. In this way,
 the previous formula looks like the "classical" projection formula (for Chow groups, cohomology,...)
\end{rem}

As an immediate corollary of the previous definition, we obtain the following
 explicit description of transfers on Milnor-Witt K-theory.
\begin{cor}\label{cor:MW_tr_general_ppty}
Let $E/k$ be a finite extension with canonical module $\omega_{E/k}$,
 and $n$ an integer.
\begin{enumerate}[wide]
\item If $n=0$ (resp. $n<0$) then through the identification
$$
\KMW0(E,\omega_{E/k})=\GW(E,\omega_{E/k})
 \text{ (resp. } \KMW n(E,\omega_{E/k})=\W(E,\omega_{E/k})\text )
$$
 of \Cref{prop:KMW&GW}, one has $\Tr^{MW}_{E/k}=\Tr^\omega_{E/k*}$
 where $\Tr^\omega_{E/k}$ is the differential trace map
 (see \Cref{df:GW-transfers}).
\item If $n>0$, any element $\sigma \in \KMW n(E,\omega_{E/k})$
 can be written as a sum of elements of the form
 $([\phi] \otimes w,\sigma')$ where:
\begin{itemize}
\item  $(V,\phi:V \otimes_E V \rightarrow E)$ is an inner product space over 
 $E$, $[\phi]$ is its class in $\W(E)$ and $[\phi] \in \I n(E)$,
\item $w \in \omega_{E/k}$ is a non-zero element,
\item $\sigma'=\{u_1,\hdots,u_n\}$ is a symbol in $\KM n(E)$,
 for certain units $u_i \in E^\times$.
\end{itemize}
For such an element, one has:
$$
\Tr^{MW}_{E/k}([\phi] \otimes w,\sigma')
 =\big([\Tr^\omega_{E/k} \circ (\phi.w)],\Tr^M_{E/k}(\sigma')\big)
$$
where $\Tr^\omega_{E/k} \circ (\phi.w)$ is the class in $\W(k)$
 (and in fact $\I n(k)$) 
 of the inner product space on $V$ over $k$ with bilinear form
$$
(x,y) \mapsto \Tr^\omega_{E/k}\big(\phi(x,y).w\big),
$$
and $\Tr^M_{E/k}$ is the transfer map on Milnor K-theory.
\end{enumerate} 
\end{cor}

\begin{ex}\label{ex:compute_MW-trace}
In general, we refer the reader to \Cref{ex:GW-diff&Tate_traces}
 and \Cref{rem:comp_Scharlau_trace} for the computation
 of the differential trace map $\Tr^\omega_{E/k*}$ on 
 the Grothendieck-Witt or Witt group.
 One can single out the following explicit computations.
\begin{enumerate}
\item If $E/k$ is separable, then $\omega_{E/k}=E$
 and $\Tr^\omega_{E/k}=\Tr_{E/k}$ is the usual trace map
 (\Cref{cor:Tr^omega_etale}).
 In particular, $\KMW 0 (E,\omega_{E/k}) \simeq \GW(E)$
 and for any unit $u \in E^\times$,
 $\Tr^{MW}(\tw u)=[\Tr_{E/k}(u.-)]$, the $\GW$-class of the 
 \emph{scaled trace form}, $(x,y) \mapsto \Tr_{E/k}(uxy)$.
\item Let $E/k$ be a finite monogenic field extension of degree $d$, with generator $\alpha$. 

According to \Cref{rem:comp_Scharlau_trace}, one has an isomorphism
 $$\omega_{E/k} \simeq \Hom_k(E,k), w \mapsto s_w.$$
 In particular, there exists
 a unique non-zero form $w \in \omega_{E/k}$ such that $s_w$ is the $k$-linear form
 which maps $\alpha^0$ to $1$ and $\alpha^i$ to $0$ for $0<i<d$.

Then for any unit $u \in E^\times$, and for the particular choice of $w$ made above,
 one has:
$$
\Tr^{MW}_{E/k}([u] \otimes w)=[N_{E/k}(u)] \in \KMW 1 (k)
$$
where $N_{E/k}:E^\times \rightarrow k^\times$ is the usual norm of the finite extension $E/k$.
 According to the previous corollary, this follows from \cite[VII, Cor. 2.4]{Lam} for the Witt part\footnote{The
 computation of \emph{loc. cit.} extends in characteristic $2$ as well}
 and \cite[I.\textsection 5, Th. 5.6]{BT} for the Milnor part.

This formula generalizes to an arbitrary finite extension provided one chooses the correct
 differential form $w$.
\item Let $k$ be a field of characteristic $p>0$, let $q=p^n$ with $n>0$, and let
 $a \in k$ be an element which is not a $p$-th power. Put
 $E=k[\sqrt[q] a]=k[t]/(t^q-a)$ and $\alpha=\sqrt[q] a \in E$.
 There exists a canonical non-zero element $w=dt \otimes (\overline{t^q-a})^*$
 of $\omega_{E/k}$ (see \Cref{ex:field_extension_complete_inter} with $n=1$).

Then for any unit $u \in E^\times$, and again for the particular choice of $w$ made above, one has:
$$
\Tr^{MW}(\tw u \otimes w)=[\tau^\alpha_{E/k}(u.-)]
$$
where $\tau^\alpha_{E/k}$ is the \emph{Tate trace map} associated with the $q$-th root $\alpha$
 (see \Cref{rem:Scheja-Storch_trace_monogeneous}), and $[\tau^\alpha_{E/k}(u.-)]$
 is the $\GW$-class of the scaled (Tate) trace form of the $k$-vector space $E$:
$$
E \otimes_k E \rightarrow k, x \otimes y \mapsto \tau^\alpha_{E/k}(uxy).
$$
Note in particular that one gets the following degree formula:
$$
\Tr^{MW}(\tw 1 \otimes w)=[\tau^\alpha_{E/k}]
$$
\end{enumerate}
\end{ex}

In comparison with the last example, one gets the following more general \emph{degree formula}
 in Milnor-Witt K-theory.
\begin{cor}\label{cor:degree_formula}
Let $E/k$ be a finite extension of degree $d$.
 We consider a minimal family of
 generators $(\alpha_1,...,\alpha_n)$ and the associated presentation 
$E=k[t_1,\hdots,t_n]/(f_1,\hdots,f_n)$
 as in \Cref{ex:field_extension_complete_inter}.
 Let $w=(\bar f_1 \wedge \hdots \wedge \bar f_n)^* \otimes (dt_1 \wedge \hdots \wedge dt_n)$
 be the canonical element of $\omega_{E/k}$ as in \emph{loc. cit.}

Then one has in $\KMW 0(k)=\GW(k)$:
$$
\Tr^{MW}_{E/k}(\tw 1 \otimes w)=d_\epsilon
$$
where we have used the notation of \Cref{num:quad_mult}.
\end{cor}
\begin{proof}
By multiplicativity of $d_\epsilon$ (\Cref{num:quad_mult}),
 and the functoriality of the MW-trace map,
 one reduces to the monogenic case. Then, it follows from \Cref{rem:degree_diff_tr-mono}.
\end{proof}

\begin{rem}
In general, any non-zero element $w' \in \omega_{E/k}$ can be written as $w'=u.w$
 for a unit $u \in E^\times$.
 One should be careful however that if one replaces $w$ in the above corollary by $w'$,
 this completely changes the above result.
 For example, in the case of \Cref{ex:compute_MW-trace}(3),
 one gets
$$
\Tr^{MW}(\tw 1 \otimes w')=[\tau^\alpha_{E/k}(u.-)].
$$
\end{rem}

\begin{num}
Finally, we want to compare the previous definition of transfers on Milnor-Witt K-theory
 with the original one due to Morel for finitely generated extensions of some perfect field:
 \cite[\textsection 4.2]{Mor}.

Recall the construction of Morel, for a finite extension $E/k$.\footnote{Note that contrary to the previously known constructions, 
 we do not have to assume that $E$ and $k$ are finitely generated extensions over some perfect base field.}
 We fix a finite generating family $\underline \alpha=(\alpha_1,\hdots,\alpha_n)$ of $E/k$,
 to which we associate a tower of finite monogenic extensions $\kappa_i=k[\alpha_1,\hdots,\alpha_i]$:
$$
k \subset \kappa_1 \subset \hdots \subset \kappa_n=E.
$$
Then we can define the following composite map, denoted by $\Tr^{MW,\underline \alpha}_{E/k}$:
\begin{align*}
\KMW q(E,\omega_{E/k})& \simeq \KMW q(E,\omega_{E/\kappa_{n-1}} \otimes \omega_{\kappa_{n-1}/k}|_E)
 \xrightarrow{\Tr^{MW,\alpha_n}_{E/\kappa_{n-1}}} \KMW q(\kappa_{n-1},\omega_{\kappa_{n-1}/k}) \\
& \simeq \KMW q(\kappa_{n-1},\omega_{\kappa_{n-1}/\kappa_{n-2}} \otimes \omega_{\kappa_{n-2}/k}|_{\kappa_{n-1}})
 \xrightarrow{\Tr^{MW,\alpha_{n-1}}_{\kappa_{n-1}/\kappa_{n-2}}} \hdots \\
& \hdots \KMW q(\kappa_1,\omega_{\kappa_1/k}) \xrightarrow{\Tr^{MW,\alpha_1}_{\kappa_{1}/k}} \KMW q(k)
\end{align*}
where the morphism $\Tr^{MW,\alpha_{i}}_{\kappa_{i}/\kappa_{i-1}}$ is the $(\omega_{\kappa_{i-1}/k})$-twisted MW-transfer
 associated with $(\kappa_i/\kappa_{i-1},\alpha_i)$, as defined in \Cref{num:MWtr-monogenic}.

The main result of \cite[\textsection 4.2]{Mor} (see Th. 4.27) is that this composite map, at least for finitely generated
 extensions of some perfect base field, is independent of the chosen family of generators. This result
 has also been proved later in \cite{FeldTohoku} by direct transport of the proof of Kato (again under the same
 assumptions). Actually, given the method we have chosen, we get another proof of this theorem
 (without any restriction on the fields considered).
\end{num}
\begin{prop}\label{prop:comparison_MW-transfers}
Consider the above notation. Then one has an equality:
$$
\Tr^{MW}_{E/k}=\Tr^{MW,\underline \alpha}_{E/k}
$$
where the left-hand side was defined in \Cref{df:MW-transfers}.
\end{prop}
In  particular, the computations given above apply to the already known
 (geometric) transfer map on Milnor-Witt K-theory.
\begin{proof}
As the transfers of \Cref{df:MW-transfers} are compatible with composition
 (\Cref{rem:MW-transfer_ppties}), one reduces to the monogenic case.
 This is then a consequence (already observed) of the definition, 
 and lemmas \ref{lm:mono_transf_F&H}, \ref{lm:mono_transf_MW&I}. 
\end{proof}

\section{Functoriality of Milnor-Witt K-theory} \label{sec:MW-mod}

We now turn to the last part of this work, where we gather the functorial properties of Milnor–Witt K-theory established in the preceding sections.
 These properties fit into the axiomatic framework of Milnor–Witt premodule theory developed by Feld in \cite{FeldMWmod},
 with several significant extensions.
 First, we work over all fields, not merely finitely generated ones over a fixed base field.
 Second, we make explicit the canonical isomorphisms arising in the twists of Milnor–Witt K-groups, which are usually left implicit.
 Finally, we formulate and prove refined versions of Feld’s structural formulas,
 and we establish all relations among the four functorialities in full generality.

Although Feld's axioms may appear somewhat intricate at first sight, a useful guiding principle is to regard Milnor–Witt 
 K-theory as a twisted cohomology theory defined on the category of integral $0$-dimensional schemes --- that is, on fields.
 Such axiomatic descriptions of functors on fields were pioneered by Rost, with his theory of cycle premodules \cite{Rost},
 and later adapted by Schmid \cite{Schmid} to the particular case of Witt groups.

\subsection{Basic maps}\label{sec:functoriality}
For any triple $(E,\cL,n)$, where $E$ is a field, $\cL$ an invertible $E$-vector space,
 and $n \in \ZZ$ an integer, one has an abelian group $\KMW n (E,\cL)$.
 It is equipped with the following basic maps, the same as in \cite[Def. 3.1]{FeldMWmod},
 except that we do not restrict to finitely generated extension fields of some base field:

\begin{enumerate}[label=(\textbf{D\arabic*}), leftmargin=*]
\item (see \Cref{num:twKMW_basic}(2)): Given any morphism
 $\varphi:E \rightarrow F$ of fields, one has a morphism of abelian groups:
$$
\varphi_*:\KMW n (E,\cL) \rightarrow \KMW n (F,\cL \otimes_E F).
$$
\item (see \Cref{df:MW-transfers}):
 Given a finite morphism $\psi:E \rightarrow F$ of fields,
 one has a \emph{transfer map}:
$$
\psi^*=\Tr^{MW}_{F/E}:\KMW n (F,\omega_{F/E} \otimes_E \cL) \rightarrow \KMW n (E,\cL)
$$
where $\omega_{F/E}$ is the canonical invertible $F$-vector space
 associated with the finite field extension $F/E$ (see \Cref{df:can_sheaf}).
\item (see \Cref{num:twKMW_basic}(1)): It has a structure of a bigraded algebra.
 Given triples $(E,\cL,n)$ and $(E,\cM,m)$, one has a \emph{product}:
$$
\KMW n (E,\cL) \otimes \KMW m (E,\cM) \rightarrow \KMW {n+m} (E,\cL \otimes_E \cM).
$$
In other words, $\KMW*(E,*)$ is a bigraded ring,
 graded with respect to $\ZZ$ and to the set of isomorphism classes
 of invertible $E$-vector spaces.
\item (see \Cref{thm:residue} and \Cref{num:residue}):
 Let $(E,v)$ be a discretely valued field
 with ring of integers $\cO_v$, $\cL$ be an invertible $\cO_v$-module
 and $n \in \ZZ$ an integer. We let $\kappa_v$ be the residue field,
 $\cL_E=\cL \otimes_{\cO_v} E$.
 One has a morphism of abelian groups, called the \emph{residue map}:
$$
\partial_v:\KMW n (E,\cL_E) \rightarrow \KMW{n-1} (\kappa_v,\omega_v \otimes_{\cO_v} \cL)
$$
where $\omega_v=(\mathcal M_v/\mathcal M_v^2)^\vee$
 --- the normal sheaf of $\Spec(\kappa_v) \rightarrow \Spec(\cO_v)$.
\end{enumerate}
There is a further functorial property hidden in the axioms of \cite{FeldMWmod}
 that we now state explicitly:
\begin{enumerate}[label=(\textbf{D\arabic*+}), leftmargin=*]
\item (see \Cref{num:twKMW_basic}(3)) Given an isomorphism
 $\theta:\cL \rightarrow \cM$
 of $E$-vector spaces, one has an isomorphism of abelian groups:
$$
\theta_*:\KMW n (E,\cL) \rightarrow \KMW n (E,\cM).
$$
\end{enumerate}

\begin{rem}
Using \Cref{rem:twisted_functoriality} and the category of twisted fields defined therein,
 one can unify the functoriality data of (D1) and (D1+).
 In fact, one can also unify (D1) and (D2) by using the appropriate category of correspondences.
 We leave the details to the interested reader.
\end{rem}

\begin{rem}\label{num:MW-premod_variants_KMW}
In fact, we have seen other examples of theories equipped with the same basic maps:
\begin{enumerate}[wide]
\item The periodized Witt ring $\W[t,t^{-1}]$, which therefore becomes a $\ZZ$-graded algebra:
 this comes from the isomorphism $\phi$ of \Cref{cor:KMW_eta-inversed}
 and the fact multiplication by $\eta$ is compatible with data D* on $\KMW*$.
 In particular, the canonical map
$$
\KMW* \rightarrow \W[t,t^{-1}]
$$
is compatible with data D*.
\item The graded algebra $\I*$ associated
 with the fundamental ideal $\I{}$ of the Witt group (see \Cref{rem:graded_I}):
 a quick way of seeing that is to use the isomorphism
 $\psi$ of \Cref{thm:KW&graded_I} and the fact multiplication by $h$ is compatible
 with all data D* on $\KMW*$.

In particular, 
 the inclusion $\I* \subset \W[t,t^{-1}]$ and
 the canonical map $\mu':\KMW* \rightarrow \I*$ obtained in \Cref{cor:fundamental_square_KMW}
 are compatible with all data D*.
\item The Milnor K-theory $\KM*$: a first way of seeing that is the isomorphism
 from~\Cref{num:KMW&KM}, and again the fact that multiplication by $\eta$
 is compatible with data D* on $\KMW*$. 

On the other hand, recall that in this example,
 one has a canonical isomorphism $\KM*(E,\cL) \simeq \KM*(E)$
 (\Cref{ex:KM_trivial_tw}). Or in other words, the data (D1+) is trivial for Milnor K-theory:
 for any automorphism $\theta:\cL \rightarrow \cL$ of invertible $E$-vector spaces,
 the map $\theta_*:\KM*(E,\cL) \rightarrow \KM*(E,\cL)$ is equal to the identity.
 We will say that $\KM*$ is \emph{orientable}.

In this case, the above functoriality (D*) actually corresponds to the functoriality
 of Rost cycle premodules \cite[Def. 1.1]{Rost}. Moreover, the forgetful map
$$
\forget:\KMW* \rightarrow \KM*
$$
of \Cref{df:KMW&KM} is compatible with data (D*): (D1) and (D1+) are obvious,
 D2 comes from \Cref{df:MW-transfers},
 D3 comes from the fact it is a morphism of rings,
 and D4 was observed in \Cref{rem:KMW&KM_residues}.
 The hyperbolic map
$$
\hyper:\KM* \rightarrow \KMW*
$$
is compatible with (D1), (D1+), (D2) and (D4), but not with (D3),
 as it is not a morphism of rings.
\item The graded algebra $\gI *$ (see again \Cref{rem:graded_I}):
 a quick way of seeing that is to use the Milnor conjecture
 \Cref{thm:MilnorConj} and to use the preceding point.
 One can also use the fact $\I *$ is a subalgebra of $\W[t,t^{-1}]$
 and therefore, all data D* descend to the quotient $\gI *$, as $t$ 
 on the right-hand side is compatible with all data D*.

Note that $\gI*$ is also orientable as in the previous point.
 According to the previous remarks, one sees that the canonical maps
 (\Cref{cor:fundamental_square_KMW})
$$
\KM* \xrightarrow \mu \gI*, \quad \I * \xrightarrow \pi \gI*
$$
are compatible with the data D*. 
\end{enumerate}
For more background on these different theories and their relations,
 we refer the reader to \cite[\textsection 3.4.1]{DFJ}.
\end{rem}

\subsection{Main properties}\label{sec:rules}
We now state the properties of the maps (D*) constructed in \Cref{sec:functoriality}.
 Apart from the fact that they are valid for any field, they correspond precisely
 to those of \cite[Def. 3.1]{FeldMWmod}, and thus we follow the same numbering.
 We will state and prove three sets of properties.

\begin{num}
We begin by stating a first list of such rules, involving only (D1) and (D2).
\begin{enumerate}[label=(\textbf{R1\alph*}), leftmargin=2.5em]
\item $(\psi \circ \varphi)_*=\psi_* \varphi_*$ for composable morphisms of fields $\varphi$, $\psi$.
\item $(\Psi \circ \Phi)^*=\Phi^* \Psi^*$ for composable finite morphisms of fields $\Phi$, $\Psi$.
\item Consider a morphism (resp. finite morphism) of fields $\psi:E \rightarrow L$ (resp. $\Phi:E \rightarrow F$).
 Assume that $\psi$ or $\Phi$ is separable. Then:
$$
\psi_*\Phi^*=\sum_{x \subset F \otimes_E L} \Phi_x^*\psi_{x*}
$$
where $x$ runs over prime ideals of $F \otimes_E L$,
 with residue field $\kappa_x=(F \otimes_E L)/x$,
 $\Phi_x:L \rightarrow \kappa_x$
 and $\psi_x:F \rightarrow \kappa_x$ are the induced maps,
 and we have used the canonical base-change isomorphism
 $\omega_{F/E}\otimes_F\kappa_x \simeq \omega_{\kappa_x/L}$.
\end{enumerate}
Each property follows from the preceding sections. Here is a detailed justification.
 Property (R1a) is clear from \Cref{df:KMW}, while (R1b) follows from  \Cref{rem:MW-transfer_ppties}.
 To prove (R1c), we use \Cref{df:MW-transfers}. This reduces
 to proving the corresponding formulas for Milnor K-theory
 and for the differential trace map, respectively.
 The case of Milnor K-theory follows from \cite[(5.8)]{BT}
 and that of the differential trace map from \cite[3.4.1]{ConradDual},
 or can be derived from the explicit computation in \Cref{prop:compute_residues_SS}.
\end{num}

\begin{rem}
We will strengthen (R1c) in \Cref{thm:R1c+}, following an idea of \cite{FeldTohoku}.
 Note however that this formula is enough to develop
 the theory of Chow-Witt groups (especially pullbacks).
\end{rem}

\begin{num}
Let us consider again a morphism $\varphi:E \rightarrow F$ of fields,
 and a finite morphism $\Phi:E \rightarrow F$ of fields.
 In addition, one considers $\sigma, \sigma', \beta$ elements of the Milnor-Witt K-group so that the next formulas
 make sense. We now state the properties of Milnor-Witt K-theory which involve products and data (D1), (D2):
\begin{enumerate}[label=(\textbf{R2\alph*}), leftmargin=2.5em]
\item $\varphi_*(\sigma.\sigma')=\varphi_*(\sigma).\varphi_*(\sigma')$
\item $\Phi^*\big(\Phi_*(\sigma) \ .\ \beta\big)=\sigma\ .\ \Phi^*(\beta)$
\item $\Phi^*\big(\sigma \ .\ \Phi_*(\beta)\big)=\Phi^*(\sigma)\ .\ \beta$
\end{enumerate}
Given the definition of the Milnor-Witt K-theory ring by generators and relations (\Cref{df:KMW}),
 formula (R2a) is clear.
 (R2b) and (R2c) are equivalent by the $\epsilon$-commutativity of the Milnor-Witt ring (\Cref{prop:KMW-epsilon-commut}).
 Then (R2c) is proved in \Cref{rem:MW-transfer_ppties}.
\end{num}

\begin{num}
Let us finally gather the elementary properties that involve the residue map (D4).
 One considers discretely valued fields $(E,v)$, $(F,w)$,
 with corresponding valuation rings $\cO_v$, $\cO_w$ and maximal ideals $\cM_v$, $\cM_w$.
 In (R3a), (R3c) and (R3d), we consider in addition a morphism $\varphi:E \rightarrow F$ of fields.
\begin{enumerate}[label=(\textbf{R3\alph*}), leftmargin=2.5em]
\item Assume that $w \circ \varphi=v$.
 Thus, one has an induced morphism $\bar\varphi:\kappa(v) \rightarrow \kappa(w)$
 and, for any prime $\pi$ of $v$, an induced isomorphism of invertible $\kappa(w)$-vector spaces:
\begin{align*}
\theta:\omega_v \otimes_{\kappa(v)} \kappa(w) &\rightarrow \omega_w \\
 \bar \pi^* \otimes 1 &\mapsto \overline{\varphi(\pi)}^*
\end{align*}
where $\omega_v=(\cM_v/\cM_v^2)^\vee$, $\omega_w=(\cM_w/\cM_w^2)^\vee$.
 Then: $\partial_w \circ \varphi_*=\theta_* \circ \bar \varphi_* \circ \partial_v$.
\addtocounter{enumi}{1}
\item Assume $w \circ \varphi=0$.
 Then $\partial_w \circ \varphi_*=0$.
\item Assume $w \circ \varphi=0$
 and let $\bar\varphi:E \rightarrow \kappa(w)$
 be the morphism induced by $\varphi:E \rightarrow F$.
 Let $\pi$ be a prime of $w$, and consider the resulting trivialization
 (sending $\bar \pi^*$ to $1$):
$$
\theta^\pi:\omega_w \rightarrow \kappa(w)
$$
so that $\partial_w^\pi=\theta^\pi_* \circ \partial_w$,
 and:
 $s_w^\pi(\sigma)=-\epsilon\partial_w^\pi([-\pi].\sigma)$.

Then $s_w^\pi \circ \varphi_*=\bar \varphi_*$.
\item The following formulas hold:
\begin{align*}
\partial_v([u].\sigma)&=\epsilon[\bar u].\partial_v(\sigma), \\
\partial_v(\eta.\sigma)&=\eta.\partial_v(\sigma)
\end{align*}
where $u \in \cO_v^\times$ is a unit of $v$.
\end{enumerate}
Formulas (R3a), (R3c) and (R3d) all follow from the construction of the residue map,
 and more precisely from \Cref{thm:residue}, (Res2). Finally,
  formula (R3e) is proved  in \Cref{num:basic_formula_res}.
\end{num}

\begin{num}\label{num:R3b}
We have left the property (R3b), following the numbering of \cite{Rost} and \cite{FeldMWmod},
 for the next theorem. It is certainly the most difficult fact to establish, even
 in the case of Milnor K-theory. Indeed, in the Milnor K-theory case,
 the only proof of (R3b) in the literature that we are aware of is \cite[Cor. 7.4.3]{GSza}.

Let us fix the notation.
 We consider a discretely valued field $(E,v)$, with ring of integers
 $\cO_v$. We fix a finite field extension $F/E$ and consider
 the integral closure $B$ of $\cO_v$ in $F$.
 According to the Krull-Akizuki theorem (\cite[VII, \textsection 2, n°5, Prop. 5]{BouAC57}),
 $B$ is a Dedekind ring. 
 In the next theorem, we will make the important assumption that:
\begin{equation}{\tag{F}}\label{eq:hyp-R3b}
\text{$B$ is a finite $\cO_v$-algebra.}
\end{equation}
 Notably, this condition holds if $F/E$ is separable (\cite[VI, \textsection 8, n°5, Th. 2]{BouAC57}),
 or if $\cO_v$ is Japanese (\emph{e.g.} excellent,
 essentially of finite type over a field or over $\ZZ$).
 We will show in \Cref{prop:R3b+}
 how to modify the next theorem in order to avoid this assumption.

Recall also that there is a bijection between the discrete valuations
 $w$ of $F$ extending $v$ and the non-zero prime ideals of $B$ (\cite[VII, \textsection 2, n°5, Prop. 6]{BouAC57}).
 For such a valuation $w$, one can consider the commutative diagram
$$
\xymatrix@=10pt{
\kappa_w\ar@{}|\Theta[rd] & \cO_w\ar[l] \\
\kappa_v\ar^-{\Phi_w}[u] & \cO_v\ar_-\Phi[u]\ar[l]
}
$$
and the canonical isomorphism of invertible $\kappa_w$-vector spaces
 (apply \Cref{rem:can_iso_can_sheaf} with $\Theta$):
$$
\theta^w:\omega_w \otimes_{\cO_w} \omega_{\cO_w/\cO_v}
 \xrightarrow{\ \sim\ } \omega_{\kappa_w/\kappa_v} \otimes_{\kappa_v} \omega_v.
$$
\end{num}
\begin{thm}[Property \textbf{(R3b)}]\label{thm:R3b}
Consider the above assumptions and notation. Then the following formula,
 involving the basic maps (D2) and (D4) of Milnor-Witt K-theory, holds:
$$
\partial_v \circ \Phi^*=\sum_{w/v} \Phi_w^* \circ \theta^w_* \circ \partial_w.
$$  
\end{thm}
Hiding the isomorphisms $\theta^w$ and using the notation of \Cref{df:MW-transfers},
 this formula can be rewritten as follows:
$$
\partial_v \circ \Tr^{MW}_{F/E}=\sum_{w/v} \Tr^{MW}_{\kappa_w/\kappa_v} \circ \partial_w.
$$

In the proof of the preceding theorem, \emph{i.e.}, relation \textbf{(R3b)},
 we will need the  following \emph{local closedness relation} for iterated residues in Milnor-Witt K-theory.
 The informed reader will recognize the codimension-$2$ relation underlying the identity $d^2=0$ for the
 Rost-Schmid complex; see the remark below for further discussion.
 Recall from \Cref{num:epsilon_KMW} that $\epsilon=-\tw{-1}$.

\begin{lm}[Local closedness]\label{lm:assoc-res}
Let $X$ be an integral (noetherian) regular scheme of dimension $2$, let
$z\in X^{(2)}$, and let $\xi$ be the generic point of $X$. Let $\cL$ be
an invertible $\cO_X$-module and let $\rho\in\KMW q(\kappa(\xi),\cL_\xi)$.
 Given a point $x\in X^{(1)}$,
 we let
$$
\partial_x:\KMW q(\kappa(\xi),\cL_\xi) \longrightarrow
 \KMW{q-1}\bigl(\kappa(x), \omega_{\kappa(x)/\cO_{X,x}}\otimes_{\kappa(x)}\cL_x\bigr)
$$
 be the residue map  associated with the discrete valuation ring
 $\cO_{X,x}$. Then the set of points
$$
\mathrm{Supp}(\rho):=\{x \in X^{(1)} \mid \partial_x(\rho) \neq 0\}
$$
is finite.

We also assume that, for every $x\in X^{(1)}$ such that $z\in\overline{\{x\}}$
 and $\partial_x(\rho)\neq 0$, the normalization of $\overline{\{x\}}$ is
 finite in a neighborhood of $z$. For such a point $x$, put
$$
\partial_z^x=\sum_{y\mapsto z} \Tr^{MW}_{\kappa(y)/\kappa(z)}\circ\partial_y,
$$
where $y$ runs over the finitely many points above $z$ in the normalization of
$\overline{\{x\}}$, and where the canonical identifications of twists
are understood. Then
$$
\sum_{x\in X^{(1)}}\partial_z^x\partial_x(\rho)=0,
$$
where the terms corresponding to the other points $x$ are understood to
be zero.

In particular, assume that $\partial_x(\rho)=0$ for every
$x\in X^{(1)}$ whose closure contains $z$, except two points $x_a$ and
$x_b$. Then
$$
\partial_z^{x_a}\partial_{x_a}(\rho)+\partial_z^{x_b}\partial_{x_b}(\rho)=0.
$$
More explicitly, assume $X=\Spec(R)$ for a $2$-dimensional regular local ring $R$,
 with closed point $z$ and residue field $k=\kappa(z)$, 
 $\rho$ satisfying the preceding vanishing, 
 and assume that $x_a$ and $x_b$ are defined by prime elements $a$ and $b$ such that
 $A=R/(a)$ and $B=R/(b)$
 are discrete valuation rings.
 Let $v$ and $w$ be respectively the $a$-adic and $b$-adic valuations,
 and let $w'$ and  $v'$ be the valuations induced at the closed point
 on $\kappa_v$ and $\kappa_w$, respectively. Consider the canonical transitivity isomorphisms
\begin{align*}
\theta_a:\omega_{A/R}\otimes_A\omega_{w'}
 &\xrightarrow{\;\sim\;}\omega_{k/R}, \\
\theta_b:\omega_{B/R}\otimes_B\omega_{v'}
 &\xrightarrow{\;\sim\;}\omega_{k/R}
\end{align*}
associated with the commutative square
$$
\xymatrix@=10pt{
R\ar[r]\ar[d]\ar[rd] & A\ar[d] \\
B\ar[r] & k.
}
$$
Then the preceding relation becomes
$$
\theta_a\circ\partial_{w'}\circ\partial_v(\rho)=\epsilon\,\theta_b\circ\partial_{v'}\circ\partial_w(\rho).
$$
\end{lm}

\begin{proof}
In this proof, we use relations (Res1) and (Res2) of \Cref{thm:residue}.

After trivializing $\cL$ on a finite affine covering of $X$, relation
(Res1) reduces the finiteness of $\mathrm{Supp}(\rho)$ to the case of
symbols. The principal divisor of each entry of such a symbol has finite
support on $X$, so only finitely many points $x\in X^{(1)}$ satisfy
$\partial_x(\rho) \neq 0$; see also the proof of \Cref{lm:FD_curve}.

\bigskip

Put $k=\kappa(z)$ and replace $X$ by $\Spec(\cO_{X,z})$.
 Let $Z$ be the union of the reduced closures of the points $x\in\mathrm{Supp}(\rho)$.
 There is nothing to prove if $Z$ is empty.

We first consider the case where $Z$ is the union
 of two regular divisors $C_a$ and $C_b$ meeting transversally at the closed point $z$ of $X$.
Choose local equations $(a,b)$ for $(C_a,C_b)$ and consider the blow-up
$q:\widehat X\longrightarrow X$ at $z$, with exceptional divisor $D\simeq\PP^1_k$.
 Let $\omega_q=\omega_{\widehat X/X}\simeq\cO(D)$. Since $q$ is an
 isomorphism at the generic point, the canonical trivialization of
 $\omega_q$ there allows us to regard $\rho$ as an element twisted by
 $q^*\cL\otimes\omega_q$.

Let $\eta_D$ be the generic point of $D$. Using the canonical identification
given by transitivity of canonical lines
$$
\omega_{\kappa(D)/\cO_{\widehat X,\eta_D}}
\otimes_{\kappa(D)}(\omega_q)_{\eta_D}
\simeq
\omega_{\kappa(D)/\kappa(z)}
\otimes_{\kappa(z)}\omega_{z/X},
$$
and the identification
$(q^*\cL)_{\eta_D}\simeq\cL_z\otimes_{\kappa(z)}\kappa(D)$,
the residue map associated with the discrete valuation ring
$\cO_{\widehat X,\eta_D}$ can be written
$$
\partial_D:
\KMW q(\kappa(\xi),\cL_\xi)
\longrightarrow
\KMW{q-1}\Bigl(\kappa(D),\omega_{\kappa(D)/\kappa(z)}\otimes_{\kappa(z)}\bigl(\omega_{z/X}\otimes\cL_z\bigr)\Bigr).
$$
Put: $\mu=\partial_D(\rho)$.

On the chart $b=aT$, the points $T=0$ and $T=\infty$ are the
intersections of $D$ with the strict transforms of $C_b$ and $C_a$,
respectively. The residue formulas give, after the canonical
identifications of twists,
\begin{align*}
\partial_0(\mu)
& =\tw{-1}\,\partial_{\bar a}\partial_b(\rho), \\
\partial_\infty(\mu)
& =\partial_{\bar b}\partial_a(\rho).
\end{align*}
The factor $\tw{-1}$ in the first formula comes from exchanging the two
normal directions. At infinity it is cancelled by the change of
parameter $T^{-1}$, since $dT=-T^2d(T^{-1})$.

For any other closed point $s$ of $D$, the same local calculation expresses
\begin{equation}\label{eq1:lm:assoc-res}
\Tr^{MW}_{\kappa(s)/k}\bigl(\partial_s(\mu)\bigr)
\end{equation}
 as a sum of iterated residues corresponding to components of $Z$ whose strict
 transforms meet $D$ at $s$.
 To see this, after trivializing $\cL$ near $z$,
 one reduces by (Res1) to symbols, factors their entries in the factorial ring
 $\cO_{X,z}$, and applies (Res2) on the two standard charts. 
 Since $Z=C_a\cup C_b$, there are no such components,
 and therefore, \eqref{eq1:lm:assoc-res} vanishes.
 The Weil reciprocity formula of \Cref{rem:Weil_RC}(2) on $D$ therefore gives
$$
\tw{-1}\,\partial_{\bar a}\partial_b(\rho)+\partial_{\bar b}\partial_a(\rho)=0
$$
As $\epsilon=-\tw{-1}$, we obtain
$$
\partial_{\bar b}\partial_a(\rho)
=
\epsilon\,\partial_{\bar a}\partial_b(\rho),
$$
which proves the last equality stated in the lemma, in the transverse case.

We now return to the general case. By the finiteness assumption on the
normalizations of the components of $Z$, there exists a sequence of
blow-ups at closed points above $z$,
$$
p:X'\longrightarrow X,
$$
which may be taken to begin with the blow-up of $z$, such that the
reduced inverse image of $Z$ is a strict normal crossing divisor and the
strict transform of each component of $Z$ identifies, near the fiber
over $z$, with its normalization; see
\cite[Tags 0BI4, 0BIC]{Stack}.

Let $\omega_p=\omega_{X'/X}$. Using its canonical trivialization at the
generic point, regard $\rho$ as an element with twist
$p^*\cL\otimes\omega_p$. Its first residues on $X'$ are supported on the
reduced inverse image of $Z$. Let
$$
E_z=p^{-1}(z)_{\mathrm{red}},
$$
and let $C$ run through the irreducible components of $E_z$.
 Each $C$ is isomorphic to $\PP^1_{k_C}$ for a finite extension $k_C/k$.
 Let $\partial_C$ denote the residue map associated with the discrete valuation ring
 $\cO_{X',\eta_C}$ at the generic point $\eta_C$ of $C$, and put
$$
\mu_C=\partial_C(\rho).
$$
By transitivity of canonical lines, the twist of $\mu_C$ is the
canonical twist of $C/k_C$, tensored with the constant line
$$
\omega_{k_C/k}\otimes_k\omega_{z/X}\otimes_k\cL_z.
$$
Consequently, the Weil reciprocity formula of
\Cref{rem:Weil_RC}(2) on $C$ gives
$$
\sum_{u\in C^{(1)}}
\Tr^{MW}_{\kappa(u)/k_C}\bigl(\partial_u(\mu_C)\bigr)=0.
$$

Apply $\Tr^{MW}_{k_C/k}$ to these identities and sum them over the components of $E_z$. 
 At a point where two components $C_1$ and $C_2$ of $E_z$ meet,
 the two corresponding iterated residues cancel by the transverse two-divisor relation established in the previous case.
 At a point of $E_z$ lying on no other component of the support,
 the corresponding term is zero. Finally, if $C$ meets the strict transform of
$\overline{\{x\}}$ at a point $y$, the same local calculation identifies
the contribution from $C$ with the negative of
$$
\Tr^{MW}_{\kappa(y)/k}\partial_y\partial_x(\rho).
$$
The points $y$ obtained in this way are precisely the points above $z$
in the normalization of $\overline{\{x\}}$. Hence
$$
\sum_{x\in X^{(1)}}\partial_z^x\partial_x(\rho)=0.
$$

Finally, the two-term assertion follows by deleting the zero summands.
 When the two rank-one normal factors are put in the same order,
 their exchange acts by $\tw{-1}$: this gives the last explicit formula
 involving $\theta_a$ and $\theta_b$.
\end{proof}

\begin{proof}[Proof of \Cref{thm:R3b}.]
\noindent\emph{Reduction to the henselian case.}
Let $A=\cO_v$, let $A^h$ be its henselization, and let $E^h$ be the
fraction field of $A^h$. For each extension $w$ of $v$ to $F$, let
$\cO_w^h$ be the henselization of $\cO_w$. Since $B$ is finite over $A$,
the standard decomposition after henselization gives
$$
B\otimes_AA^h\simeq\prod_{w/v}\cO_w^h.
$$
Putting $F_w^h=\Frac(\cO_w^h)$, one further obtains
$$
F\otimes_EE^h\simeq\prod_{w/v}F_w^h.
$$
Using (R3a) and (R1c), applied to the separable extension $E^h/E$ and
the finite extension $F/E$, one reduces to the case where $E$ is
henselian. There is then a unique valuation $w$ on $F$ extending $v$,
and $B=\cO_w$.

Fix $\sigma\in\KMW n(F,\omega_{F/E})$. Hiding the canonical
identifications of twists, it remains to prove
\begin{equation}\label{eq:thm:R3b_0}
\partial_v\bigl(\Tr^{MW}_{F/E}(\sigma)\bigr)
=
\Tr^{MW}_{\kappa_w/\kappa_v}\bigl(\partial_w(\sigma)\bigr).
\end{equation}

\medskip

\noindent\emph{Reduction to a monogenic extension.}
Choose a tower of finite extensions
$$
E=F_0\subset F_1\subset\cdots\subset F_r=F
$$
such that every extension $F_i/F_{i-1}$ is monogenic. Let $v_i$ be the
unique extension of $v$ to $F_i$, let $\kappa_i$ be its residue field,
and put
$$
A_i=B\cap F_i.
$$
Since $B=\cO_w$ and $v_i$ is the restriction of $w$ to $F_i$,
 one has $A_i=\cO_{v_i}$.
 In particular, $A_i$ is a henselian discrete valuation ring. Moreover,
 $A_i$ is the integral closure of $A$ in $F_i$, and also the integral
 closure of $A_{i-1}$ in $F_i$.
 Since $A$ is noetherian and $B$ is a finite $A$-module,
 $A_i$ is finite over $A$. Hence $A_i$ is also finite over $A_{i-1}$.
 Thus condition \eqref{eq:hyp-R3b} of \Cref{num:R3b} holds at every step of the tower.

According to the compatibility of the MW-transfer with composition,
 formula (R1b),
 it is enough to prove \eqref{eq:thm:R3b_0} for every extension $F_i/F_{i-1}$.
 Indeed, if the formula holds for $F/L$ and $L/E$, then
\begin{align*}
\partial_v\Tr^{MW}_{F/E}
&\stackrel{\text{(R1b)}}=\partial_v\Tr^{MW}_{L/E}\Tr^{MW}_{F/L} \\
&=\Tr^{MW}_{\kappa_L/\kappa_v} \partial_{v_L}\Tr^{MW}_{F/L} \\
&=\Tr^{MW}_{\kappa_L/\kappa_v}\Tr^{MW}_{\kappa_w/\kappa_L}\partial_w
\stackrel{\text{(R1b)}}=\Tr^{MW}_{\kappa_w/\kappa_v}\partial_w,
\end{align*}
where the canonical identifications of twists are understood.

We can therefore assume that $F=E[\alpha]$. After multiplying $\alpha$
by an element of $E^\times$, we may assume that $\alpha\in B$. Its monic
minimal polynomial $f$ then belongs to $A[t]$. Using
\Cref{ex:algorithm-norms}, choose
 $\varphi\in\KMW{n+1}\bigl(E(t),\omega_{E(t)/E}\bigr)$
 such that
\begin{align*}
\partial_f(\varphi)&=\sigma,\\
\partial_g(\varphi)&=0
\qquad\text{for every monic irreducible polynomial }g\neq f,\\
\partial_\infty(\varphi)&=-\Tr^{MW}_{F/E}(\sigma).
\end{align*}

Put $k=\kappa_v$. Let $\nu$ be the valuation on $E(t)$ associated with
the special fiber of $\AA^1_A$, and put
$$
\psi=\epsilon\,\partial_\nu(\varphi)
\in
\KMW n\bigl(k(t),\omega_{k(t)/k}\otimes_k\omega_v\bigr),
$$
where the canonical identification of twists is understood.

Let
$$
D=\Spec\bigl(A[t]/(f)\bigr)\subset\AA^1_A.
$$
Its normalization is $\Spec(B)$. Moreover, $D$ is local: indeed, $B$ is
local and integral over $A[\alpha]$, so the lying-over theorem (see \cite[Tag 00GQ]{Stack})
 shows that $A[\alpha]$ has a unique maximal ideal.
 Let $x$ be the closed point of $D$. Then
$$
\kappa(x)=k[\bar\alpha]\subset\kappa_w.
$$

Apply \Cref{lm:assoc-res} on $\AA^1_A$, with
$\cL=\omega_{\AA^1_A/A}$, at the point $x$. Its finiteness hypothesis
holds because the special fiber is regular and the normalization of $D$
is $\Spec(B)$. The only possibly non-zero
first residues through $x$ are those along the special fiber and $D$.
Using the defining properties of $\varphi$, the oriented two-term
relation gives
$$
\partial_x\partial_\nu(\varphi)
=
\epsilon\,
\Tr^{MW}_{\kappa_w/\kappa(x)}
\bigl(\partial_w\partial_f(\varphi)\bigr).
$$
Since (twisted) residue maps commute with multiplication by $\epsilon$
(see the end of \Cref{num:basic_formula_res}) and $\epsilon$ is a unit,
 it follows that
$$
\partial_x(\psi)=\Tr^{MW}_{\kappa_w/\kappa(x)}\bigl(\partial_w(\sigma)\bigr).
$$
For every other closed point $x'\in\AA^1_k$, the divisor $D$ does not
pass through $x'$, and all the horizontal residues of $\varphi$ through
$x'$ vanish. So applying again \Cref{lm:assoc-res} gives
$$
\partial_{x'}(\psi)=0.
$$

Weil reciprocity on $\PP^1_k$, with the constant twist $\omega_v$, and
compatibility of transfers with composition now give
\begin{equation}
\label{eq:thm:R3b_residue}
\partial_\infty(\psi)=-\Tr^{MW}_{\kappa(x)/k}\bigl(\partial_x(\psi)\bigr)
= -\Tr^{MW}_{\kappa_w/k}\bigl(\partial_w(\sigma)\bigr).
\end{equation}

Finally, consider the closed point at infinity of $\PP^1_A$ and apply
\Cref{lm:assoc-res} with $\cL=\omega_{\PP^1_A/A}$. Since $f$ is monic,
the closure of $D$ does not meet the section at infinity. Thus the only
possibly non-zero first residues of $\varphi$ along prime divisors
through this point are those along the special fiber and the section at
infinity. The oriented two-term relation gives
$$
\partial_\infty\partial_\nu(\varphi)
=
\epsilon\,\partial_v\partial_\infty(\varphi).
$$
Hence, again using $\epsilon^2=1$,
$$
\partial_\infty(\psi)=\partial_v\bigl(\partial_\infty(\varphi)\bigr)=-\partial_v\bigl(\Tr^{MW}_{F/E}(\sigma)\bigr).
$$
Together with \eqref{eq:thm:R3b_residue}, this proves
\eqref{eq:thm:R3b_0}.
\end{proof}

\begin{rem}
\begin{enumerate}
\item The proof of the preceding theorem and formula (R3b) are equally valid for Milnor K-theory.
 As said before, the only proof of (R3b) for Milnor K-theory that is known to us
 is in \cite[Cor. 7.4.3]{GSza}. The proof is based on initial results  due to Bass-Tate and Kato.
 The above proof, based on structural theorems of finite extensions of discretely valued fields,
 is more direct.
\item In the case of Milnor-Witt K-theory, one can derive a proof of (R3b) for an essentially smooth valuation $k$-algebra
 from \cite[Th. 5.26]{Mor}. One can also find an argument in \cite[Cor. 10.4.5]{FaselCW}, when $k$ has characteristic different from $2$.\footnote{Beware
 that the indicated corollary is claimed for arbitrary regular schemes,
 but the theorems on which the corollary is based, \emph{loc. cit.}, Ths. 2.3.1 and 8.3.4, are only proved for essentially smooth $k$-schemes.} 
\end{enumerate}
\end{rem}

It is worth mentioning that one can deduce the following generalization 
 of the classical Weil reciprocity formula for Milnor-Witt K-theory.
\begin{cor}[Weil reciprocity on proper curves]\label{cor:Weil_RC_curves}
Let $C$ be a proper regular integral curve over a field $k$, and let
$\cL$ be an invertible $k$-vector space. Then, for every
$$
\rho\in\KMW{q+1}\bigl(\kappa(C),\omega_{\kappa(C)/k}\otimes_k\cL\bigr),
$$
one has
\begin{equation}\label{eq:Weil_RC_curves}
\sum_{x\in C_{(0)}}
\Tr^{MW}_{\kappa(x)/k}\bigl(\partial_x(\rho)\bigr)=0
\quad\text{in }\KMW q(k,\cL),
\end{equation}
where the canonical identifications of twists are understood.
\end{cor}

\begin{proof}
Choose an element $t\in\kappa(C)$ transcendental over $k$. The associated
morphism $p:C\rightarrow\PP^1_k$ is finite. Put $E=k(t)$ and
$F=\kappa(C)$. By transitivity of canonical lines,
$$
\tau=\Tr^{MW}_{F/E}(\rho)
\in\KMW{q+1}\bigl(E,\omega_{E/k}\otimes_k\cL\bigr).
$$
The finiteness hypothesis of \Cref{thm:R3b} holds because $p$ is finite.
Thus \Cref{thm:R3b}, compatibility of transfers with composition, and
\Cref{rem:Weil_RC}(2) give
\begin{align*}
0
&=\sum_{a\in(\PP^1_k)_{(0)}}
  \Tr^{MW}_{\kappa(a)/k}\bigl(\partial_a(\tau)\bigr)\\
&=\sum_{a\in(\PP^1_k)_{(0)}}\ \sum_{x\mapsto a}
  \Tr^{MW}_{\kappa(x)/k}\bigl(\partial_x(\rho)\bigr),
\end{align*}
which is \eqref{eq:Weil_RC_curves}.
\end{proof}

\begin{num}
It remains to state the last property, which is specific to Milnor-Witt K-theory.
 Let us first remark that, given an invertible $E$-vector space $\cL$,
 one has an isomorphism:
$$
E^\times \rightarrow \mathrm{Aut}_E(\cL), u \mapsto (l \mapsto u.l).
$$
Given an $E$-automorphism $\Theta$ of $\cL$, one denotes by $\delta_\Theta \in E^\times$
 the corresponding unit.
\begin{enumerate}[label=\textbf{(R4a)}, leftmargin=2.5em]
\item For any automorphism $\Theta$ of an invertible $E$-vector space $\cL$,
 and any $\sigma \in \KMW*(E,\cL)$,
 one has $\Theta_*(\sigma)=\tw{\delta_\Theta}.\sigma$.
\end{enumerate}
\end{num}

\begin{rem}
 Given a base scheme $S$, we call an $S$-field a field $E$ together with
 a morphism $\Spec(E) \rightarrow S$ essentially of finite type.
\begin{enumerate}
\item Let $S$ be a scheme essentially of finite type over a field.
 The data and rules seen so far, except that one has to consider the slightly stronger property (R3a+)
 (see \Cref{prop:R3a+}) show that $\KMW*$ restricted to $S$-fields
 forms a MW-premodule in the sense of \cite[Def. 3.1]{FeldMWmod}.
\item Let $S$ be a noetherian scheme equipped with a dimension function.
 Then the data and rules obtained above show that $\KMW*$ restricted to $S$-fields
 forms a cohomological MW-premodule in the sense of \cite[Def. 3.2.1]{DFJ}.
\item Assume moreover that $S$ is excellent and admits a pinning
 in the sense of \cite[Def. 3.3.1]{DFJ}.
 The local closedness relation of \Cref{lm:assoc-res}, together with
 \Cref{thm:R3b} and \Cref{cor:Weil_RC_curves}, gives the closedness
 property of \cite[Def. 3.3.2]{DFJ} for Milnor-Witt K-theory.
 Indeed, the case of an arbitrary excellent integral local surface
 follows by resolving its singularities: \Cref{lm:assoc-res} gives
 closedness on the resulting regular surface, the contributions of the
 exceptional curves vanish by \Cref{cor:Weil_RC_curves}, and
 \Cref{thm:R3b} identifies the remaining terms with the corresponding
 residues on the original surface. The general closedness relation then
 follows by the usual reduction to the local $2$-dimensional case.

 Consequently, $\KMW*$ is a MW-module in the sense of
 \cite[Def. 3.3.2]{DFJ}; equivalently, the differentials in the general
 Rost-Schmid complex defined in \cite[\S 3.3.7]{DFJ} satisfy
 $d^2=0$. This gives an alternative proof of \cite[Ex. 3.3.6]{DFJ}.
\end{enumerate}
\end{rem}

\subsection{Finer properties and quadratic multiplicities}\label{sec:strong}

In this section,
 we formulate, following \cite{FeldMWmod, FeldTohoku}, stronger forms of properties (R1c) and (R3a) involving multiplicities,
 as in the theory of cycle modules \cite{Rost}.
 We also give a refinement of (R3b) which, even for Milnor K-theory,  is new.

Note that the main difficulty compared to Rost's theory
 is the necessity to describe what happens on twists. Compared to the formula given by Feld,
 we make explicit the isomorphisms needed to get coherent formulas.

\begin{num}
We start with the stronger form of (R3a).
 We consider a ramified extension $\varphi:E \rightarrow F$ of valued fields $(E,v)$, $(F,w)$
 with ramification index $e>0$: $w \circ \varphi=e.v$.
 We still denote by $\varphi:\cO_v \rightarrow \cO_w$ the induced morphism
 on the ring of integers, and by $\bar \varphi:\kappa_v \rightarrow \kappa_w$ the induced map
 on the residue fields.

Let us choose uniformizers $\pi_v \in \cO_v$, $\pi_w \in \cO_w$, respectively for $v$ and $w$.
 One deduces a canonical isomorphism of $\kappa_w$-vector spaces:
$$
\theta:\omega_v \otimes_{\kappa_v} \kappa_w\rightarrow \omega_w, \bar \pi_v^* \otimes 1 \mapsto \bar \pi_w^*
$$
where $\omega_v=(\cM_v/\cM_v^2)^\vee$ (resp. $\omega_w=(\cM_w/\cM_w^2)^\vee$)
 are the respective normal sheaves.

Note that there exists a uniquely defined unit $u \in \cO_w^\times$
 such that $\varphi(\pi_v)=u.\pi_w^e$.
\end{num}
\begin{prop}[Property \textbf{(R3a+)}]\label{prop:R3a+}
Consider the above hypothesis and notation. Then the following diagram commutes:
$$
\xymatrix@=20pt{
\KMW*(E)\ar^-{\partial_v}[r]\ar_{\varphi_*}[d] & \KMW*(\kappa(v),\omega_v)\ar^{\tw{\bar u}.e_\epsilon.(\bar \varphi_* \otimes \theta_*)}[d] \\
\KMW*(F)\ar^-{\partial_w}[r] & \KMW*(\kappa(w),\omega_w).
}
$$
Moreover, the right vertical map does not depend on the choice of uniformizers $\pi_v$ and $\pi_w$.
\end{prop}
\begin{proof}
Consider an element $\sigma \in \KMW*(E)$.
 As all maps commute with multiplication by $\eta$, one reduces to consider a symbol of the form $\sigma=[u_1,\ldots,u_n]$.
 By using relation (MW2) of Milnor-Witt K-theory, the fact $w(\varphi(\pi_v))>0$, and the properties of the residue map,
 one reduces to the case where $\sigma=[\pi_v,u_2,\ldots,u_n]$, with $u_i \in \cO_v^\times$.
 We compute the composite of the maps through the left-down right corner:
$$
\partial_w\big(\varphi_*(\sigma)\big)=\partial_w\big([u.\pi_w^e,\varphi(u_2),\ldots,\varphi(u_n)]\big)=
 \tw{\bar u}.e_\epsilon.[\bar \varphi(\bar u_2),\ldots,\bar \varphi(\bar u_n)] \otimes \bar \pi_w^*,
$$
where the last equality follows from \Cref{thm:residue}(Res2).
 Another application of \emph{loc. cit.} gives $\partial_v(\sigma)=[\bar u_2,\ldots,\bar u_n] \otimes \bar \pi_v^*$, and so the first
 assertion follows.

For the second assertion, we write $\pi'_v=u_v\pi_v$, $\pi'_w=u_w\pi_w$, with $u_v$, $u_w$ units.
 Then a straightforward computation reduces to show the equality in $\GW(\kappa_w)$:
\begin{equation}\label{eq:strong_R3c}
\tw{\bar u_w}.e_\epsilon=\tw{\bar u_w^e}.e_\epsilon.
\end{equation}
If $e$ is odd, one gets $\tw{\bar u_w^e}=\tw{\bar u_w}$ and therefore \eqref{eq:strong_R3c} is true.
 If $e=2n$ is even, $e_\epsilon=n.h$.
 But for any unit $a \in \kappa_w^\times$, one has: $\tw a.h=h$ (\Cref{thm:GW}(GW3)).
 Thus \eqref{eq:strong_R3c} holds true in that latter case.
\end{proof}

\begin{rem}
In the preceding proposition, one cannot avoid in general the presence of the correcting unit $\bar u$
 in the formula of the right vertical map.
 Using property R4a, it is possible to give a more compact definition of this map.
 Indeed, working in the abelian group 
$$
\ZZ[Hom_{\kappa_w}(\omega_v \otimes_{\kappa_v} \kappa_w,\omega_w)]
 =\ZZ[Hom_{\kappa_v}(\omega_v,\omega_w)],
$$
one can define the element:
$$
\theta_u^e=\sum_{i=0}^{e-1} \big((-1)^i\bar u\big).\theta.
$$
With that definition, the formula of the preceding proposition reads:
$$
\partial_w \circ \varphi_*=(\bar \varphi_* \otimes \theta_u^e) \circ \partial_v.
$$
This last formula agrees with the computation of the $\AA^1$-homotopical defect
 of the purity isomorphism done in \cite[Th. 2.2.2]{Feld2}.
\end{rem}

\begin{num}
The preceding formula has interesting corollaries.
 Let us set up the notation before stating the first one.

We let $\varphi:E \rightarrow F$ be an arbitrary field extension,
 $\varphi':E(t) \rightarrow F(t)$ the induced extension.
 A closed point $x \in (\AA^1_E)_{(0)}$ corresponds to a monic irreducible polynomial
 $\pi_x \in E[t]$ and we denote by $v_x$ the corresponding $\pi_x$-adic valuation on $E(t)$.
 One can consider the prime decomposition in $F[t]$:
$$
\varphi'(\pi_x)=\prod_{y/x} \pi_y^{e_{y/x}}.
$$
The product runs over a finite family of closed points $y \in (\AA^1_F)_{(0)}$, corresponding
 to the irreducible polynomials $\pi_y \in F[t]$, and the integers $e_{y/x}$ are some multiplicities.
 For every $y/x$, put
$$
u_{y/x}=\varphi'(\pi_x).\pi_y^{-e_{y/x}} \in \cO_{v_y}^\times,
$$
and let $\bar u_{y/x}\in\kappa(y)^\times$ be its residue.
 Equivalently, the $\pi_y$-adic valuations $v_y$ on $F(t)$ run over the extensions of the
 valuation $v_x$, such that $v_y \circ \varphi'=e_{y/x}.v_x$.
 As for the preceding proposition, we consider $\omega_x$ and $\omega_y$ to be the respective
 normal sheaves associated with $v_x$ and $v_y$.
 Then one considers the isomorphism
 $\theta_y:\omega_x \otimes_{\kappa_x} \kappa_y \rightarrow \omega_y$,
 sending $\bar \pi_x^* \otimes 1$ to $\bar \pi_y^*$.
 We let $\varphi_y:\kappa(x) \rightarrow \kappa(y)$ be the induced morphism.
\end{num}
\begin{cor}
Consider the above notation. Then the following diagram commutes:
$$
\xymatrix@R=15pt@C=25pt{
0\ar[r] & \KMW*(E)\ar^{}[r]\ar_{\varphi_*}[d] & \KMW*(E(t))\ar^-{d_E}[r]\ar_{\varphi'_*}[d]\ar@{}|{(2)}[rd]
 & \bigoplus_{x \in (\AA^1_E)_{(0)}} \KMW*(\kappa(x),\omega_x)\ar[r]
      \ar^{\sum_{y/x} \tw{\bar u_{y/x}}.e_{y/x,\epsilon}.(\varphi_{y*} \otimes \theta_{y*})}[d]
 & 0 \\
0\ar[r] & \KMW*(F)\ar^{}[r] & \KMW*(F(t))\ar^-{d_F}[r]
 & \bigoplus_{y \in (\AA^1_F)_{(0)}} \KMW*(\kappa(y),\omega_y)\ar[r] & 0 
}
$$
where the two horizontal sequences are the split short exact sequences deduced
 from \Cref{thm:htp_inv_fields}.
\end{cor}
\begin{proof}
The commutativity of the left-hand square is the easy formula (R1a).
 For square (2), we consider an element $\sigma \in \KMW*(E(t))$.
 As all maps involved commute with $\eta$, one can assume
 $\sigma=[f_1,\ldots,f_n]$, $f_i \in E(t)^\times$. Let $S \subset (\AA^1_E)_{(0)}$ be the finite
 set of points such that the family $(\pi_x)_{x \in S}$ is exactly made of the irreducible polynomials
 appearing in the prime decomposition of the $f_i$.
 Thus, $d_E(\sigma)=\sum_{x \in S} \partial_{v_x}(\sigma)$.

Similarly, let $T \subset (\AA^1_F)_{(0)}$ be the finite set such that the family $(\pi_y)_{y \in T}$
 is made of the irreducible polynomials appearing in the prime decomposition of the $\varphi'(f_i)$.
 Then $d_F(\varphi'_*(\sigma))=\sum_{y \in T} \partial_{v_y}(\varphi'_*(\sigma))$.

With this notation, the conclusion comes from applying \Cref{prop:R3a+} to each point $y \in T$
 and then taking the sum of the resulting formulas.
\end{proof}

\begin{rem}
Considering the Rost-Schmid complex as defined in \Cref{df:CHW_curves},
 the right vertical map of the diagram can be seen as the definition
 of a pullback map
$$
f^*:\Cq 1(\AA^1_E)_* \rightarrow \Cq 1(\AA^1_F)_*
$$
 associated to the flat morphism $f:\AA^1_F \rightarrow \AA^1_E$.
 In fact, the commutativity of square (2) gives
 (after adding twists with a line bundle $\cL$ over $\AA^1_E$)
 a well-defined morphism of complexes, called the \emph{flat pullback}:
$$
f^*:\Cq *(\AA^1_E,\cL)_* \rightarrow \Cq *(\AA^1_F,f^*\cL)_*.
$$
The definition of pullbacks on Rost-Schmid complexes, and thus Chow-Witt groups,
 associated with smooth morphisms is well-known (see e.g., \cite{FeldMWmod}).
 Contrary to the case of Chow groups, pullbacks with respect to flat morphisms
 have been left open since the foundational work of Fasel, \cite{FaselCW}.
 It will be treated in the forthcoming paper \cite{FJ}.
\end{rem}

Using the Bass-Tate approach to transfers in the monogenic case,
 and especially the characterization obtained in \Cref{prop:BT}, one deduces from the commutativity
 of square (2) in the previous theorem the following result, which we state as a lemma
 for the next statement.
\begin{lm}\label{lm:R1c+}
Let $\varphi:E \rightarrow F$ be an arbitrary field extension,
 and consider the notation of the previous corollary.
 Then the following square is commutative:
$$
\xymatrix@R=20pt@C=65pt{
\bigoplus_{x \in (\PP^1_E)_{(0)}} \KMW*(\kappa(x),\omega_{\kappa(x)/E})\ar^-{\sum_x \Tr^{MW}_{\kappa(x)/E}}[r]
      \ar_{\sum_{y/x} \tw{\bar u_{y/x}}.e_{y/x,\epsilon}.(\varphi_{y*} \otimes \theta_{y*})}[d]
 & \KMW*(E)\ar^{\varphi_*}[d] \\
\bigoplus_{y \in (\PP^1_F)_{(0)}} \KMW*(\kappa(y),\omega_{\kappa(y)/F})\ar^-{\sum_y \Tr^{MW}_{\kappa(y)/F}}[r]
 & \KMW*(F)
}
$$
where the sum on the vertical left-hand side runs over the points $y \in \PP^1_F$
 which lie above $x \in \PP^1_E$. For the finite points, $e_{y/x}$ and $u_{y/x}$
 are defined as in the previous corollary, and we put
 $e_{\infty/\infty}=1$, $u_{\infty/\infty}=1$.
 We have abused notation by denoting by $\theta_y$ the isomorphism induced by the one
 of the previous corollary. For finite points, it is explicitly:
\begin{align*}
\theta_y:\omega_{\kappa(x)/E} \otimes_{\kappa(x)} \kappa(y)
 \simeq \big(\omega_x \otimes \omega_{\AA^1_{\kappa(x)}/\kappa(x)}\big) \otimes_{\kappa(x)} \kappa(y)
 & \rightarrow \omega_y \otimes \omega_{\AA^1_{\kappa(y)}/\kappa(y)} \simeq \omega_{\kappa(y)/F}  \\
(\bar \pi_x^* \otimes dt) \otimes 1 &\mapsto \bar \pi_y^* \otimes dt.
\end{align*}
At infinity, $\theta_{\infty}$ is the canonical isomorphism
 $\omega_{E/E} \otimes_E F \simeq \omega_{F/F}$.
\end{lm}
Indeed, it suffices to apply the preceding corollary and \Cref{prop:BT},
 together with (R3a) for the case of the valuation at infinity on $E(t)$.

\begin{num}
We are now ready to state the stronger form of axiom (R1c),
 without any assumption of separability. Namely, we consider a commutative square of rings:
$$
\xymatrix@=10pt{
E\ar^\Phi[r]\ar_\varphi[d] & L\ar[d] \\
F\ar[r] & R
}
$$
where $E$, $F$, $L$ are fields,
 $\Phi$ is finite, and we have put: $R=F \otimes_E L$.
 Then $R$ is an Artinian ring, not necessarily reduced.
 We let $X=\Spec(R)$.

Let us fix a point $x \in X$, and put
$$
A_x=R_x, \qquad \kappa_x=\kappa(x), \qquad
\mathfrak m_x=xA_x.
$$
Since $L/E$ is finite lci and $F/E$ is flat, $A_x/F$ is a finite local and lci algebra.
 In particular, $A_x$ is an Artinian local Gorenstein ring  (see \cite[Tag 0DWA]{Stack}), and its socle
$$
S_x=\mathrm{Soc}(A_x):=\mathrm{Ann}_{A_x}(\mathfrak m_x)
$$
is a one-dimensional $\kappa_x$-vector space.

Flat base change and localization give a canonical isomorphism
$$
\beta_x:\omega_{L/E} \otimes_L A_x \xrightarrow{\ \sim\ } \omega_{A_x/F}.
$$
On the other hand, finite duality (\Cref{num:trivial_duality}) identifies
$\omega_{A_x/F}$ and $\omega_{\kappa_x/F}$ respectively with
$\Hom_F(A_x,F)$ and $\Hom_F(\kappa_x,F)$.
 Thus, the quotient $q_x:A_x \rightarrow \kappa_x$ induces, by precomposition,
 a canonical injection
$$
\iota_x:\omega_{\kappa_x/F} \lhook\joinrel\longrightarrow \omega_{A_x/F}
$$
whose image is $\mathrm{Ann}_{\omega_{A_x/F}}(\mathfrak m_x)$.
 As $\omega_{A_x/F}$ is invertible over $A_x$, multiplication induces a canonical isomorphism
$$
S_x \otimes_{\kappa_x}
 \big(\omega_{A_x/F} \otimes_{A_x} \kappa_x\big)
 \xrightarrow{\ \sim\ }
 \mathrm{Ann}_{\omega_{A_x/F}}(\mathfrak m_x).
$$
Combining these maps with $\beta_x$, we obtain a canonical isomorphism
\begin{equation}\label{eq:socle_comparison_R1c}
\chi_x:S_x \otimes_{\kappa_x}
 \big(\omega_{L/E} \otimes_L \kappa_x\big)
 \xrightarrow{\ \sim\ } \omega_{\kappa_x/F}.
\end{equation}
It is uniquely characterized by the equality
$$
\iota_x\big(\chi_x(s \otimes \bar w)\big)=s.\beta_x(w),
$$
where $s \in S_x$, $\bar w \in \omega_{L/E}\otimes_L\kappa_x$,
 and $w \in \omega_{L/E}\otimes_L A_x$ is any lift of $\bar w$.
 The right-hand side does not depend on the chosen lift since $s.\mathfrak m_x=0$.

We now choose a tower of monogenic extensions
\begin{equation}\label{eq:monogenic_tower_R1c}
\underline\alpha:\quad
E=L_0 \subset L_1 \subset \cdots \subset L_n=L,
\qquad L_i=L_{i-1}[\alpha_i].
\end{equation}
Let $f_i \in L_{i-1}[t_i]$ be the monic minimal polynomial of $\alpha_i$, and put
$R_i=F \otimes_E L_i$. For $x \in X$, let $x_i$ be the image of $x$ in $\Spec(R_i)$,
 and write
$$
A_{x,i}=(R_i)_{x_i}, \qquad
\kappa_{x,i}=\kappa(x_i), \qquad
\mathfrak m_{x,i}=x_iA_{x,i}, \qquad
S_{x,i}=\mathrm{Soc}(A_{x,i}).
$$
Thus $A_{x,0}=\kappa_{x,0}=S_{x,0}=F$, and $A_{x,n}=A_x$, $\kappa_{x,n}=\kappa_x$,
 $S_{x,n}=S_x$.

For $1 \leq i \leq n$, the morphism $A_{x,i-1} \rightarrow A_{x,i}$ is finite flat and local.
 Its closed fiber
$$
B_{x,i}=A_{x,i} \otimes_{A_{x,i-1}} \kappa_{x,i-1}
$$
is a monogenic local Artinian $\kappa_{x,i-1}$-algebra with residue field $\kappa_{x,i}$.
 More precisely, if $\bar f_i$ denotes the image of $f_i$ in
$\kappa_{x,i-1}[t_i]$, there exist a monic irreducible polynomial $g_{x,i}$ and a unit
$u_{x,i} \in \kappa_{x,i-1}[t_i]_{(g_{x,i})}^{\times}$ such that
$$
B_{x,i}\simeq
 \kappa_{x,i-1}[t_i]_{(g_{x,i})}/(\bar f_i),
\qquad
\bar f_i=u_{x,i}.g_{x,i}^{e_{x,i}},
$$
where $e_{x,i}=\lg(B_{x,i})$.
 If $\gamma_{x,i}$ denotes the image of $g_{x,i}$ in $B_{x,i}$, then
$$
\sigma_{x,i}=\gamma_{x,i}^{e_{x,i}-1}
$$
is a basis of $\mathrm{Soc}(B_{x,i})$ over $\kappa_{x,i}$.
 Multiplication induces a canonical isomorphism
\begin{equation}\label{eq:socle_composition_R1c}
S_{x,i-1} \otimes_{\kappa_{x,i-1}} \mathrm{Soc}(B_{x,i})
 \xrightarrow{\ \sim\ } S_{x,i},
\qquad s \otimes \bar b \longmapsto s.b,
\end{equation}
where $b$ is any lift of $\bar b$ to $A_{x,i}$.
 The product is independent of the lift since
 $S_{x,i-1}.\mathfrak m_{x,i-1}A_{x,i}=0$.
 Moreover, if $s \neq 0$, flatness gives
 $\mathrm{Ann}_{A_{x,i}}(s)=\mathfrak m_{x,i-1}A_{x,i}$;
 hence the map is nonzero, and therefore is an isomorphism.

Starting from $s_{x,0}=1$, we define inductively $s_{x,i}$ as the image of
$s_{x,i-1}\otimes\sigma_{x,i}$ under \eqref{eq:socle_composition_R1c}, and put
$$
s_x^{\underline\alpha}=s_{x,n}.
$$
Thus $s_x^{\underline\alpha}$ is a basis of the one-dimensional
$\kappa_x$-vector space $S_x$, or equivalently defines a trivialization
$$
\kappa_x\xrightarrow{\sim}S_x,\qquad 1\longmapsto s_x^{\underline\alpha}.
$$
Moreover, the length formula for finite flat local morphisms gives
$$
e_x:=\lg(A_x)=e_{x,1}\cdots e_{x,n}.
$$
The chosen tower therefore gives an isomorphism
\begin{align*}
\theta_x^{\underline\alpha}:\omega_{L/E} \otimes_L \kappa_x
 &\xrightarrow{\ \sim\ } \omega_{\kappa_x/F}, \\
w &\longmapsto \chi_x\big(s_x^{\underline\alpha}\otimes w\big),
\end{align*}
and a twisted quadratic multiplicity
$$
\mu_x^{\underline\alpha}
 =(s_x^{\underline\alpha})_*(e_{x,\epsilon})
 \in \GW(\kappa_x,S_x)=\KMW0(\kappa_x,S_x).
$$
Note that the basis element $s_x^{\underline\alpha}$ here plays the role of a \emph{local orientation}
 of the socle line $S_x$, and thereby of the comparison of twists at $x$.

We finally put
\begin{align*}
\delta_x^{\underline\alpha}:\KMW*(L,\omega_{L/E})
 &\longrightarrow \KMW*(\kappa_x,\omega_{\kappa_x/F}), \\
\sigma &\longmapsto
 \chi_{x*}\big(\mu_x^{\underline\alpha}.\varphi_{x*}(\sigma)\big).
\end{align*}
Equivalently,
\begin{equation}\label{eq:delta_R1c}
\delta_x^{\underline\alpha}
 =e_{x,\epsilon}.\big(\varphi_{x*}\otimes\theta_{x*}^{\underline\alpha}\big).
\end{equation}
The individual isomorphism $\theta_x^{\underline\alpha}$ generally depends on the chosen tower.
 The relevant object is the weighted map in \eqref{eq:delta_R1c}.
 The following result generalizes \cite[Th. 2.2.14]{FeldTohoku}
 and makes explicit the required identifications of twists.
\end{num}
\begin{thm}[Property \textbf{(R1c+)}]\label{thm:R1c+}
Consider the above notation. Then the following diagram is commutative:
$$
\xymatrix@R=30pt@C=40pt{
\KMW*(L,\omega_{L/E})\ar^-{\Phi^*}[r]\ar_{\sum_x \delta_x^{\underline\alpha}}[d]
 & \KMW*(E)\ar^-{\varphi_*}[d] \\
\bigoplus_{x \in X} \KMW*(\kappa_x,\omega_{\kappa_x/F})\ar^-{\sum_x \Phi_x^*}[r]
 & \KMW*(F).
}
$$
Here the maps $\Phi_x:F \rightarrow \kappa_x$ and $\varphi_x:L \rightarrow \kappa_x$
 are induced respectively by $\Phi$ and $\varphi$.

Moreover, the left-hand vertical map is independent
 of the chosen monogenic tower $\underline\alpha$.
\end{thm}
\begin{proof}
We follow the induction used in \emph{loc. cit.},
 making explicit at each step the required comparison of twists.
 We first consider the monogenic case $n=1$.
 To simplify the notation, write $f$, $g_x$, $u_x$, $\gamma_x$ and $e_x$ for
$f_1$, $g_{x,1}$, $u_{x,1}$, $\gamma_{x,1}$ and $e_{x,1}$.
 Put
$$
w_f=\bar f^*\otimes dt, \qquad w_{g_x}=\bar g_x^*\otimes dt.
$$
In the conormal module, one has
$\bar f=\bar u_x.\overline{g_x^{e_x}}$, and therefore
$w_{g_x^{e_x}}=\bar u_x.w_f$.
 Under finite duality, precomposition with $A_x \rightarrow \kappa_x$ sends
$w_{g_x}$ to $\gamma_x^{e_x-1}.w_{g_x^{e_x}}$.
 Consequently,
$$
\iota_x(w_{g_x})=\bar u_x.\gamma_x^{e_x-1}.w_f,
$$
where $\bar u_x \in \kappa_x^\times$ is the residue of $u_x$.
 The defining property of $\chi_x$ then gives
$$
\theta_x^{\underline\alpha}(w_f)=\bar u_x^{-1}.w_{g_x}.
$$
Since $\tw{\bar u_x^{-1}}=\tw{\bar u_x}$ in $\GW(\kappa_x)$, formula
\eqref{eq:delta_R1c} agrees with the left-hand vertical map of \Cref{lm:R1c+}.
 The commutativity of the diagram follows in this case.

For a general tower \eqref{eq:monogenic_tower_R1c}, we use (R1b) and apply the monogenic
 case successively to the extensions $L_i/L_{i-1}$, decomposing at each step into the actual
 local Artinian factors, as in \emph{loc. cit.}
 The transitivity of finite duality (\Cref{rem:diff_trace_composition}) and the canonical
 isomorphisms \eqref{eq:socle_composition_R1c} show that the resulting comparison of twists
 is precisely $\theta_x^{\underline\alpha}$.
 Together with the multiplicativity
$$
(e_{x,1}\cdots e_{x,n})_\epsilon
 =(e_{x,1})_\epsilon\cdots(e_{x,n})_\epsilon
$$
(see the end of \Cref{num:quad_mult}), this proves the first assertion by induction.

The last assertion is proved by the same induction.
 In the monogenic case, the weighted map is independent of the chosen parameter by the last
 assertion of \Cref{prop:R3a+}. At each further step, the socle orientations are combined by
 the canonical map \eqref{eq:socle_composition_R1c}, and the comparisons $\chi_x$ are compatible
 with composition by finite duality. Thus the twist identifications are compatible with the
 inductive diagrams of \emph{loc. cit.}, and the map
 $\sum_x\delta_x^{\underline\alpha}$ is independent of $\underline\alpha$.
\end{proof}

\begin{ex}
Suppose that, for a fixed point $x$, the chosen tower
\eqref{eq:monogenic_tower_R1c} is represented, as in
\Cref{ex:field_extension_complete_inter}, by polynomials
$\pi_i\in E[t_1,\ldots,t_i]$, monic in $t_i$, whose image in
$L_{i-1}[t_i]$ is $f_i$.
We denote by $\varphi(\pi_i)$ their images in
$F[t_1,\ldots,t_i]$.

Suppose moreover that the chosen tower gives a compatible
independent-power presentation
$$
A_x\simeq F[t_1,\ldots,t_n]_{\mathfrak m}/
 \big(G_{x,1}^{e_{x,1}},\ldots,G_{x,n}^{e_{x,n}}\big),
$$
where $\mathfrak m$ is the maximal ideal corresponding to $x$ and,
for each $i$, $G_{x,i}\in F[t_1,\ldots,t_i]$ maps to
$g_{x,i}$ in $\kappa_{x,i-1}[t_i]$. Assume that, successively in the
corresponding local polynomial rings,
$$
\varphi(\pi_i)\equiv U_{x,i}G_{x,i}^{e_{x,i}}
\mod
\big(G_{x,1}^{e_{x,1}},\ldots,G_{x,i-1}^{e_{x,i-1}}\big),
$$
for some unit $U_{x,i}$.

Let $\Gamma_{x,i}$ be the image of $G_{x,i}$ in $A_x$, and let
$\bar u_{x,i}$ denote the image of $U_{x,i}$ in $\kappa_x^\times$.
Then
$$
s_x^{\underline\alpha}
 =\Gamma_{x,1}^{e_{x,1}-1}\cdots
  \Gamma_{x,n}^{e_{x,n}-1}.
$$
Moreover, under the canonical transitivity isomorphisms,
\begin{align*}
&\theta_x^{\underline\alpha}
 \big((\bar\pi_1\wedge\ldots\wedge\bar\pi_n)^*
      \otimes(dt_1\wedge\ldots\wedge dt_n)\big) \\
&\hspace{25mm}=
 \left(\prod_{i=1}^n \bar u_{x,i}\right)^{-1}
 (\bar G_{x,1}\wedge\ldots\wedge\bar G_{x,n})^*
      \otimes(dt_1\wedge\ldots\wedge dt_n),
\end{align*}
where the bars in the last formula denote the corresponding classes in the
conormal modules.

Thus, whenever such a compatible independent-power presentation exists,
 for example by means of local coordinates,
 the preceding formula gives an explicit expression of the local socle orientation,
 and hence a concrete way of computing the quadratic multiplicity appearing in (R1c+).
\end{ex}

\begin{rem}
This theorem is in fact the projection formula $f^*p_*=q_*g^*$
 in (graded) Chow-Witt groups for the Cartesian square:
$$
\xymatrix@R=10pt@C=20pt{
\PP^1_F\ar^-q[r]\ar_-g[d] & \Spec(F)\ar^f[d] \\
\PP^1_E\ar^-p[r] & \Spec(E).
}
$$
This is one of the main properties needed for a flat pullback.
 Compared with \cite{FeldTohoku}, we have avoided the need of a (perfect)
 base field, and we have made explicit the comparison of twists entering
 the maps $\delta_x^{\underline\alpha}$.
 
Surprisingly, $f$ is not necessarily of finite type,
 as the corresponding field extension needs not be finitely generated.
\end{rem}

\begin{num}
As a corollary of the preceding theorem, one can refine \Cref{thm:R3b}.
 Therefore, we consider the assumptions of \Cref{num:R3b} without assuming condition (F).
 So $(E,v)$ is a discretely valued field and $F/E$ is any finite field extension.
 We let $A=\cO_v$ be the ring of integers of $(E,v)$ and $B$ be the integral closure
 of $A$ in $F$. 

As recalled in \emph{loc. cit.}, the maximal ideals of $B$
 are in bijection with the set $I$ of discrete valuations $w$ on $F$ that extend $v$.
 For such a valuation $w$, we let $e_w$ be the \emph{ramification index}, such that
 $w|_E=e_w.v$. We also let $f_w$ be the \emph{residual degree}, that is the degree
 of the induced extension of residue fields $\kappa_w/\kappa_v$.

We also let $A^h$ (resp. $\hat A$) be the henselization
 (resp. completion) of the discrete valuation ring $A$,
 and let $E^h$ (resp. $\hat E$) be its fraction field. To simplify,
 we abusively denote by $v$ the canonical valuations on $E^h$ and $\hat E$.
 Then, as $A^h/A$ is ind-étale,
 $B \otimes_A A^h$ is reduced, semi-local,
 and a product of discrete valuation rings indexed by $I$:
$$
B \otimes_A A^h=\prod_{w/v} B^h_w
$$
where $B^h_w$ is the localization of $B \otimes_A A^h$ at the prime corresponding
 to the valuation $w/v$. We let $F^h_w$ be the fraction field of $B^h_w$.
 It follows that $F^h_w/E^h$ is finite, and that $B^h_w$ is the ring of integers
 of the unique valuation $w^h$ on $F^h_w$ that extends $v$
 (see e.g. \cite[(2.8) Theorem]{FV}). Note that, by considering the total ring
 of fractions in the preceding identification, we also have the canonical identity:
$$
F \otimes_E E^h=\prod_{w/v} F^h_w.
$$
Moreover, $F^h_w$ coincides with the henselization of the discretely valued field
 $(F,w)$. One deduces the relation:
$$
[F:E]=\sum_{w/v} [F_w^h:E^h].
$$

On the other hand, the ring $B \otimes_A \hat A$ is non-reduced in general,
 as well as the ring $F \otimes_E \hat E$
 (see \cite[VI, \textsection 8, n°2, Prop. 2]{BouAC57}).
 It is still semi-local, and one has an identification (see \emph{loc. cit.}):
$$
B \otimes_A \hat A=\prod_{w/v} \hat B_{(w)}
$$
where $\hat B_{(w)}$ is the localization of the ring $B \otimes_A \hat A$
 at the prime corresponding to $w$.
 The reduction of $\hat B_{(w)}$ coincides with the complete discrete valuation ring
 $\hat \cO_w$, completion of the ring of integers of $(F,w)$.
 Letting $\hat F_w$ be the fraction field of $\hat \cO_w$, we also get the identification:
$$
(F \otimes_E \hat E)_{red}=\prod_{w/v} \hat F_w.
$$
We can now state the following result of valuation theory,
 which enlightens property (F) of \cref{num:R3b}.
\end{num}
\begin{prop}\label{prop:defect}
Consider the above notation. Given a valuation $w$ on $F$ which extends $v$,
 there exists a unique integer $d_w=p^{r_w}$, where $p$ is
 the characteristic exponent of $E$ and $r_w \geq 0$, such that
\begin{equation}\label{eq:prop:defect1}
[F^h_w:E^h]=d_w.e_w.f_w.
\end{equation}
Moreover, the following conditions are equivalent:
\begin{enumerate}
\item[(i)] $d_w=1$.
\item[(ii)] $B^h_w$ is a finite $A^h$-algebra.
\item[(iii)] $\hat B_{(w)}$ is reduced.
\end{enumerate}
Further, put
$$
C_w=\hat B_{(w)} \otimes_{\hat A} \hat E.
$$
Then $C_w$ is a local Artinian ring with residue field $\hat F_w$, and one has:
\begin{equation}\label{eq:prop:defect2}
d_w=\frac{[F^h_w:E^h]}{[\hat F_w:\hat E]}=\mathrm{lg}(C_w).
\end{equation}
\end{prop}
Following \cite{Kuhlmann}, we call $d_w$ the \emph{defect} of $(F,w)$ over $(E,v)$.
\begin{proof}
According to the preceding discussion, we can reduce to the case where
 $A=A^h$ is a henselian local ring, so that $(E,v)$ is a henselian
 discretely valued field and there is only one discrete valuation $w$
 on $F$ extending $v$. In particular, one has $F=F_w^h$, $E=E^h$.
 We put $\hat F=\hat F_w$, $C=F \otimes_E\hat E$, and write $e=e_w$ for the ramification index,
 $f=f_w$ for the residual degree.

We know from \cite[VI, \textsection 8, n°5, Th. 2/Cor. 2]{BouAC57} that:
 $[\hat F:\hat E]=ef$.
 Moreover, as $C$ is a local Artinian ring with residue field
 $\hat F$, one gets:
$$
\dim_E(F)=\dim_{\hat E}(C)=\mathrm{lg}(C).\dim_{\hat E}(\hat F).
$$
In particular, we can put $d_w=\mathrm{lg}(C)$, in order to get
 relation \eqref{eq:prop:defect1}. It follows that $d_w$ is a power of the characteristic
 exponent of $E$. Moreover, \eqref{eq:prop:defect2} follows from what was already said,
 and the fact that $C$ is the generic fiber of $\hat B_{(w)}$.
 The equivalence between (i) and (iii) is obvious according to this definition.
 The equivalence between properties (i) and (ii) follows
 from \cite[VI, \textsection 8, n°5, Th. 2]{BouAC57}.
\end{proof}

\begin{rem}
In particular, the defect $d_w$ for various $w$ over $v$,
 and various finite field extensions $F/E$
 measures the failure of the valuation ring $\cO_v$ to be Japanese.

Examples of discrete valuation rings with non-trivial defect
 are given in \cite[Ex. 2.31]{Liu}, \cite[Ex. 2.3.5]{Temkin},
 \cite[Tag 09E1]{Stack}.
\end{rem}

Based on the notion of defect of a finite extension of valuation rings,
 we can refine formula (R3b) of \Cref{thm:R3b} by removing the assumption (F).
\begin{thm}[Property \textbf{(R3b+)}]\label{prop:R3b+}
We consider a discretely valued field $(E,v)$
 and a finite field extension $F/E$. For any valuation $w$ on $F$ extending
 $v$, we let $d_w$ be its defect, as defined above.
 Then the following formula involving the basic maps (D2), (D3) and (D4) of Milnor-Witt K-theory
 holds:
$$
\partial_v \circ \Phi^*=\sum_{w/v} (d_w)_{\epsilon}.\Phi_w^* \circ \theta^w_* \circ \partial_w.
$$ 
\end{thm}
\begin{proof}
The formula is obtained by combining (R1c+) with respect to the field extensions $F/E$
 and $\hat E/E$, and (R3b) with respect to each extension $(\hat F_w,w)/(\hat E,v)$.
\end{proof}

\begin{rem}
After taking reduction modulo $\eta$, the previous formula is valid for Milnor K-theory,
 where one can replace the quadratic form $(d_w)_{\epsilon}$ by the integer $d_w$.
 In this form, it makes explicit Remark (1.8) of \cite{Rost}.
\end{rem}
 
\section{Appendix: coherent duality, traces and residues}\label{sec:Gduality}

\subsection{Categorical duality and traces}

\begin{num}
We recall the classical categorical framework for expressing duality.
 We refer the reader to \cite{McL}, Section XI.1 (resp. VII.7),
 for references on symmetric monoidal categories
 (resp. closed symmetric monoidal categories).
 To simplify the exposition, we will apply Mac Lane's coherence theorem (\emph{loc. cit.}, XI.1, Th. 1) and consider
 that any composite of coherence isomorphisms (\emph{i.e.} expressing unity, associativity, commutativity of the symmetric monoidal structure)
 is an identity.

The historical references for the next two definitions are \cite{DP} and \cite{Saav}.
\end{num}
\begin{df}\label{df:dualizable}
Let $(\C,\otimes,\un)$ be a symmetric monoidal category 
 and $M$ an object of $\C$.
 We say that $M$ is \emph{strongly dualizable}\footnote{Another terminology which appears in the context
 of the Tannakian formalism is \emph{rigid}.}
 if there exists an object $M^\vee$ and morphisms:
$$
\pair:M \otimes M^\vee \rightarrow \un, \copair:\un \rightarrow M^\vee \otimes M,
$$
respectively called the duality \emph{pairing} and \emph{co-pairing}, such that
 the following composite maps
\begin{align*}
M \xrightarrow{\Id \otimes \copair} M  \otimes M^\vee \otimes M \xrightarrow{\pair \otimes \Id} M \\
M^\vee \xrightarrow{\copair \otimes \Id} M^\vee \otimes M \otimes M^\vee \xrightarrow{\Id \otimes \pair} M^\vee
\end{align*}
are the identity. One says that $M^\vee$ is the \emph{dual} of $M$.
\end{df}
The triple $(M^\vee,\pair,\copair)$ uniquely determines $M^\vee$ as the dual of $M$.
 It follows from the definition that the functor $\tau_M=(M \otimes -)$ is both left and right adjoint
 to the functor $\tau_{M^\vee}=(M^\vee \otimes -)$. In particular, $\tau_{M^\vee}$ is the internal Hom functor with source $M$,
 and one gets a canonical isomorphism, bifunctorial in $M$ and $N$:
$$
M^\vee \otimes N \simeq \uHom(M,N).
$$
When the monoidal category $\C$ is closed, there is an isomorphism
 $M^\vee \simeq \uHom(M,\un)$, uniquely characterized as the evaluation at $\un$
 of the canonical isomorphism $\tau_{M^\vee} \simeq \uHom(M,-)$.

\begin{ex}[exercise]\label{ex:Amodules_dual}
Let $A$ be a (commutative) ring,
 and $\smod A$ be the closed symmetric monoidal category of $A$-modules.
 Then the following conditions are equivalent:
\begin{enumerate}[label=(\roman*)]
\item $M$ is strongly dualizable in $\smod A$;
\item $M$ is a finitely generated projective $A$-module.
\end{enumerate}
\end{ex}

\begin{ex}\label{ex:perfect}
The preceding example generalizes to a quasi-compact quasi-separated scheme $X$.
 Let $\Der(\cO_X)$ be the derived category of $\cO_X$-modules, endowed with its closed symmetric monoidal structure
 via the derived tensor product. Let $\Dqc(X)$ be the full sub-category of $\Der(\cO_X)$
 made of complexes with quasi-coherent cohomology. Let $K$ be an object of $\Der(\cO_X)$.
 Then the following conditions are equivalent
 (see for instance \cite[Ex. 0FPC, Lem. 0FPD, Prop. 09M1]{Stack}):
\begin{enumerate}[label=(\roman*)]
\item $K$ is strongly dualizable in $\Der(\cO_X)$;
\item $K$ is a perfect complex of $\cO_X$-modules;
\item $K$ belongs to $\Dqc(X)$ and is compact in $\Dqc(X)$.
\end{enumerate}
\end{ex}

\begin{df}
Consider the above setting and let $M$ be a strongly dualizable object with dual $(M^\vee,\pair,\copair)$.
 We define the \emph{trace} of an endomorphism $f:M \rightarrow M$ as the following element of
 the ring $\End_\C(\un)$:
$$
\tr_M(f):\un \xrightarrow{f'} M \otimes M^\vee \xrightarrow{\pair} \un
$$
where $f'$ is obtained by adjunction from $f$. This defines a map:
$$
\tr_M:\End_\C(M) \rightarrow \End_\C(\un).
$$
\end{df} 

\begin{rem}[exercise]
One can derive the following formulas for the traces just defined:
\begin{itemize}
\item $\tr_M(f \circ g)=\tr_M(g \circ f)$.
\item $\tr_{M \otimes N}(f \otimes g)=\tr_M(f) \otimes \tr_N(g)$.
\item $\tr_M(\lambda.f)=\lambda.\tr_M(f)$, $\lambda \in \End_\C(\un)$.
\end{itemize}
\end{rem}

\begin{ex}
Consider the setting of \Cref{ex:Amodules_dual}. Obviously, $\un=A$ and $\End_A(A)=A$, as a ring.

Given a strongly dualizable $A$-module $M$,
 the trace map $\Tr_M:\End_A(M) \rightarrow A$ defined above coincides
 with the classical notion in number theory.
 In particular, when $M$ admits a (global) $A$-basis $(f_1,\hdots,f_n)$, through the induced isomorphism $\End_A(M) \simeq \mathcal M_n(A)$,
 the map $\Tr_M$ is the usual trace map of matrices.
\end{ex}

\begin{num}
Consider again the abstract situation of a symmetric monoidal category $(\C,\otimes,\un)$,
 and a strongly dualizable object $M$ of $\C$ with dual $(M^\vee,\pair,\copair)$.

We remark that the trace map $\tr_M$ is induced by an internal trace map:
$$
\underline{\tr}_M:\uHom(M,M) \simeq M^\vee \otimes M \simeq M \otimes M^\vee \xrightarrow{\pi} \un.
$$
This means that $\tr_M=\Hom_\C(\un,\underline{\tr}_M)$.

Assume now that $M$ admits a product map $\mu:M \otimes M \rightarrow M$
 (for example, $M$ is a monoid, \cite[VII.3]{McL}).
 Then one gets a $\mu$-trace morphism:
$$
\Tr^\mu_M:M \xrightarrow{\mu'} \uHom(M,M) \xrightarrow{\underline{\tr}_M} \un
$$
As a particular case, one gets back the following classical definition from algebra:
\end{num}
\begin{df}\label{df:algebraic_trace}
Let $A$ be a ring and $B$ be a commutative $A$-algebra which is projective and finitely generated as an $A$-module.

Then $B$ is a strongly dualizable $A$-module and we define the \emph{trace morphism} 
$$\Tr_{B/A}:B \rightarrow A$$
 as the $A$-linear map associated above with respect to the multiplication map
 $B \otimes_A B \rightarrow B$.
\end{df}
Concretely, the trace of an element $b \in B$ is the trace of the endomorphism
 $\gamma_b$ such that $\gamma_b(x)=b.x$. It can be computed locally by choosing bases of
 the $A$-module $B$ and using the trace of matrices. The local definitions then glue
 using faithfully flat descent.

Let us recall the following classical result.
\begin{prop}
Let $B/A$ be a finitely generated projective ring extension. Then the following conditions are equivalent:
\begin{enumerate}[label=(\roman*)]
\item $B/A$ is étale.
\item For every prime ideal $\mathfrak p$ in $A$, the $\kappa(\mathfrak p)$-algebra
 $B \otimes_A \kappa(\mathfrak p)$ is reduced and, for every prime ideal $\mathfrak q$ in $B$
 lying above $\mathfrak p$, one has: $\Tr_{\kappa(\mathfrak q)/\kappa(\mathfrak p)}\neq 0$.
\item The bilinear form $B \otimes_A B \rightarrow A, x \otimes y \mapsto \Tr_{B/A}(xy)$ is
 non-degenerate --- \emph{i.e.} induces by adjunction an isomorphism $B \rightarrow B^\vee$ of $A$-modules.
\end{enumerate}
In particular, a finite field extension $L/K$ is separable if and only if $\Tr_{L/K} \neq 0$.
\end{prop}

Therefore, the above notion of trace map is inadequate for inseparable field extensions,
 since it yields the zero map.
 This justifies the use of a finer duality theory, which was introduced by Grothendieck.
 We recall the abstract setting to end this section.
\begin{df}
Let $(\C,\otimes,\un)$ be a closed symmetric monoidal category.
 Let $K$ be an object of $\C$, and write $\mathrm D_K(M)=\uHom(M,K)$.
 The evaluation map $M \otimes \uHom(M,K) \rightarrow K$ induces by adjunction
 a canonical map $\omega_M:M \rightarrow \mathrm D_K \circ \mathrm D_K(M)$.

One says that $K$ is \emph{dualizing} if the natural transformation
 $\omega:\Id_\C \rightarrow \mathrm D_K \circ \mathrm D_K$ is an isomorphism.
\end{df}
In words, $\mathrm D_K(M)$ is called the weak $K$-dual of $M$,
 and the definition asks that any object $M$ is isomorphic to
 its double weak $K$-dual, by the canonical map $\omega_M$.

\begin{rem}
\begin{enumerate}
\item In the original definition of a dualizing complex (\cite[Definition~p.~258]{HartRD}),
 one had additional assumptions (finite injective dimension and lower boundedness).
 One has progressively dismissed these assumptions, in order to extend
 Grothendieck's theory to other contexts (torsion \'etale sheaves \cite[VII, 6.1.1]{Gabber},
 constructible pro-\'etale sheaves \cite[6.7.20]{BSproet}, $D$-modules,
 motivic homotopy \cite[2.3.73]{Ayoub1} and motivic complexes \cite[4.4.24]{CD3}).
\item The seminal definition of Grothendieck has been extended in several directions.
 We refer the reader to \cite{BoyDrin} for further developments.
\end{enumerate}
\end{rem}

\begin{num}
Consider a dualizing object $K$ of $\C$ as in the above definition.
 Then one has the following basic properties:
\begin{enumerate}
\item The map $\un \rightarrow \uHom(K,K)$, deduced from $\Id_K$ by adjunction, is an isomorphism.
\item For any object $M$, $N$ in $\C$, one has an isomorphism:
$$
\mathrm D_K\big(M \otimes \mathrm D_K(N)\big) \simeq \uHom(M,N).
$$
\item An object $K'$ of $\C$ is dualizing if and only if there exists a $\otimes$-invertible object $L$ such that $K'= K \otimes L$. \\
 Moreover, in this case, one has $L \simeq \mathrm D_{K'}(K)=\uHom(K,K')$.
\item If $M$ is a strongly dualizable object in $\C$ with dual $M^\vee$, then $\mathrm D_K(M) \simeq M^\vee \otimes K$.
\item The dualizing object $K$ is strongly dualizable if and only if it is invertible.
\end{enumerate}
We give the arguments for completeness:
\begin{enumerate}
\item use the isomorphism $\un \simeq D_K(D_K(\un))$ and $\uHom(\un,-) \simeq \Id_\C$ (as a right adjoint to the functor $(\un \otimes -)$).
\item Use the sequence of isomorphisms:
$$
\mathrm D_K\big(M \otimes \mathrm D_K(N)\big)
 =\uHom\big(M \otimes \mathrm D_K(N),K\big)
 \simeq \uHom\big(M,\mathrm D_K\mathrm D_K(N)\big)
 \simeq \uHom(M,N).
$$
\item $\Leftarrow$: use $\uHom(M,N \otimes L) \simeq \uHom(M  \otimes L^{-1},N) \simeq \uHom(M,N) \otimes L$. \\
$\Rightarrow$: one reduces to proving that $M \mapsto M \otimes \mathrm D_K(K')$ is an equivalence.
 It suffices to apply the equivalence
 $\mathrm D_K$, point (2) to reduce to the fact that $D_{K'}$ is an equivalence.
\item Follows from definitions.
\item Follows from point (4).
\end{enumerate}
\end{num}

\begin{ex}
In the category of locally compact abelian groups, the unit circle $\mathbb T=\RR/\ZZ$ is a dualizing object.
 We will see more examples in the next section.
\end{ex}

\subsection{Grothendieck differential trace map and duality}

\begin{num}
Let $f:X \rightarrow S$ be a morphism of quasi-compact and quasi-separated schemes.
 We have an adjoint pair:
$$
\derL f^*:\Der(\cO_S) \leftrightarrows \Der(\cO_X):\derR f_*.
$$
We say that a complex $K$ of $\Der(\cO_X)$ is quasi-coherent if its cohomology
 sheaves are quasi-coherent. We let $\Dqc(X)$ be the full sub-category of $\Der(\cO_X)$
 made of quasi-coherent complexes.\footnote{For noetherian schemes,
 this category is equivalent to the derived category of the abelian category
 of quasi-coherent sheaves:
 see \cite[Tag 09T4]{Stack}.} Both functors $\derL f^*$ and $\derR f_*$ preserve
 quasi-coherent complexes (see \cite[Tags 08DW and 08D5]{Stack}).
 The following theorem is one of the essential parts of Grothendieck's duality theory.
\end{num}
\begin{thm}
Consider the above assumptions.
\begin{enumerate}
\item The functor $\derR f_*:\Dqc(X) \rightarrow \Dqc(S)$ admits a right adjoint.
 When $f$ is \emph{in addition proper}, we denote this right adjoint by $f^!:\Dqc(S) \rightarrow \Dqc(X)$.
\item Assume in addition that $S$ is noetherian and $f$ is finite syntomic.
 There exists a canonical isomorphism:
$$
\pur_f:\omega_{X/S} \rightarrow f^!(\cO_S).
$$
It is called the \emph{purity isomorphism} associated with $f$.
\end{enumerate}
\end{thm}
\begin{proof}
The first statement is Neeman's theorem (see \cite{NeeG}, \cite[Tag 0A9E]{Stack}).

Let us prove point (2). We first recall the construction that we will use.
 Given a locally quasi-finite syntomic morphism $g:Y\rightarrow T$ of locally
 noetherian schemes, the Stacks Project constructs a canonical isomorphism
 from the determinant of the naive cotangent complex to the relative dualizing
 module of $g$. Following the terminology of \cite[Tag 0FKB]{Stack}, this
 isomorphism is called the \emph{Tate map}; its existence, uniqueness and
 compatibility with base change are proved in \cite[Tag 0FKD]{Stack}.

We apply this construction to $f$. As $f$ is finite syntomic, it is finite flat
 and lci. By \cite[Tag 08UW]{Stack}, the naive cotangent complex is canonically
 identified with the truncation $\tau_{\geq -1}L_{X/S}$ of the cotangent complex.
 Since $f$ is lci, $L_{X/S}$ has amplitude in $[-1,0]$, and therefore the
 determinant occurring in the Tate map is canonically identified with
$$
\det(L_{X/S})=\omega_{X/S},
$$
taking into account our homological convention. Thus the Tate map gives a
 canonical isomorphism from $\omega_{X/S}$ to the relative dualizing module of $f$.

Finally, as $f$ is finite flat, this relative dualizing module is canonically
 identified with $f^!(\cO_S)$, which is concentrated in degree $0$;
 see \cite[Tags 0BUL and 0BVI]{Stack}. The resulting composite is the
 purity isomorphism $\pur_f$.
\end{proof}

\begin{rem}
\begin{enumerate}
\item For a finite morphism, the assumption ``syntomic'' is equivalent to being lci
 of virtual relative dimension $0$. Indeed, a finite syntomic morphism is finite flat lci
 and its cotangent complex has virtual rank $0$; conversely, a finite lci morphism
 of virtual relative dimension $0$ is locally given by a complete intersection presentation
 with the same number of equations as variables, and is therefore syntomic
 by \cite[Tag 0FKY]{Stack}.
\item The Tate isomorphism used above is in fact constructed more generally for locally quasi-finite
 syntomic morphisms of locally noetherian schemes, and is compatible with base change
 (see \cite[Tag 0FKD]{Stack}). We restrict ourselves to finite morphisms, which is the case
 needed below.

In fact, in the finite syntomic case the purity isomorphism can be generalized in the coherent
 context as follows. For any quasi-coherent complex $K$ over $S$, there is a canonical isomorphism:
$$
f^!(K) \simeq f^*(K) \otimes \omega_{X/S},
$$
which follows from finite duality (see \cite[Tag 0FKW]{Stack}).
 Note this is specific to the coherent case. The analogous isomorphism does not hold in other
 six functors formalisms such as the \'etale $\ell$-adic or motivic ones, unless further
 restrictions on $f$ are assumed (e.g., $f$ is smooth, or a nil-immersion).
\end{enumerate}
\end{rem}

\begin{df}\label{df:diff_trace}
Assume that $S$ is noetherian and $f:X \rightarrow S$ is finite syntomic.
 We will denote by
$$
\Tr^\omega_f:\derR f_*(\omega_{X/S}) \rightarrow \cO_S
$$
the map in $\Dqc(S)$ obtained by adjunction from $\pur_f$ and call it
 the \emph{differential trace map} associated with $f$.

As $f$ is finite, the source and target of $\Tr^\omega_f$ are concentrated in degree $0$.
 Therefore, it gives a canonical morphism of quasi-coherent $\cO_S$-modules,
 and we will consider it as such. Taking global sections, we will also consider
 the induced trace map:
$$
\Tr^\omega_{X/S}:\Gamma(X,\omega_{X/S}) \rightarrow \Gamma(S,\cO_S).
$$
Finally, if $X/S$ is the spectrum of a finite syntomic ring extension $B/A$,
 with $A$ noetherian, the above map will be denoted by:
$$
\Tr^\omega_{B/A}:\omega_{B/A} \rightarrow A.
$$
In that case, the Tate isomorphism gives a canonical identification
$$
\omega_{B/A} \xrightarrow \sim \Hom_A(B,A),
$$
and $\Tr^\omega_{B/A}$ corresponds to evaluation at $1$ (see \cite[Tag 0FKW]{Stack}).
\end{df}

\begin{rem}\label{rem:diff_trace}
The above definition clarifies several properties of the differential trace map.
 It is compatible with arbitrary base change between noetherian schemes,
 and compatible with disjoint sums in $X$. This follows from the base change property
 of the Tate isomorphism (see \cite[Tag 0FKD]{Stack}), together with the description
 of the differential trace as evaluation at $1$; for disjoint sums, see also
 \cite[Tag 0BT5]{Stack}.

This means in particular that if we have an isomorphism of
 finite syntomic $A$-algebras, with $A$ noetherian:
$$
\Theta:B \xrightarrow \sim \prod_{i \in I} B_i,
$$
the following diagram is commutative:
$$
\xymatrix@=15pt{
\omega_{B/A}\ar^-{\Theta_*}_-\sim[rr]\ar_{\Tr^\omega_{B /A}}[rd] && \prod_{i \in I} \omega_{B_i/A}\ar^{\prod_i \Tr^\omega_{B_i/A}}[ld] \\
& A &
}
$$
\end{rem}

\begin{rem}\label{rem:diff_trace_composition}
The compatibility with composition of the trace map is more involved.
 Consider a factorization $X \xrightarrow g Y \xrightarrow h S$ of $f$
 by finite syntomic morphisms of noetherian schemes.
 First recall that there exists a canonical isomorphism (see \Cref{eq:can_iso_can_sheaf}):
$$
\psi:\omega_{X/S} \simeq \omega_{X/Y} \otimes (g^*\omega_{Y/S}).
$$
The Tate isomorphisms used in the preceding proof are compatible with this transitivity isomorphism.
 Indeed, this can be checked locally by choosing complete intersection presentations
 for $g$ and $h$, and concatenating them to obtain one for $f$.
 The local description of the determinant and its canonical section
 in \cite[Tag 0FKC]{Stack} shows that, through $\psi$, the canonical section
 associated with $f$ corresponds to the tensor product of those associated
 with $g$ and $h$.
 On the duality side, \cite[Tag 0BT4]{Stack} gives a canonical composition map
 between the relative dualizing modules. In the finite flat case it is an isomorphism
 and, after the symmetry isomorphism, it is given for a tower of finite syntomic rings
 $A \rightarrow B \rightarrow C$ by:
$$
\Hom_B(C,B) \otimes_B \Hom_A(B,A) \longrightarrow \Hom_A(C,A),
\qquad u \otimes v \longmapsto v \circ u.
$$
 Under this map, the trace sections correspond by the transitivity of the usual trace
 of finite locally free algebras. The characterization of the Tate isomorphism
 in \cite[Tag 0FKD]{Stack} therefore gives the required compatibility.

It follows that the compatibility with composition of the differential trace map is
 expressed by the following commutative diagram (again we discard the symbols $\derR$ and
 $\derL$ for readability):
$$
\xymatrix@R=10pt@C=80pt{
f_*(\omega_{X/S})\ar^-{\Tr^\omega_{X/S}}[rr]\ar^\sim_{f_*\psi}[d] & & \cO_S\ar@{=}[dd] \\
f_*\big(\omega_{X/Y} \otimes (g^*\omega_{Y/S})\big)\ar_\sim[d] && \\
h_*\big(g_*(\omega_{X/Y}) \otimes \omega_{Y/S}\big)
 \ar^-{h_*\big(\Tr^\omega_{X/Y}\otimes \Id\big)}[r] &
 h_*(\omega_{Y/S})\ar^-{\Tr^\omega_{Y/S}}[r] & \cO_S
}
$$
The second vertical map is obtained by the projection formula, which holds here
 as $g$ is finite.

In the affine case, $X=\spec C$, $Y=\spec B$, $S=\spec A$, $f$, $g$ and $h$ being finite
 syntomic morphisms, the diagram takes the following simpler form:
$$
\xymatrix@R=10pt@C=80pt{
\omega_{C/A}\ar^-{\Tr^\omega_{C/A}}[rr]\ar^\sim_\psi[d] & & A\ar@{=}[d] \\
\omega_{C/B} \otimes_B \omega_{B/A}\ar^-{\Tr^\omega_{C/B}\otimes \Id}[r] &
 \omega_{B/A}\ar^-{\Tr^\omega_{B/A}}[r] & A.
}
$$
\end{rem}

\begin{num}\textit{A particular case of duality}. \label{num:trivial_duality}
For any quasi-coherent complex $K$ over $X$,
 and any finite syntomic morphism $f:X \rightarrow S$ with $S$ noetherian,
 the adjunction property of the pair $(\derR f_*,f^!)$ gives an isomorphism:
\begin{align*}
\Hom_{\Der(\cO_X)}(K,\omega_{X/S}) & \xrightarrow \sim \Hom_{\Der(\cO_S)}(\derR f_*(K),\cO_S), \\
 \big(u:K \rightarrow \omega_{X/S}\big) & \mapsto \big(\Tr_{X/S}^\omega \circ \derR f_*(u)\big).
\end{align*}
In the case of a finite syntomic ring extension $B/A$, with $A$ noetherian,
 and for a $B$-module $M$, this boils down to an isomorphism:
\begin{align*}
\Hom_B(M,\omega_{B/A}) & \xrightarrow \sim \Hom_A(M,A), \\
 \big(u:M \rightarrow \omega_{B/A}\big) & \mapsto \big(\Tr_{B/A}^\omega \circ u\big).
\end{align*}
Taking $M=B$, we get an isomorphism between $A$-linear forms on $B$
 and elements of $\omega_{B/A}$:
\begin{align*}
\omega_{B/A} & \xrightarrow \sim \Hom_A(B,A), \\
w & \mapsto \big(\psi_w:\lambda \mapsto \Tr_{B/A}^\omega(\lambda.w)\big).
\end{align*}
This is the affine form of the Tate isomorphism.
\end{num}

\begin{ex} We end this section with a few classical examples
 of duality in the case of coherent sheaves.
\begin{enumerate}[wide]
\item A concrete case of duality is obtained when $S$ is the spectrum of any field $k$,
 and $X$ is a proper smooth $k$-scheme of pure dimension $d$.
 In that case, there is a canonical purity isomorphism
$$
\omega_{X/k}[d] \xrightarrow \sim f^!(k)
$$
(see \cite[Tag 0BRT]{Stack}). Applied to $K[n]$, where $K$ is a bounded complex
 with coherent cohomology, the adjunction gives an isomorphism of $k$-vector spaces:
$$
\Ext^{d-n}_{\cO_X}(K,\omega_{X/k}) \xrightarrow \sim H^n(X,K)^*.
$$
The counit similarly gives a trace map
$$
\Tr_{X/k}^\omega:\derR f_*(\omega_{X/k})[d] \rightarrow k,
$$
and therefore what I will call the \emph{Gysin map} associated with $f$:
$$
f_!:H^d(X,\omega_{X/k}) \rightarrow k.
$$
The above duality isomorphism is induced by the \emph{Poincaré duality} (perfect) pairing:
\begin{align*}
\Ext^{d-n}_{\cO_X}(K,\omega_{X/k}) \otimes H^n(X,K) &\rightarrow k, \\
(x,y) &\mapsto f_!(x.y).
\end{align*}
\end{enumerate}
\end{ex}

\begin{rem}
In the case of a proper smooth morphism $f:X \rightarrow S$ of pure relative dimension $d$,
 with $S$ noetherian, the same purity isomorphism (see \cite[Tag 0BRT]{Stack})
 allows one to interpret Grothendieck duality, for $K=\cO_X$, by saying that
 $\derR f_*(\cO_X)$ is strongly dualizable (\Cref{df:dualizable}, \Cref{ex:perfect})
 with dual given by $\derR f_*(\omega_{X/S})[d]$.
 Here $\derR f_*(\cO_X)$ is perfect by \cite[Tag 0B6F]{Stack}.
 One of the pairings coming from this duality is a relative version of the Poincaré duality pairing:
$$
\derR f_*(\cO_X) \otimes \derR f_*(\omega_{X/S})[d]
 \rightarrow \derR f_*(\omega_{X/S})[d] \xrightarrow{\Tr_{X/S}^\omega} \cO_S,
$$
where the first map comes from the fact $\derR f_*$ is weakly monoidal
 (as the right adjoint of a monoidal functor).

The theory can of course be considerably generalized, although we will only need
 the case of finite field extensions. Under the usual noetherian finite type assumptions,
 the purity isomorphisms considered above are particular cases of the canonical
 isomorphism
$$
f^!(\cO_S) \simeq \omega_{X/S}[d]
$$
associated with a proper lci morphism $f$ of virtual relative dimension $d$.
 This can be deduced from Grothendieck's formulation of duality as exposed in \cite{HartRD};
 see, e.g., \cite[Prop. B.1]{BW}.
 For an arbitrary proper morphism $f:X\rightarrow\Spec(k)$, Grothendieck duality
 instead provides a dualizing complex
$$
K_X=f^!(k),
$$
which is not in general an invertible object of $\Der(\cO_X)$
 (see for example \cite[Tag 0FVV]{Stack}); it is invertible when $X$ is Gorenstein,
 see \cite[V, 9.3]{HartRD}.
 This is the original framework of Grothendieck duality developed in
 \cite{HartRD}; see also \cite{ConradDual} or \cite[Chap. 1]{LH}.
\end{rem}

\subsection{Grothendieck and Scheja-Storch Residues}\label{sec:G-SS-residues}

\begin{num}\label{num:scheme_conditions_res}
In this section, base schemes and rings are assumed to be noetherian.

We consider a commutative diagram of schemes:
$$
\xymatrix@=10pt{
& P\ar^p[rd] & \\
X\ar^i[ru]\ar_f[rr] & & S
}
$$
such that $f$ is finite syntomic,
 $p$ is smooth of relative dimension $n$, and $i$ is a closed immersion
 with ideal sheaf $\mathcal I \subset \cO_P$.
 The hypotheses imply that $i$ is regular of codimension $n$.

Recall that we can associate to the above commutative diagram
 a canonical isomorphism (see paragraph \Cref{eq:can_iso_can_sheaf}):
$$
\Theta:\omega_{X/S} \simeq \omega_{X/P} \otimes_{\cO_X} i^*\omega_{P/S}
 \simeq \big(\Lambda^n(\cI/\cI^2)\big)^\vee \otimes_{\cO_P} \Omega^n_{P/S}.
$$
\end{num}
\begin{df}\label{df:Grothendieck_residues}
Consider a global differential $n$-form $w \in \Gamma(P,\Omega^n_{P/S})$
 and a global regular parametrization $(f_1,\hdots,f_n)$ of $\cI$.
 We define the Grothendieck residue (symbol) of $w$ at $(f_1,\hdots,f_n)$ as the element of
 $\Gamma(S,\cO_S)$:
$$
\Res_{P/S}\left\lbrack \begin{matrix}w \\ f_1 \hdots f_n\end{matrix}\right\rbrack
 = \Tr_{X/S}^\omega\big( (\bar f_1 \wedge \hdots \wedge \bar f_n)^* \otimes i^*(w) \big)
$$
where we have used the differential trace map of \Cref{df:diff_trace}
 and we have considered the element $(\bar f_1 \wedge \hdots \wedge \bar f_n)^* \otimes i^*(w)$
 as an element of $\omega_{X/S}$ via the isomorphism $\Theta$.
\end{df}
This definition agrees with that of \cite[III, \textsection 9]{HartRD}
 and that of \cite[Appendix A, (A.1.4)]{ConradDual}.\footnote{The sign in the latter
 can be explained as:
$$(\bar f_1 \wedge \hdots \wedge \bar f_n)^*
=\bar f_n^* \wedge \hdots \wedge \bar f_1^*
=(-1)^{n(n-1)/2} \bar f_1^* \wedge \hdots \wedge \bar f_n^*.$$ }

\begin{num}\label{num:affine_conditions_res}
We now explain a method of Scheja and Storch to compute the above residue,
 and therefore the differential trace map.
 Our reference is \cite[\textsection 8]{Kunz}.

We will restrict to the affine case.
 Let $B$ be a finite projective $A$-algebra: in other words,
 $B$ is strongly dualizable as an $A$-module, see \Cref{ex:Amodules_dual}.

In what follows, we will use a set of indeterminates $\underline t=(t_1,\hdots,t_n)$,
 and put $A[\underline t]=A[t_1,\hdots,t_n]$ for brevity.
 We assume that $B$ is a complete intersection $A$-algebra:
 there exist elements $\alpha=(\alpha_1,\hdots,\alpha_n) \in B^n$
 which generate $B$ as an $A$-algebra and the kernel of the surjective map
$$
A[\underline t] \rightarrow B, t_i \mapsto \alpha_i
$$
admits a regular parametrization $I=(f_1,\hdots,f_n)$ for polynomials
 $f_i \in A[\underline t]$. We say that $f=(f_1,\hdots,f_n)$ is a \emph{presentation}
 of the lci $A$-algebra $B$.
 As $B$ is finite projective and lci over $A$, the morphism $\Spec(B)\rightarrow\Spec(A)$ is finite syntomic
 (see \cite[Tag 0FKY]{Stack}).

Putting $S=\Spec(A)$, $X=\Spec(B)$, $P=\AA^n_S=\Spec(A[\underline t])$,
 we are therefore in the conditions of paragraph \Cref{num:scheme_conditions_res}.
 Consider the composite map\footnote{Geometrically,
 this map corresponds to the graph $\gamma_i:X \rightarrow X \times_S P$ of the
 closed immersion $i:X \rightarrow P$. As $i$ is regular and $P/S$ is smooth,
 $\gamma_i$ is regular. Algebraically, it is just the map evaluating $t_i$ at $\alpha_i$.}:
$$
\sigma:B[\underline t] \rightarrow B \otimes_A B \xrightarrow \mu B
$$
where $\mu$ is the multiplication map, and the first arrow is the natural surjection
 coming from the identification $B \otimes_A B=B \otimes_A A[\underline t]/I=B[\underline t]/I$.
 We consider the ideals:
\begin{align*}
J&=\Ker(\mu) \subset B \otimes_A B, \\
 K&=\Ker(\sigma) \subset B[\underline t].
\end{align*}
Moreover, $K$ admits a regular parametrization, $K=(t_1-\alpha_1,\hdots,t_n-\alpha_n)$
 and one obtains the identification $J=K/I$ as obviously $I \subset K$ as ideals of $B[\underline t]$.
 Therefore, there exist polynomials $c_{ij} \in B[\underline t]$ such that
$$
\forall i \in [1,n], f_i=\sum_{j=1}^n c_{ij}.(t_j-\alpha_j).
$$
Finally, the element
$$
\Delta_f=\det\big(c_{ij}(\alpha)_{1 \leq i,j \leq n}\big) \in B \otimes_A B
$$
is independent of the chosen polynomials $c_{ij}$ (see \cite[Lemma 4.10]{Kunz}).
\end{num}
\begin{df}\label{df:bezoutian}
Consider the above notation: $B/A$ is a finite projective complete intersection
 and $f=(f_1,\hdots,f_n) \in A[\underline t]^n$ is a fixed presentation of $B/A$.
 Then the element $\Delta_f \in B \otimes_A B$ is called the \emph{Bézoutian}
 associated with the presentation $f$ of $B/A$.
\end{df}

\begin{rem}
This definition is of course an extension of the classical
 Bézoutian (or rather the determinant of the Bézout matrix) arising
 from Euler and Bézout elimination theory, which
 corresponds to the case where $A=k$ is a field and $n=2$.
\end{rem}

\begin{num}
Consider again the setting of \Cref{num:affine_conditions_res}.
 We put $B^*=\Hom_A(B,A)$, which is the (canonical) dual of the
 strongly dualizable $A$-module $B$ (see \Cref{df:dualizable} and what follows).\footnote{Following the usage,
 we identify the set of morphisms $\Hom_A$ with the internal Hom-functor in the category of $A$-modules.}
 As $B$ is strongly dualizable, the canonical map:
\begin{align*}
\Phi:B \otimes_A B & \rightarrow \Hom_A(B^*,B), \\
b \otimes b' & \mapsto (\varphi \mapsto \varphi(b).b')
\end{align*}
is an isomorphism. The following lemma is now a formality
 (see \cite[8.13]{Kunz}; beware to translate the notation:
 $I$ (resp. $\omega_{B/A}$) in \emph{loc. cit.} is what we denote by $J$ (resp. $B^*$) here.)
\begin{lm}
Consider the above notation. Recall that $J=\Ker(B \otimes_A B \xrightarrow \mu B)$,
 seen as an ideal of $B \otimes_A B$.
 Then $\Phi$ induces an isomorphism:
$$
\mathrm{Ann}(J) \xrightarrow \sim \Hom_B(B^*,B).
$$
\end{lm}
According to \cite[Cor. 4.12]{Kunz}, $\mathrm{Ann}(J)$ is the principal ideal generated by $\Delta_f$.
 As $B^*$ is an invertible $B$-module, the preceding lemma shows that
 $\Phi(\Delta_f):B^* \rightarrow B$ is an isomorphism. Therefore, there exists a unique
 $A$-linear map $\tau_f:B \rightarrow A$, equivalently $\tau_f \in B^*$, such that:
\begin{equation}\label{eq:Bézout&trace}
\Phi(\Delta_f)(\tau_f)=1_B.
\end{equation}
\end{num}
\begin{df}\label{df:SS_trace}
Consider the above notation, as in \Cref{df:bezoutian}.
 We call the $A$-linear map $\tau_f:B \rightarrow A$
 the \emph{Scheja-Storch trace map} associated with the presentation $f$ of $B/A$.
\end{df}

\begin{ex}\label{ex:Scheja-Storch_trace_monogeneous}
We consider the \emph{monogenic} case: 
$$B=A[\alpha]=A[t]/(f)$$
 where $f$ is a monic polynomial in one variable $t$:
$$
f(t)=a_0+\cdots+a_{n-1}.t^{n-1}+t^n.
$$
Thus $B$ is a free $A$-module with basis $1, \alpha, \hdots, \alpha^{n-1}$.
Then one can compute $\Delta_f$ explicitly and one finds that
$$
\tau_f=(\alpha^{n-1})^*, \alpha^i \mapsto \delta_{n-1}^i.
$$
\end{ex}

\begin{rem}\label{rem:Scheja-Storch_trace_monogeneous}
In the above example, the map $\tau_f$ does depend on the chosen generator $\alpha$ of $B/A$,
 or more explicitly on the chosen presentation of $B/A$.
 Therefore, it is sometimes customary to put:
$$
\tau_{B/A}^\alpha=\tau_f.
$$
In view of \cite[\textsection 1, (2)]{Tate}, corresponding to the case where $B/A$ 
 is an inseparable extension field, the map $\tau_f$ is sometimes
 called the \emph{Tate trace map} (cf. \cite{Kunz}).
\end{rem}

\begin{thm}
Consider the assumption of the above definition.

Then the $A$-linear map $\tau_f:B \rightarrow A$ is not $B$-torsion, 
 and in fact is a $B$-basis of $\Hom_A(B,A)$.

In other words, the symmetric bilinear form
$$
\varphi_f:B \otimes_A B \rightarrow A, b \otimes b' \mapsto \tau_f(bb')
$$
is non-degenerate: the associated map
$$
B \rightarrow B^*=\Hom_A(B,A), b \mapsto b.\tau_f=\varphi_f(b,-)
$$
is an isomorphism.
\end{thm}
\begin{proof}
Using the definitions of \Cref{num:affine_conditions_res},
 one obtains that $\mathrm{Ann}(J)$ is a principal ideal generated by the Bézoutian $\Delta_f$,
 a result attributed to Wiebe; see \cite[Cor. 4.12]{Kunz}.
 According to the previous lemma, $\mathrm{Ann}(J) \simeq \Hom_B(B^*,B)$ is also an invertible $B$-module.
 So $\Delta_f$ is a $B$-basis of the $B$-module $\mathrm{Ann}(J)$.
 Relation \eqref{eq:Bézout&trace} then implies that $\tau_f$ is a $B$-basis of $B^*$ as expected.
 The other assertions are formal consequences of this fact.
\end{proof}

We are now ready to state the link between the concrete construction
 of Scheja and Storch and the theory of Grothendieck residue symbols
 (\Cref{df:Grothendieck_residues}).
\begin{prop}\label{prop:compute_residues_SS}
Recall the situation of the previous theorem and definition:
\begin{itemize}
\item $B$ is a complete intersection, finite and projective $A$-algebra
\item $f$ is a presentation of $B/A$: $f=(f_1,\hdots,f_n)$ is a regular sequence of elements of $R=A[t_1,\hdots,t_n]$,
 $I=(f_1,\hdots,f_n)$ and $B=R/I$.
\end{itemize}
Recall that we have a canonical isomorphism (see \Cref{ex:affine_complete_intersection})
$$
\Theta:\omega_{B/A} \simeq \Lambda^n(I/I^2)^\vee \otimes_R \Omega^n_{R/A}.
$$
Then, for any $\lambda \in R$, with image $\bar \lambda$ in $B=R/I$, we get:
$$
\Res_{R/A}\left\lbrack \begin{matrix} \lambda.dt_1\wedge\hdots\wedge dt_n  \\ f_1 \cdots f_n\end{matrix}\right\rbrack
 =\tau_f(\bar \lambda).
$$
\end{prop}
In other words, if we let $w=(\bar f_1 \wedge \hdots \wedge \bar f_n)^* \otimes i^*(dt_1 \wedge\hdots \wedge dt_n)$
 seen as an element of $\omega_{B/A}$ via the isomorphism $\Theta$, for any $b \in B$, one gets:
\begin{equation}\label{eq:compute_residues_SS}
\Tr^\omega_{B/A}(b.w)=\tau_f(b).
\end{equation}
Or equivalently, with the notation of \Cref{num:trivial_duality}: $\psi_w=\tau_f$.
\begin{proof}
In the case where $A=k$ is a field (the only case we will need!), this is \cite[Prop. 8.32]{Kunz}.
 More generally, if $A$ is reduced, one can reduce to the field case by base change.
 Indeed, both traces $\Tr_{B/A}^\omega$ and $\tau_f$ are compatible with arbitrary
 base change. Thus, after base change to any residue field $\kappa(s)$ of $S$,
 they coincide by the field case. Their difference therefore takes values in the
 nilradical of $A$, and hence vanishes as $A$ is reduced.

Alternatively, both definitions of residues, via respectively Grothendieck and Scheja-Storch methods,
 satisfy the properties (R1)-(R10) of \cite[III, \textsection 9]{HartRD}
 (see respectively \cite[Appendix A]{ConradDual} and \cite{HLRes, HopkinsRes}).
 This uniquely characterizes the residue symbol.
\end{proof}

In view of \Cref{ex:Scheja-Storch_trace_monogeneous},
 we deduce the following comparison of the Grothendieck differential trace map
 and the Tate trace map (\Cref{rem:Scheja-Storch_trace_monogeneous}).
\begin{cor}\label{cor:compute_residues_SS}
Suppose $B/A$ is a monogenic extension ring, of the form $B=A[t]/I$ where
 $I=(f)$ for a monic polynomial $f \in A[t]$.
 We identify $\omega_{B/A}$ with the $B$-module $(I/I^2)^* \otimes_{A[t]} \Omega_{A[t]/A}$
 (via the isomorphism $\Theta$ of \Cref{ex:affine_complete_intersection}).

Then for any $b \in B$, one gets:
\begin{equation}
\Tr^\omega_{B/A}(b.\bar f^* \otimes dt)=\tau_{B/A}^\alpha(b)
\end{equation}
with the notation of \Cref{rem:Scheja-Storch_trace_monogeneous}.
\end{cor}

\begin{cor}\label{cor:Tr^omega_etale}
Let $B$ be a finite \'etale $A$-algebra.
 Then $\omega_{B/A}=B$ and the following diagram commutes:
$$
\xymatrix@R=-3pt@C=40pt{
\omega_{B/A}\ar@{=}[dd]\ar^{\Tr^\omega_{B/A}}[rd] & \\
& A \\
B\ar_{\Tr_{B/A}}[ru] &
}
$$
where $\Tr_{B/A}$ is the ``usual'' trace map (\Cref{df:algebraic_trace}).
\end{cor}
\begin{proof}
This is asserted without proof in \cite[Remark p. 187]{HartRD}.
 Both trace maps are compatible with arbitrary base change.
 Since a finite étale morphism is étale locally on the base a finite disjoint union
 of copies of the base, it suffices to use the compatibility of both trace maps
 with finite products and the case $B=A$, where both maps are the identity.
\end{proof}

\bibliographystyle{amsalpha}
\bibliography{quadratic}

\end{document}